\newcommand{\pbaddress}{biran@math.ethz.ch}
\newcommand{\mkaddress}{michael.khanevsky@math.ethz.ch}
\documentclass[12pt]{amsart}

\usepackage[colorlinks=true, urlcolor=blue,bookmarks=true,bookmarksopen=true,
citecolor=blue
]{hyperref}
\usepackage{pdfsync}
\usepackage{graphicx}

\usepackage{epsfig}
\usepackage{amscd}
\usepackage[mathscr]{eucal}
\usepackage{amssymb}
\usepackage{amsxtra}
\usepackage{amsmath}
\usepackage[all]{xy}
\usepackage{enumerate}
\usepackage{mathrsfs}

\usepackage{color}
\theoremstyle{plain}

\newtheorem{thm}{Theorem}[subsection]

\newtheorem{cor}[thm]{Corollary}

\newtheorem{lem}[thm]{Lemma}
\newtheorem{prop}[thm]{Proposition}

\theoremstyle{definition}

\theoremstyle{remark}
\newtheorem{rem}[thm]{Remark}

\theoremstyle{plain}

\newcommand{\Qed}{\hfill \qedsymbol \medskip}

\newcommand{\Id}{{{\mathchoice {\rm 1\mskip-4mu l} {\rm 1\mskip-4mu l}
      {\rm 1\mskip-4.5mu l} {\rm 1\mskip-5mu l}}}}
\newcommand{\hooklongrightarrow}{\lhook\joinrel\longrightarrow}

\newcommand{\R}{\mathbb{R}}
\newcommand{\Z}{\mathbb{Z}}

\newcommand{\C}{\mathbb{C}}

\newcommand{\Crit}{{\rm{Crit\/ }}}

\newcommand{\data}{\mathscr{D}}
\newcommand{\cplx}{\mathcal{C}}
\newcommand{\pearlspace}{\mathcal{P}}
\newcommand{\diff}{d}
\newcommand{\disk}{D}

\begin{document}

\title{A Floer-Gysin exact sequence for Lagrangian submanifolds}
\date{\today} \thanks{Both authors were partially supported by the
  ISRAEL SCIENCE FOUNDATION (grant No. 1227/06 *)}

\author{Paul Biran and Michael Khanevsky} 
\address{Paul Biran, Department of Mathematics, ETH-Z\"{u}rich,
  R\"{a}mistrasse 101, 8092 Z\"{u}rich, Switzerland}
% \address{Paul Biran, School of
%  Mathematical Sciences, Tel-Aviv University, Ramat-Aviv, Tel-Aviv
%  69978, Israel}
\email{\pbaddress} \address{Michael Khanevsky, School of Mathematical
  Sciences, Tel-Aviv University, Ramat-Aviv, Tel-Aviv 69978, Israel}
\email{\mkaddress}

\bibliographystyle{alphanum}

%----------------------------------------------------------------------
%
% Abstract
%
%\begin{abstract}
% ...
%\end{abstract}

\maketitle

%----------------------------------------------------------------------
%
% Beginning of text
%

%\tableofcontents

\section{Introduction and main results} \label{S:intro} This paper is
concerned with a Floer-theoretic analogue of the well known
Gysin-sequence from algebraic topology. In this paper we focus on the
case of circle bundles only. Recall that given a circle bundle $\pi:
\Gamma^{n+1} \longrightarrow L^n$ over a closed manifold $L$ there is
a long exact sequence in cohomology:
\begin{equation*}
   \begin{CD}
      \cdots @>>> H^{k}(L) @>{\cup\, e}>> H^{k+2}(L) @>{\pi^*}>>
      H^{k+2}(\Gamma) @>{\pi_*}>> H^{k+1}(L) @>>> \cdots
   \end{CD}
\end{equation*}
where $e \in H^2(L)$ is the Euler class of the circle bundle and
$\pi_*:H^{*+1}(\Gamma) \longrightarrow H^*(L)$ is the map that can be
identified under Poincar\'{e} duality with the map induced by the
projection $H_{n-*}(\Gamma) \longrightarrow H_{n-*}(L)$ (sometimes the
map $\pi_*$ is also called ``integration along the fibres'').

In this paper we will develop a Floer analogue of this sequence
associated to a Lagrangian submanifold $L$ and certain circle bundle
over it that appears naturally in certain geometric circumstances.

Let $(\Sigma, \omega_{\Sigma})$ be a closed symplectic manifold with
an integral symplectic structure, i.e. $[\omega_{\Sigma}] \in
H^2(\Sigma; \mathbb{R})$ admits a lift to $H^2(\Sigma; \mathbb{Z})$.
Let $L \subset \Sigma$ be a Lagrangian submanifold. One of the motives
of this paper is to study the Floer cohomology of $L$ and derive from
it possible applications, e.g. to questions concerning the topology of
$L$.

Our starting point is that one can associate to $L$ in a natural way a
flat circle bundle $\Gamma_L \to L$ whose total space $\Gamma_L$ can
be realized as a Lagrangian submanifold in a new symplectic manifold
which is a bundle over $\Sigma$. The construction is simple. Fix a
lift $a \in H^2(\Sigma;\mathbb{Z})$ of $[\omega_{\Sigma}]$ and let
$\mathcal{N} \to \Sigma$ be the complex line bundle with
$c_1^{\mathcal{N}} = a$. One can endow $\mathcal{N}$ with a hermitian
metric and a connection so that the curvature of $\mathcal{N}$ is
$\tfrac{i}{2 \pi} \omega_{\Sigma}$. The total space of $\mathcal{N}$
can be endowed with a canonical symplectic structure
$\omega_{\textnormal{can}}$ which restricts to $\omega_{\Sigma}$ on
$\Sigma$. Fix $r_0>0$ and let $P_{r_0} \subset \mathcal{N}$ be the
circle bundle of radius $r_0$ and denote by $\pi: P_{r_0} \to \Sigma$
the projection. Then $\Gamma_L = \pi^{-1}(L)$ becomes a Lagrangian
submanifold of $(\mathcal{N}, \omega_{\textnormal{can}})$. Note that
$\Gamma_L$ in fact lies in $\mathcal{N} \setminus \Sigma$.

Ideally one would like to relate the symplectic topology of $L \subset
\Sigma$ to that of $\Gamma_L \subset \mathcal{N}$ or that of $\Gamma_L
\subset \mathcal{N} \setminus \Sigma$, hoping that the latter would
shed some new light on $L$. The problem is that both $\mathcal{N}$ and
$\mathcal{N} \setminus \Sigma$ have a symplectically concave end (at
infinity) which apriori makes them inaccessible to the current
techniques of symplectic topology, in particular Floer theory.
Nevertheless, we will see that one can still define a version of Floer
cohomology for $\Gamma_L \subset \mathcal{N} \setminus \Sigma$.
Moreover, we will see that the Floer cohomology of $L$ and that of
$\Gamma_L$ are related by a long exact sequence which is analogous to
the Gysin sequence relating the singular cohomologies of $L$ and
$\Gamma_L$.

Although we can define the Floer cohomology for $\Gamma_L \subset
\mathcal{N} \setminus \Sigma$ this notion is apriori not very useful
unless we can establish some geometric properties of this cohomology,
such as invariance under Hamiltonian isotopies, a vanishing criterion
when $\Gamma_L$ is displaceable etc. This is not so clear in general
since the manifold $\mathcal{N} \setminus \Sigma$ has a concave end.
However, there is one situation in which one can go through these
difficulties: when the contact manifold $P_{r_0}$ is Weinstein
fillable. This means that $\mathcal{N} \setminus \Sigma$ (which is
just the negative symplectization of $P_{r_0}$) can be compactified at
the negative (or concave) end into a Weinstein manifold $W$. As we
will see later the Floer cohomology of $\Gamma_L$ in $\mathcal{N}
\setminus \Sigma$ coincides with that of $\Gamma_L$ in $W$. The latter
is already a completely standard object in symplectic topology and
enjoys the usual geometric properties expected from the theory.  The
fundamental example of fillable $P_{r_0}$ is when $\Sigma$ appears as
a symplectic hyperplane section in closed symplectic manifold $M$ (of
one complex dimension higher). Then $W = M \setminus \Sigma$ is
Weinstein and if one removes from it the isotropic skeleton $\Delta
\subset W$ we have $W \setminus \Delta \approx \mathcal{N} \setminus
\Sigma$. In view of this we will from now on work in this geometric
framework. Here is the setup.

Let $(M, \omega)$ be a symplectic manifold with an integral symplectic
structure, i.e. $[\omega] \in H^2(M;\mathbb{Z})$. Let $\Sigma \subset
M$ be a symplectic hyperplane section of degree $k$, so that
$PD[\Sigma] = k[\omega]$ (see~\cite{Do:hyperplane}). In this setup,
the {\em Lagrangian circle bundle}
construction~\cite{Bi:Nonintersections, Bi-Ci:closed} associates to
every Lagrangian submanifold $L \subset \Sigma$ a new Lagrangian
submanifold $\Gamma_L \subset M \setminus \Sigma$ which topologically
is a circle bundle over $L$. The construction of $\Gamma_L$ is roughly
the following (see~\S\ref{S:lag-s1-constr} and more
specifically~\S\ref{Sb:lag-s1} for the precise details). Take a
tubular neighborhood $\mathcal{U}$ of $\Sigma$ in $M$ which looks like
a disk bundle over $\Sigma$, say $\mathcal{U} \to \Sigma$. Its
boundary $P =
\partial \mathcal{U}$ is a circle bundle $\pi: P \to \Sigma$ over
$\Sigma$. Define $$\Gamma_L = \pi^{-1}(L) \subset M \setminus
\Sigma.$$ For an appropriate choice of the neighborhood $\mathcal{U}$
the resulting $\Gamma_L$ will be a Lagrangian submanifold of $M
\setminus \Sigma$. This procedure, which was introduced
in~\cite{Bi:Nonintersections, Bi-Ci:closed}, proved to be useful for
studying Lagrangians in manifolds $\Sigma$ that appear as hyperplane
sections (in some manifold $M$). The point is that the symplectic
topology of $M \setminus \Sigma$ is sometimes easier to study than
that of $\Sigma$ itself.

As $\Gamma_L \to L$ is a circle bundle the singular cohomologies of
the manifolds $\Gamma_L$ and $L$ are related by the Gysin long exact
sequence. As we will see soon, there is an analogous long exact
sequence relating their Floer cohomologies too.

Before we state our main theorem we need to introduce some notation
and elaborate more about the setting. Given a symplectic hyperplane
section $\Sigma \subset M$, put $W = M \setminus \Sigma$.  We will
assume from now on that $W$ is a Weinstein manifold. (This is often
assumed as part of the definition of ``symplectic hyperplane
section''.) The basic familiar example is when $M$ is K\"{a}hler and
$\Sigma$ is a complex hyperplane section (then $W$ is in fact affine).
As for the Lagrangian $L \subset \Sigma$ we will henceforth assume
that it is monotone with minimal Maslov number $N_L \geq 2$ (see
e.g.~\cite{Bi-Co:rigidity} for the definition). In what follows we
will mostly work with $\mathbb{Z}_2$ as the ground field both for
Floer cohomology as well as for singular cohomology. In particular
when we refer to the Euler class $e$ of the circle bundle $\Gamma_L
\to L$ we actually mean the $\mathbb{Z}_2$-reduction of the integral
Euler class, so that $e \in H^2(L;\mathbb{Z}_2)$.

We denote by $HF^*(L)$ the Floer cohomology of the pair $(L,L)$. Since
$L$ is monotone the coefficient ring will be taken to be the ring of
Laurent polynomials $\Lambda = \mathbb{Z}_2[t^{-1},t]$ where $\deg t =
N_L$ (see e.g.~\cite{Bi-Co:rigidity}). Similarly we denote by
$HF^*(\Gamma_L)$ the Floer cohomology of the pair $(\Gamma_L,
\Gamma_L)$. Note that $\Gamma_L$ can be viewed as a Lagrangian
submanifold of both $W$ and $M$. Here, by $HF^*(\Gamma_L)$ we mean the
Floer cohomology {\em in $W$} (not in $M$!). By the results
of~\cite{Bi:Nonintersections} when $\dim_{\mathbb{R}} \Sigma \geq 4$
the monotonicity of $L$ implies that $\Gamma_L \subset W$ is monotone
too and that $N_{\Gamma_L} = N_L$.  The same continues to hold if
$\dim_{\mathbb{R}} \Sigma = 2$ provided that $W$ is subcritical.

Our main result is the following.
\begin{thm} \label{T:exact-seq-1} Let $M$, $\Sigma$ and $L \subset
   \Sigma$ be as above and assume that either $\dim_{\mathbb{R}}
   \Sigma \geq 4$ or $W$ is subcritical. Then there exist canonical
   maps $$i: HF^*(L) \to HF^*(\Gamma_L), \quad p: HF^*(\Gamma_L) \to
   HF^{*-1}(L)$$ and a class $e_F \in HF^2(L)$ which all fit together
   into the following long exact sequence:
   \begin{equation*}
      \begin{CD}
         \cdots @>>> HF^{k}(L) @>{* e_F}>> HF^{k+2}(L) @>{i}>>
         HF^{k+2}(\Gamma_L) @>{p}>> HF^{k+1}(L) @>>> \cdots
      \end{CD}
   \end{equation*}
   where $* e_F$ stand for the Floer quantum product with the class
   $e_F$. Moreover, the maps $i$ and $p$ satisfy the following
   multiplicative properties with respect to the quantum products on
   $HF(L)$ and $HF(\Gamma_L)$:
   \begin{equation} \label{eq:multip-i-p} i(\alpha*\beta) =
      i(\alpha)*i(\beta), \quad p(\widetilde{\alpha}*i(\beta)) =
      p(\widetilde{\alpha})*\beta, \quad
      p(i(\alpha)*\widetilde{\beta}) = \alpha*p(\widetilde{\beta}),
   \end{equation}
   for every $\alpha, \beta \in HF^*(L)$ and $\widetilde{\alpha},
   \widetilde{\beta} \in HF^*(\Gamma_L)$.
\end{thm}
A similar theorem (in a somewhat different setting) has been
independently obtained by Perutz~\cite{Per:gysin} by completely
different methods, based on the theory of quilted Floer homology
developed by Wehrheim-Woodward~\cite{We-Wo:funct-quilted,
  We-Wo:quilted}.

The exact sequence of Theorem~\ref{T:exact-seq-1} can be regraded as a
Floer-homological analogue of the classical Gysin sequence associated
to the circle bundle $\Gamma_L \to L$. Indeed, if we replace the Floer
cohomologies by singular cohomologies in the above sequence and the
class $e_F \in HF^2(L)$ by the Euler class $e \in H^2(L;\mathbb{Z}_2)$
of $\Gamma_L \to L$ we get precisely the Gysin sequence. For this
reason we call this sequence the Floer-Gysin sequence and the class
$e_F$ the Floer-Euler class.  Moreover, we will see below that the
maps $i$ and $p$ are in fact Floer-homological analogues of the pull
back map $\pi^*$ of $\pi: \Gamma_L \to L$ and of the integration along
the fiber, respectively.

Note that since $\Sigma \subset M$ represents the Poincar\'{e} dual to
a multiple of $[\omega]$ and $L \subset \Sigma$ is Lagrangian the
bundle $\Gamma_L \to L$ is flat and so the $\mathbb{Z}$-Euler class is
torsion in $H^2(L;\mathbb{Z})$. This might look like a restrictive
situation for the choice of bundles $\Gamma_L$, however the main
object of study here is $L$ rather than $\Gamma_L$. In fact $\Gamma_L$
can be viewed as an auxiliary object for studying $L$.

In what follows we will actually establish a more general theorem
than~\ref{T:exact-seq-1} which allows to take $L$ to be a monotone
Lagrangian submanifold of $\Sigma \times Q$ for any closed symplectic
manifold $Q$. In contrast to the case $Q = \textnormal{pt}$, in this
case the circle bundle $\Gamma_L \to L$ is not necessarily flat
anymore. This generalization is described in~\S\ref{S:further}. An
application of this generalization to the topology of symplectic
hyperplane sections with subcritical complement is presented in
Theorem~\ref{t:QH-Sigma-2-priodic} in~\S\ref{S:applications-exp}.

\subsection{Applications} \label{Sb:applications} Here is an immediate
corollary of Theorem~\ref{T:exact-seq-1}.
\begin{cor} \label{C:subcrit-1} Suppose that $\Sigma$ appears as a
   symplectic hyperplane section in a symplectic manifold $M$ such
   that $W = M \setminus \Sigma$ is subcritical. Let $L \subset
   \Sigma$ be a monotone Lagrangian submanifold with $N_L \geq 2$.
   Then, either $HF(L)=0$, or the Floer-Euler class $e_F \in HF^2(L)$
   is invertible with respect to the quantum product. In particular
   $HF^*(L)$ is $2$--periodic, i.e.  for every $i \in \mathbb{Z}$
   there exists an isomorphism $HF^i(L) \cong HF^{i+2}(L)$.
\end{cor}
See~\S\ref{sb:proofs-applic} for the proof. The simplest example
satisfying this corollary is $M = {\mathbb{C}}P^{n+1}$ and $\Sigma =
{\mathbb{C}}P^n$, since we have $W = {\mathbb{C}}P^{n+1} \setminus
{\mathbb{C}}P^n \approx \textnormal{Int\,} B^{2n+2}(1)$.

Here is another corollary related to subcriticality.
\begin{cor} \label{C:subcrit-2} Let $L \subset \Sigma$ be as in
   Corollary~\ref{C:subcrit-1} but assume now that $N_L \geq 3$.
   Denote by $\mathcal{N} \to \Sigma$ the normal bundle of $\Sigma$ in
   $M$. If $HF(L) \neq 0$ then the {\em classical} Euler class $e \in
   H^2(L;\mathbb{Z}_2)$ of the restriction $\mathcal{N}|_L$ is
   non-trivial. In particular the circle bundle $\Gamma_L \to L$ is
   non-trivial and $H^2(L;\mathbb{Z})$ has torsion.
\end{cor}
The proof is given in~\S\ref{sb:proofs-applic}. An example of a
Lagrangian satisfying this corollary is $L = {\mathbb{R}}P^n \subset
{\mathbb{C}}P^n$, $n \geq 2$.

Let $\Sigma \subset {\mathbb{C}}P^{n+1}$ be a smooth quadric
hypersurface, endowed with the symplectic structure induced from
${\mathbb{C}}P^{n+1}$. As all such quadrics are symplectomorphic we
choose a specific model: $\Sigma = \{ z_0^2 + \cdots + z^n =
z_{n+1}^2\} \subset {\mathbb{C}}P^{n+1}$. Put $$L_0 = \{ [z_0: \cdots:
z_{n+1}] \in \Sigma \mid z_i \in \mathbb{R} \; \forall i\}.$$ It is
easy to see that $L_0$ is a smooth Lagrangian sphere.
\begin{cor} \label{C:quadric-1} Let $L \subset \Sigma$ be a monotone
   Lagrangian submanifold with $\dim L \geq 2$. If $HF(L) \neq 0$ then
   $L \cap L_0 \neq \emptyset$.
\end{cor}
We will prove in~\S\ref{sb:proofs-applic} a slightly stronger result.
Note that the quadric $\Sigma$ has many Lagrangians $L$ satisfying the
conditions appearing in the Corollary (see Section~1.3
in~\cite{Bi:Nonintersections} for such examples).

Since quite a few of the Corollaries above make use of the assumption
that $HF(L) \neq 0$ it is worthwhile to list some topological
conditions on $L$ that ensure this assumption.
\begin{prop}[See~\cite{Bi-Co:rigidity}]
   \label{p:conditions-for-non-narrow} Let $L \subset \Sigma$ be a
   monotone Lagrangian submanifold with minimal Maslov number $N_L$.
   Assume that $L$ satisfies one of the following conditions:
   \begin{enumerate}
     \item $N_L \geq 3$ and the cohomology ring of $L$,
      $H^*(L;\mathbb{Z}_2)$ is generated by $H^1(L;\mathbb{Z}_2)$ with
      respect to the cup product.
     \item More generally, assume that $H^*(L;\mathbb{Z}_2)$ is
      generated by $H^{< N_L -1}(L;\mathbb{Z}_2)$.
     \item $H_i(L;\mathbb{Z}_2) = 0$ for every $i \in \mathbb{Z}$ with
      $i \equiv -1 \bmod(N_L)$. (This happens for example if $L
      \approx S^n$ with $N_L \not{\mid} \,\,n+1$.)
   \end{enumerate}
   Then $HF(L) \neq 0$. In fact, we have $HF(L) \cong
   H(L;\mathbb{Z}_2) \otimes \Lambda$.
\end{prop}
Applying these conditions in each of
Corollaries~\ref{C:subcrit-1},~\ref{C:subcrit-2},~\ref{C:quadric-1}
one can obtain topological restrictions on Lagrangians appearing in
the corresponding $\Sigma$'s.

\subsection{Examples} \label{Sb:examples}

\subsubsection{Lagrangians $L \subset {\mathbb{C}}P^n$ with
  $2H_1(L;\mathbb{Z})=0$}
\label{Sbb:2H1=0}

Take $\Sigma = {\mathbb{C}}P^n$, $M = {\mathbb{C}}P^{n+1}$ and let $L
\subset {\mathbb{C}}P^n$ be a Lagrangian submanifold with $2
H_1(L;\mathbb{Z})=0$. For example, one could take here $L =
{\mathbb{R}}P^n$. It is easy to see that $L$ is monotone. By the
results of~\cite{Bi-Co:rigidity, Bi-Co:Yasha-fest} we have $N_L =
(n+1)$ and moreover:
\begin{enumerate}
  \item $H^*(L;\mathbb{Z}_2) \cong H^*(\mathbb{R}P^n;\mathbb{Z}_2)$,
   i.e. $H^i(L;\mathbb{Z}_2) = \mathbb{Z}_2$ for every $0 \leq i \leq
   n$.
  \item There exists a canonical isomorphism of $\Lambda$--modules
   $HF^*(L) \cong (H(L;\mathbb{Z}_2) \otimes \Lambda)^*$. Note
   however, that the ring structures on these modules are different.
\end{enumerate}
We will see later in~\S\ref{sb:exps-again} that the Floer-Euler class
coincides with the classical Euler class, $e_F = e$, which is the
generator of $H^2(L;\mathbb{Z}_2) = \mathbb{Z}_2$. Note that $e_F=e$
is invertible with respect to the quantum product $*$ on $HF(L)$, but
of course not with respect to the classical cup product $\cup$ on
$H^*(L;\mathbb{Z}_2)$.

\subsubsection{The Clifford torus} \label{Sbb:clif} Let $\Sigma =
{\mathbb{C}}P^n$, $M = {\mathbb{C}}P^{n+1}$ and $L =
\mathbb{T}_{\textnormal{clif}} \subset {\mathbb{C}}P^n$ the Clifford
torus given by $$L = \{ [z_0: \cdots: z_n] \in {\mathbb{C}}P^n \mid
|z_i|=1\, \forall i\}.$$ This is a monotone Lagrangian torus with
minimal Maslov number $N_L = 2$. It is well known that there exists an
isomorphism $HF(L) \cong H(L;\mathbb{Z}_2) \otimes \Lambda$
(See~\cite{Cho:Clifford}, see also~\cite{Bi-Co:rigidity}). Note that
this isomorphism is {\em not canonical} (see~\cite{Bi-Co:rigidity} for
the details), however there exists a canonical injection
$H^0(L;\mathbb{Z}_2) \otimes \Lambda \hookrightarrow HF^*(L)$ sending
the unit of $H^*(L;\mathbb{Z}_2)$ to the unit of $HF^*(L)$.

A simple computation shows that $\Gamma_L \subset {\mathbb{C}}P^{n+1}
\setminus {\mathbb{C}}P^n \approx \textnormal{Int\,} B^{2n+2}(1)$ is in
this case the split monotone torus. As we will see later on, the
Floer-Euler class in this case is $e_F = t \in HF^2(L)$. Note that the
classical Euler class $e \in H^2(L;\mathbb{Z}_2)$ of $\Gamma_L \to L$
vanishes since this bundle is trivial. Thus the classical Gysin
sequence splits into many short exact sequences: $$0 \longrightarrow
H^i(L;\mathbb{Z}_2) \longrightarrow H^i(\Gamma_L;\mathbb{Z}_2)
\longrightarrow H^{i-1}(L;\mathbb{Z}_2) \longrightarrow 0.$$ On the
other hand, since $M \setminus \Sigma = {\mathbb{C}}P^{n+1} \setminus
{\mathbb{C}}P^n$ is subcritical we have $HF(\Gamma_L) = 0$. It follows
that the Floer-Gysin sequence splits into many isomorphisms:
\begin{equation*}
   \begin{CD}
      0 @>>> HF^i(L) @>{* t}>> HF^{i+2}(L) @>>> 0.
   \end{CD}
\end{equation*}
We will work out this example and related ones in more detail
in~\S\ref{sb:exps-again}.

\subsection{Main ideas in the proof of Theorem~\ref{T:exact-seq-1}}
\label{sb:main-ideas}
Our approach to proving Theorem~\ref{T:exact-seq-1} goes via the pearl
complex and Lagrangian quantum cohomology. Recall
from~\cite{Bi-Co:Yasha-fest, Bi-Co:rigidity, Bi-Co:qrel-long} that the
self Floer cohomology $HF(L)$ is canonically isomorphic to the
Lagrangian quantum cohomology $QH(L)$. The latter is the homology of a
cochain complex which is a deformation of the Morse complex of $L$.
The underlying vector space of this complex is the same as that of the
Morse complex, however the differential on the pearl complex is
different. It counts combinations of gradient trajectories with
holomorphic disks attached to them (we call such configurations
``pearly trajectories''). The resulting cohomology has also a ring
structure coming from a quantum product. We briefly recall the
construction of this cohomology theory in~\S\ref{S:hf}. The quantum
cohomology $QH(L)$ together with its ring structure is canonically
isomorphic to $HF(L)$ via an isomorphism called the PSS. The same
holds for $QH(\Gamma_L)$ and $HF(\Gamma_L)$, hence we can replace
everywhere in Theorem~\ref{T:exact-seq-1} $HF^*$ by $QH^*$.

The long exact sequence in Theorem~\ref{T:exact-seq-1} comes in fact
from a short exact sequence of pearl complexes 
\begin{equation*}
   \begin{CD}
      0 @>>> \cplx^*(L) @>{i}>> \cplx^*(\Gamma_L) @>{p}>>
      \cplx^{*-1}(L) @>>> 0
   \end{CD}
\end{equation*}
which is described in detail in~\S\ref{S:short-exact}. Exactness of
this sequence is easy to verify, and the non-trivial part lies in
showing that $i$ and $p$ are cochain maps. This is done by comparing
the pearly trajectories on $\Gamma_L$ with those on $L$. The exactness
of the sequence follows from a correspondence between the $0$ and
$1$-dimensional moduli spaces of pearly trajectories on $L$ and on
$\Gamma_L$.

The correspondence between pearly trajectories on $L$ and on $\Gamma_L$
is done in two main steps. First note that if one removes the
Lagrangian/isotropic skeleton $\Delta$ from $W$ then we have a well
defined projection $W \setminus \Delta \to \Sigma$. Fix an almost
complex structure $J_{\Sigma}$ on $\Sigma$ and Morse data on $L$.
Given a pearly trajectory on $\Gamma_L$ we would like to project it to
$\Sigma$ and obtain a pearly trajectory on $L$. For this to work we
have to use Morse data on $\Gamma_L$ which is adapted to the Morse
data on $L$.  Moreover, in order for the holomorphic disks in the
pearly trajectories to project to holomorphic disks in $L$ we need to
work with almost complex structures $J$ on $W$ that are adapted to
$J_{\Sigma}$ in the sense that the projection is
$(J,J_{\Sigma})$--holomorphic. It is easy to find such $J$'s on $W
\setminus \Delta$ however in general they will not extend to $\Delta$.
Thus we have to allow our $J$'s to be adapted to $J_{\Sigma}$ away
from some small neighborhood $U$ of $\Delta$. We then show that for
small enough $U$, the relevant pearly trajectories on $\Gamma_L$
cannot intersect $U$, hence they all lie in the region of $W$ on which
the projection is holomorphic and so they can be safely projected to
pearly trajectories on $L$. An essential ingredient in the proof of
this fact comes from symplectic field theory (SFT), in particular we
use a neck stretching procedure for this purpose. This is all done
in~\S\ref{S:stretching}.

The second step is to show that pearly trajectories on $L$ can be
lifted to pearly trajectories on $\Gamma_L$. The lifting of the
gradient lines in a pearly trajectory can be done via standard
arguments from Morse theory. The lifting of the holomorphic disks is
done by an elementary argument from classical analysis which allows us
to lift disks with boundary on $L$ to disks in $W$ with boundary on
$\Gamma_L$. The basic construction here amounts to solving the
classical Dirichlet problem for harmonic functions on the
$2$-dimensional disk. This is done in~\S\ref{S:lifting}.

Apart from the above, one has to deal also with transversality issues
for holomorphic disks in $W$. The point is that the set of admissible
almost complex structure $J$ on $W$ is not arbitrary since we need $J$
to be adapted to $J_{\Sigma}$ and moreover have a long enough
``neck''). Thus we cannot choose $J$ to be generic in the usual sense.
Nevertheless we show that by choosing $J_{\Sigma}$ in a generic way
the set of admissible $J$'s on $W$ is large enough to obtain
transversality. This is done in~\S\ref{S:transversality}.

\medskip
\subsection{Organization of the paper}
The rest of the paper is organized as follows.
In~\S\ref{S:lag-s1-constr} we recall the precise construction of the
Lagrangian circle bundle $\Gamma_L \to L$ and recall also some
relevant facts about symplectic hyperplane sections and Weinstein
manifolds. As mentioned above we will use the Lagrangian quantum
cohomology model for Floer homology. The basic setting of this theory
is recalled in~\S\ref{S:hf}. Then in~\S\ref{S:short-exact} we describe
a short exact sequence of pearl complexes that gives rise to the long
exact sequence in cohomology that appears in
Theorem~\ref{T:exact-seq-1}. In~\S\ref{S:stretching} we explain the
stretching of the neck procedure and show how to use it in order
to assure that the relevant pearly trajectories on $\Gamma_L$ can be
indeed safely projected to $L$. The transversality issues are dealt
with in~\S\ref{S:transversality}. \S\ref{S:lifting} is dedicated to
lifting pearly trajectories from $L$ to $\Gamma_L$. Then
in~\S\ref{S:aux-data} we prove that the cohomological exact sequence
is canonical, namely that it does not depend on various choices made
in the construction (such as Morse data and almost complex
structures). In~\S\ref{S:products} we prove the multiplicative
properties of the exact sequence mentioned in
Theorem~\ref{T:exact-seq-1}. In~\S\ref{S:floer-euler} we define the
Floer-Euler class. In~\S\ref{S:positive} we show that the exact
sequence continues to hold also for the positive version of quantum
cohomology.  In~\S\ref{S:floer-euler-more} we give more information on
the Floer-Euler class and its relation to the classical Euler class.
In~\S\ref{S:seq-in-closed} we present a variant of the exact sequence
that holds when one considers $\Gamma_L$ as a Lagrangian submanifold
of $M$ (rather than $W$) and discuss its relation the sequence from
Theorem~\ref{T:exact-seq-1}. Finally, in~\S\ref{S:further} we present
some generalizations of the exact sequence that appear in other
geometric settings and discuss further potential applications.

\medskip
\subsubsection*{Acknowledgments}
We would like to thank Octav Cornea for several useful suggestions
concerning the algebraic structures in the paper as well as the idea
to use almost gradient vector fields which simplified some of our
constructions. Special thanks to Misha Sodin for his help with
Lemma~\ref{S:lift-disks}.

\section{ The Lagrangian circle bundle construction}
\label{S:lag-s1-constr}
Here we recall a construction from~\cite{Bi:Nonintersections,
  Bi-Ci:closed} which associates to a Lagrangian submanifold $L
\subset \Sigma$ a new Lagrangian $\Gamma_L \subset W$.  Before doing
that we briefly go over a few necessary notions such as Weinstein
manifolds and symplectic hyperplane sections that will be used in the
sequel.

\subsection{Weinstein manifolds}
\label{Sb:Weinstein}
A vector field $X$ on a manifold $W$ is called {\em gradient-like} for
a smooth function $\varphi:W \to \mathbb{R}$ if there exists a
positive function $\rho:W \to \mathbb{R}$ and a Riemannian metric on
$W$ such that $d \varphi (X) \geq \rho \|d \varphi \|^2$ everywhere in
$W$ (see~\cite{Gir:convex}). An open symplectic manifold $(W, \omega)$
is called {\em Weinstein} if there exists a primitive $\lambda$ of
$\omega$ such that the dual vector field $X$, defined by $i_X \omega =
\lambda$, is gradient-like with respect to a Lyapunov Morse function
$\varphi : W \to \R$. Moreover, $\varphi$ is assumed to be proper,
bounded below and have finitely many critical points.  Similarly we
have the notion of a {\em Weinstein domain}. By this we mean a compact
symplectic manifold with boundary $(W, \omega)$ such that there exist
$\lambda$ and $\varphi$ as before only that now we assume that
$\varphi: W \to [a,b]$, where $-\infty < a < b < \infty$ and that
$\partial W = \varphi^{-1}(b)$ is a regular level set of $\varphi$.

Weinstein manifolds have special topology. They have the homotopy type
of a CW-complex of dimension $\leq \frac{1}{2} \dim_{\mathbb{R}} W$.
In fact, the function $\varphi$ has the following property: for every
$x \in \textnormal{Crit}(\varphi)$ we have $\textnormal{ind}_x \varphi
\leq \frac{1}{2} \dim_{\mathbb{R}} W$ (see~\cite{El-Gr:convex,
  El:psh}). A Weinstein manifold is called subcritical if there exists
$\lambda$ and $\varphi$ such that for every $x \in
\textnormal{Crit}(\varphi)$ we have a {\em strict} inequality
$\textnormal{ind}_x \varphi < \frac{1}{2} \dim_{\mathbb{R}} W$.

The basic example of a Weinstein manifold is a Stein manifold of
finite type, namely a complex manifold $W$ which admits a proper and
bounded below smooth plurisubharmonic function $\varphi: W \to
\mathbb{R}$ without critical points outside some compact subset.
Clearly we can perturb $\varphi$ with compact support to make it Morse
and still plurisubharmonic. Take $\lambda = -d^{\mathbb{C}} \varphi$.
Since $\varphi$ is plurisubharmonic, $\omega = d \lambda$ is a
symplectic form. Each level set of $\varphi$ is pseudo-convex (away
from the critical points) hence the complex tangency distribution
$\xi$ is contact and clearly we have $\xi = \ker \lambda$ on the level
sets of $\varphi$. A simple computation shows that the contact form
that $\lambda$ induces on each level set of $\varphi$ is positive. The
simplest example of a subcritical (Wein-)Stein manifold is
$W=\mathbb{C}^n$, $\lambda = \frac{i}{2}\sum_{k=1}^n (z_k d \bar{z_k}
- \bar{z_k}dz_k)$ and $\varphi(z_1, \ldots, z_n) = \sum_{k=1}^n
|z_k|^2$.

\subsection{Standard symplectic disk bundles}
\label{Sb:std-disk-bundle} 
Let $(\Sigma, \tau)$ be an integral symplectic manifold, i.e. the de
Rham cohomology class $[\tau]$ has an integral lift in
$H^2(\Sigma;\mathbb{Z})$.  Fix a complex line bundle $\pi: \mathcal{N}
\to \Sigma$ such that $c_1(\mathcal{N})$ is a lift of $\tau$. (We
denote here the symplectic structure on $\Sigma$ by $\tau$, rather
than $\omega_{\Sigma}$, since sometimes we might want to take $\tau$
to be a multiple of $\omega_{\Sigma}$.)  Pick any hermitian metric $|
\cdot |$ on $\mathcal{N}$ and denote by $P_{1} \to \Sigma$ the
associated unit circle bundle. Choose a hermitian connection $\nabla$
on $\mathcal{N}$ with curvature $R^{\nabla} = \frac{i}{2 \pi} \tau$.
Denote $H^{\nabla}$ the horizontal distribution and by
$\alpha^{\nabla}$ the global angular 1-form on $\mathcal{N} \setminus
0$ associated to $\nabla$, i.e.
$$\alpha^{\nabla}|_{H^{\nabla}}=0, \quad \alpha^{\nabla}_{(u)}(u)=0, \quad
\alpha^{\nabla}_{(u)}(iu) = \frac{1}{2 \pi}, \quad \forall u \in
\mathcal{N} \setminus 0.$$ With these conventions we have $d
\alpha^{\nabla} = -\pi^*\tau$. Denote by $r$ the radial coordinate on
the fibres of $\mathcal{N}$ defined by $| \cdot |$.  Define a
symplectic form $\omega_{\textnormal{can}}$ on the total space of
$\mathcal{N}$ by
\begin{equation} \label{Eq:omcan} \omega_{\textnormal{can}} =
   -d(e^{-r^2} \alpha^{\nabla}) = e^{-r^2} \pi^*\tau + 2r e^{-r^2}dr
   \wedge \alpha^{\nabla}.
\end{equation}
The form $\omega_{\textnormal{can}}$ extends smoothly to the
$0$-section of $\mathcal{N}$ and is symplectic. The fibres of
$\mathcal{N}$ are symplectic and they all have area $1$ with respect
to $\omega_{\textnormal{can}}$. Next, note that $\alpha^{\nabla}$ is a
contact form on each of the circle bundles $P_r = \{ u \in \mathcal{N}
\, | \, |u|=r \}$, $r > 0$. Moreover, if we put $\alpha =
\alpha^{\nabla}|_{P_1}$ then $(\mathcal{N} \setminus 0,
\omega_{\textnormal{can}})$ can be naturally identified with the
negative symplectization of $(P_1, \alpha)$. Finally we remark that
the symplectic structure $\omega_{\textnormal{can}}$ is independent,
up to symplectomorphism, of the hermitian metric and the choice of the
connection. We will refer to $\omega_{\textnormal{can}}$ as the
canonical symplectic structure on $\mathcal{N}$ induced by
$(\Sigma,\tau)$.

Denote by $$E_r = \{ u \in \mathcal{N} \,|\, |u| \leq r \}$$ the
(closed) disk bundle of radius $r$ and by $\textnormal{Int\,} E_r = \{
u \in \mathcal{N} \,|\, |u| < r \}$ its interior. We will call $(E_r,
\omega_{\textnormal{can}})$ a {\em standard symplectic disk bundle
  over $(\Sigma, \tau)$}. (Note that the area of the fibres of $E_r$
is $1-e^{-r^2}$.)

\subsection{Symplectic hyperplane sections} \label{Sb:symp-hsection}
Let $(M^{2n+2}, \omega)$ be an integral symplectic manifold, i.e.
$[\omega] \in H^2(M; \mathbb{R})$ admits an integral lift $a \in
H^2(M; \mathbb{Z})$. Fix such a lift $a$. A symplectic hyperplane
section is a codimension-$2$ symplectic submanifold $\Sigma^{2n}
\subset M^{2n+2}$ such that:
\begin{enumerate}
  \item $[\Sigma] \in H_{2n}(M;\mathbb{Z})$ is Poincar\'{e} dual to $k
   a \in H^2(M;\mathbb{Z})$ for some $k \in \mathbb{N}$.
  \item There exists a tubular neighborhood $\mathcal{U}$ of $\Sigma$
   in $M$ whose closure is symplectomorphic to a standard symplectic
   disk bundle $(E_{\epsilon}, \frac{1}{k}\omega_{\textnormal{can}})$
   over $(\Sigma, k \omega|_{\Sigma})$.
  \item $(M \setminus \textnormal{Int\,} E_{\epsilon}, \omega)$ is a
   Weinstein domain.
\end{enumerate}
We will refer to $k$ as the degree of $\Sigma$.
From now one we will denote $\omega_{\Sigma} = \omega|_{\Sigma}$.

The basic examples of symplectic hyperplane sections come from
algebraic geometry. Let $M$ be a projective algebraic manifold and let
$\Sigma \subset M$ be a smooth ample divisor. Let $\omega$ be a
K\"{a}hler form on $M$ representing $c_1$ of the bundle
$\mathcal{O}_M(\Sigma)$. By the results of~\cite{Bi:Barriers}, $\Sigma
\subset M$ is a symplectic hyperplane section. There are also
non-algebraic examples. By a theorem of
Donaldson~\cite{Do:hyperplane}, combined with results of
Giroux~\cite{Gir:ICM2002} every integral symplectic manifold has
symplectic hyperplane sections of any large enough degree $k$.

The following proposition summarizes some relevant facts
from~\cite{Bi:Barriers}.
\begin{prop} \label{P:embedding} Let $(M, \omega)$ be an integral
   symplectic manifold and $\Sigma \subset M$ a symplectic hyperplane
   section of degree $k$. Denote by $\mathcal{N}$ the normal bundle of
   $\Sigma$ in $M$ and let $\omega_{\textnormal{can}}$ be the
   canonical symplectic form on $\mathcal{N}$ induced by $(\Sigma,
   \tau = k \omega_{\Sigma})$.  Then there exists a symplectic
   embedding $F: (\mathcal{N}, \frac{1}{k} \omega_{\textnormal{can}})
   \longrightarrow M$ with the following properties:
   \begin{enumerate}
     \item $F(x,0)=x$ for every $x \in \Sigma$. Here $(x,0) \in
      \mathcal{N}$ stands for the point in the zero section of
      $\mathcal{N}$ corresponding to $x \in \Sigma$.
     \item $\Delta = M \setminus F(\mathcal{N})$ has the structure of
      an isotropic CW-complex with respect to $\omega$.
     \item For every $r>0$, $(M \setminus F(\textnormal{Int\,} E_r),
      \omega)$ is a Weinstein domain.
     \item If the Weinstein manifold $(M \setminus \Sigma, \omega)$ is
      subcritical then $\Delta$ does not contain any Lagrangian cells,
      hence $\dim \Delta < \frac{1}{2} \dim_{\mathbb{R}}M$.
   \end{enumerate}
\end{prop}
Note that in~\cite{Bi:Barriers} these statements were proved under the
additional assumption that $(M, \omega)$ is K\"{a}hler, however they
easily extend to the non-K\"{a}hler case due to the definition of the
notion ``symplectic hyperplane section'' we gave
in~\S\ref{Sb:symp-hsection} above. The point is that our definition of
``symplectic hyperplane section'' assumes that the complement of
tubular neighborhood of $\Sigma$ is Weinstein. It is a rather
non-trivial theorem (which we will not use) that for large enough $k$
the symplectic submanifolds provided by Donaldson's
theorem~\cite{Do:hyperplane} are indeed hyperplane sections (in the
sense that their complements are Weinstein). See~\cite{Gir:ICM2002}
for more on that.

\subsection{Lagrangian circle bundles} \label{Sb:lag-s1} Let
$(M^{2n+2}, \omega)$ be an integral symplectic manifold and $\Sigma
\subset M$ a hyperplane section of degree $k$. Let $L^n \subset
\Sigma^{2n}$ be a Lagrangian submanifold. Let $\pi: \mathcal{N} \to
\Sigma$ be the normal bundle of $\Sigma$ in $M$ and
$\omega_{\textnormal{can}}$ the canonical symplectic structure induced
by $(\Sigma, \tau = k \omega_{\Sigma})$. Pick an arbitrary radius
$r_0$ and let $P_{r_0} \subset \mathcal{N}$ be the associated circle
bundle of radius $r_0$ and $\pi_{r_0}: P_{r_0} \to \Sigma$ the
projection.  Define
$$\Gamma_L = \pi_{r_0}^{-1}(L)$$ to be the restriction of this
bundle to $L$. A simple computation shows that $\Gamma_L^{n+1}$ is a
Lagrangian submanifold of $(\mathcal{N}, \omega_{\textnormal{can}})$.
Using the embedding $F: (\mathcal{N},
\frac{1}{k}\omega_{\textnormal{can}}) \longrightarrow (M, \omega)$
coming from Proposition~\ref{P:embedding} we obtain a Lagrangian
submanifold $F(\Gamma_L) \subset M \setminus \Sigma$ which in fact
lies on the boundary of the Weinstein domain $M \setminus
F(\textnormal{Int\,} E_{r_0})$.  Because of that we will identify from
now on $\Gamma_L$ with $F(\Gamma_L)$ and view $\Gamma_L$ as a
Lagrangian submanifold of $W = M \setminus \Sigma$. We call $\Gamma_L$
the {\em Lagrangian circle bundle over $L$}.

\begin{rem} \label{R:r0} Clearly $\Gamma_L$ depends on the choice of
   $r_0$. Although different choices lead to Lagrangian isotopic
   $\Gamma_L$'s, they are not Hamiltonianly isotopic. Nevertheless the
   $\Gamma_L$'s corresponding to different $r_0$'s are conformally
   symplectic equivalent in $W$. In particular, if $\Gamma_L$ is
   monotone for some $r_0$ it will continue to be so for every choice
   of $r_0$ and the minimal Maslov number is not affected by this
   choice. Moreover, the Floer homology, $HF(\Gamma_L)$, of $\Gamma_L$
   in $W$ (whenever it is well defined) does not depend on the choice
   of $r_0$. For this reason we will ignore the dependence on $r_0$,
   keeping in mind that everything we prove for $\Gamma_L \subset W$
   holds for any choice of $r_0$.  This however has one exception:
   later on in~\S\ref{S:seq-in-closed} we will also view $\Gamma_L$ as
   a Lagrangian submanifold of $(M, \omega)$. We will see that in that
   case, when $L$ is monotone, there is precisely one choice of $r_0$
   which will make $\Gamma_L \subset M$ monotone too.
\end{rem}

Using the embedding $F$ from Proposition~\ref{P:embedding} we will
often make the identification $F: \mathcal{N} \setminus \Sigma \to W
\setminus \Delta$. Translating the projection $\mathcal{N} \to \Sigma$
via this identification we obtain a projection $$\pi: (W \setminus
\Delta, \Gamma_L) \to (\Sigma, L).$$ Since $(P_{r_0}, \Gamma_L) \to
(\Sigma, L)$ is a fibration it is easy to see that
\begin{equation} \label{eq:pi_*}
   \pi_*: \pi_2(W \setminus \Delta, \Gamma_L) \to (\Sigma, L) \; \; 
   \textnormal{is an isomorphism}.
\end{equation}
Denote by $\iota: W \setminus \Delta \to W$ the inclusion. The
following proposition relates the monotonicity of $L$ to that of
$\Gamma_L$. For a Lagrangian submanifold $K$ of a symplectic manifold
$(V, \omega)$ we denote by $\mu_K:\pi_2(V,K) \to \mathbb{Z}$ the
Maslov index and by $N_K$ the minimal Maslov number
(see~\cite{Bi-Co:rigidity}).

\begin{prop}[See~\cite{Bi:Nonintersections}]
   \label{P:monotone-gl}
   Assume that either $\dim_{\mathbb{R}} \Sigma \geq 4$, or that
   $\dim_{\mathbb{R}} \Sigma = 2$ and $W = M \setminus \Sigma$ is
   subcritical. Then:
   \begin{enumerate}
     \item The homomorphism $\iota_*: \pi_2(W \setminus \Delta,
      \Gamma_L) \to \pi_2(W, \Gamma_L)$, induced by the inclusion, is
      surjective. When $\dim_{\mathbb{R}} \Sigma \geq 6$, $\iota_*$ is
      an isomorphism. The same statement holds also for homology,
      i.e.  if one replaces $\pi_2$ by $H_2$.
     \item For every $B \in \pi_2 (W \setminus \Delta, \Gamma_L)$ we
      have:
      $$\mu_{\Gamma_L}(B) = \mu_L(\pi_*(B)).$$
   \end{enumerate}
   In particular, if $L \subset \Sigma$ is monotone then $\Gamma_L
   \subset W$ is monotone too, and $N_{\Gamma_L} = N_L$.
\end{prop}

\begin{proof}
   The first statement follows easily from the fact that $\dim \Delta
   \leq \tfrac{1}{2} \dim_{\mathbb{R}} M$. The second statement is
   proved in~\cite{Bi:Nonintersections} (see Proposition 4.1.A there).
\end{proof}

Clearly given $B \in \pi_2(W,\Gamma_L)$, any class $B' \in
\pi_2(W\setminus \Delta, \Gamma_L)$ with $\iota_*(B') = B$ will have
the same Maslov index as $B$. Therefore even when $\iota_*$ is not an
isomorphism, we can always reduce the calculation of the Maslov index
in $(W,\Gamma_L)$ to $(W \setminus \Delta, \Gamma_L)$. This in turn
can be reduced to computing the Maslov index in $(\Sigma, L)$.  In
fact, as we will see later, the holomorphic disks that will be
relevant for computing the quantum cohomology of $\Gamma_L \subset W$
all lie in $W \setminus \Delta$.

\subsection{A small simplification of the setting}
\label{sb:simplification} Recall that $\Sigma \subset M$ is assumed to
be a symplectic hyperplane section in $M$, hence $PD[\Sigma] =
k[\omega]$ for some $k \in \mathbb{N}$. Rescaling the symplectic
structure $\omega$ by $k$ we may assume from now on that $PD[\Sigma] =
[\omega]$. By doing so we can assume without loss of generality that
$k=1$ and can get rid of the $k$ and $\tfrac{1}{k}$ factors that
appear in many formulas earlier in this section (e.g. in
Proposition~\ref{P:embedding}). Clearly, this will not change anything
related to the Floer cohomologies of neither $L$ nor $\Gamma_L$.

\section{Lagrangian quantum cohomology versus Floer cohomology}
\label{S:hf}

In what follows we will use the pearl complex described
in~\cite{Bi-Co:qrel-long, Bi-Co:rigidity, Bi-Co:Yasha-fest}. We refer
the reader to these papers for the precise construction of the theory.
Below we briefly recall the main definitions and setup the notation.

Let $(V, \omega)$ be a tamed symplectic manifold, $K \subset V$ a
monotone Lagrangian with minimal Maslov class $N_K \geq 2$. Since
Maslov indices come in multiples of $N_K$ we will often use the
following normalized version of the Maslov index:
$$\overline{\mu}_K = \tfrac{1}{N_K} \mu_K : 
\pi_2(V,K) \longrightarrow \mathbb{Z}.$$ We will sometimes omit the
subscript $K$ from $\mu_K$ and $\overline{\mu}_K$ when the Lagrangian
$K$ in question is obvious. Also, we will sometime prefer to work with
homology, namely $H_2(V,K)$ instead of $\pi_2(V,K)$. This will not
pose any difficulties since the Maslov index $\mu_K$ can be defined in
a compatible way also as a homomorphism $H_2(V,K) \to \mathbb{Z}$.

Put $\Lambda = \Z_2[t^{-1}, t]$ which is graded by $|t| = N_K$.  Let
$\data = (f, (\cdot, \cdot), J)$ denote a choice of auxiliary data,
where $f:K \to \R$ is a Morse function, $(\cdot, \cdot)$ is a
Riemannian metric on $L$ and $J$ an almost complex structure tamed by
$\omega$.  The pearl complex associated to $\data$ is
\[
\cplx(\data) = \Z_2 \langle \Crit f \rangle \otimes \Lambda ,
\]
where the critical points are graded by Morse index and the total
grading comes from both factors. The complex is endowed with the
differential
\[
\diff: \cplx^* (\data) \to \cplx^{*+1} (\data)
\]
whose definition we briefly recall now.  Denote by $\Phi_t : K \to K$
the negative gradient flow of $f$. Let $x, y \in \Crit f$ and denote
by $W_x^u$ and $W_y^s$ the unstable and stable submanifolds of the
critical points $x$ and $y$ respectively, with respect to negative
gradient flow of $f$. Let $\mathbf{A} = (A_1, \ldots, A_l)$ be a
vector of non-zero homology classes $A_i \in H_2 (V, K)$.

Define $\pearlspace (x, y, \mathbf{A};\data)$ to be the space of
tuples $(u_1, t_1, \ldots, u_{l-1}, t_{l-1}, u_l)$ where $t_i \in (0,
\infty)$, $u_i : (\disk, \partial \disk) \to (V, K)$ are
$J$-holomorphic disks in the class $A_i$ and we have the following
incidence relations:
\begin{equation}
   \begin{cases}
      & \Phi_{t_i} (u_i (1)) = u_{i+1} (-1)
      \quad \text{for} \;\; 1 \leq i \leq l-1 \\
      & u_1(-1) \in W^u_x \\
      & u_l(1) \in W^s_y .
   \end{cases}
\end{equation}
Moreover, in this definition the definition each of the holomorphic
disks $u_i$ is taken modulo the reparametrization subgroup of $Aut
(\disk)$ consisting of those elements that fix the points $\{1, -1\}$.
Finally, we allow $\mathbf{A}$ to consist of the zero class and define
in this case $\pearlspace (x, y, 0; \data) = \left(W^s_y \cap W^u_x
\right) / \R$. We call elements of $\mathcal{P}(x,y, \mathbf{A};
\data)$ {\em pearly trajectories}.

The space of pearly trajectories $\pearlspace (x, y,
\mathbf{A};\data)$ has virtual dimension
\begin{equation}
   \delta (x, y, \mathbf{A}) = |x| - |y| + \mu(\mathbf{A}) - 1 
\end{equation}
where $\mu(\mathbf{A}) = \sum_i \mu (A_i)$. We will also say that
trajectories $\gamma \in \pearlspace (x, y, \mathbf{A};\data)$
have {\em index} $\delta(\gamma): = \delta(x,y,\mathbf{A})$. By the
results of~\cite{Bi-Co:qrel-long}, for generic choices of $\data$ the
space of pearly trajectories has the following properties. When
$\delta = \delta (x, y, \mathbf{A}) \leq 1$, the space $\pearlspace
(x, y, \mathbf{A};\data)$ is a smooth manifold of dimension $\delta$.
Moreover, when $\delta = 0$, this manifold is compact, hence consists
of finitely many points.  Further regularity properties of these
spaces are described in~\cite{Bi-Co:qrel-long, Bi-Co:rigidity,
  Bi-Co:Yasha-fest}.

We define 
\begin{equation}
   \diff y = \sum_{x, \mathbf{A}} \# 
   \pearlspace (x, y, \mathbf{A};\data) \cdot x 
   \, t^{\overline{\mu}(\mathbf{A})},
\end{equation}
where the sum is taken over all pairs $x \in \Crit f$ and vectors
$\mathbf{A}$ (including $\mathbf{A}=0$) such that $\delta (x, y,
\mathbf{A}) = 0$. The count $\# \pearlspace (x, y, \mathbf{A};\data)$
is done in $\mathbb{Z}_2$.

It is proved in~\cite{Bi-Co:qrel-long} that $d^2 = 0$ and that the
cohomology of this complex $H^* (\cplx(\data), \diff)$ is independent
of the choices of the generic triple $\data$
(see~\cite{Bi-Co:qrel-long, Bi-Co:rigidity, Bi-Co:Yasha-fest} for more
details). This cohomology is called the {\em quantum cohomology} of
$K$ and denoted by $QH^*(K)$. (Sometime we will also call it the
``pearl cohomology of $K$''.) Note that $QH(K)$ has additional
structures such as a product $*$ which turns it into an associative
unital ring (see~\S\ref{S:products}).

\subsection{Negative almost gradient vector fields}
\label{sb:neg-almost-grad} In what follows we will sometimes use also
the following slight variation on the pearl complex construction.  Let
$f : K \to \R$ be a Morse function and $Y$ a vector field on $K$. We
call the pair $(f, Y)$ {\em negative almost gradient} if
\begin{enumerate}
  \item $(-f)$ is a Lyapunov function for $Y$, i.e. $df(Y) < 0$ away
   from the critical points of $f$.
  \item 
  	For every critical point $x \in \Crit f$ there
   	exists a neighborhood $\mathcal{U} \subset K$ of $x$ and a
   	Riemannian metric $\rho$ on $\mathcal{U}$ such that in
   	$\mathcal{U}$, $Y = - \textnormal{grad}_{\rho}f$.
\end{enumerate}
Sometimes, instead of working with triples $\data = (f, (\cdot,
\cdot), J)$ we will work with $\data = (f, Y, J)$ and replace the
negative gradient flow $\Phi_t$ in the definition of $\pearlspace$ by
the flow of the vector field $Y$, which we continue to denote
$\Phi_t$. The theory of Lagrangian quantum cohomology remains
unchanged in this setting in the sense that the resulting cohomology
is canonically isomorphic to $QH^*(K)$.

\subsection{Relation to Floer homology} \label{sb:pss} The quantum
cohomology $QH(K)$ of a monotone Lagrangian $K$ has the following
important property: it is canonically isomorphic to the self Floer
cohomology $HF(K):=HF(K,K)$ via a well-known isomorphism commonly
called PSS (see~\cite{Bi-Co:qrel-long, Bi-Co:rigidity}). Moreover,
this isomorphism identifies the quantum product on $QH(K)$ with the
corresponding product on $HF(K,K)$ defined by counting holomorphic
triangles. In view of this, from now on we will replace the Floer
cohomologies that appear in Theorem~\ref{T:exact-seq-1} by the quantum
cohomologies $QH(L)$ and $QH(\Gamma_L)$.

\section{A short exact sequence of pearly chain complexes}
\label{S:short-exact}

In this section we construct a short exact sequence of Floer cochain
complexes that gives rise the the long exact sequence of
Theorem~\ref{T:exact-seq-1}.

\subsection{Setting} \label{Sb:setting} Let $\Sigma \subset M$ be a
symplectic hyperplane section and $L \subset \Sigma$ a monotone
Lagrangian submanifold with minimal Maslov number $N_L \geq 2$. Fix
once and for all $r_0 > 0$ and put
$$P = P_{r_0} = \{ u \in \mathcal{N} \mid |u| = r_0 \}.$$
Using the symplectic embedding of Proposition~\ref{P:embedding} we can
view $P$ also as a subset of $W = M \setminus \Sigma$.  Let $\Gamma_L
\subset P$ be the Lagrangian circle bundle associated to $L \subset
\Sigma$. We denote by $\pi: \mathcal{N} \to \Sigma$, $\pi_P = \pi|_P:
P \to \Sigma$, $\pi_{\Gamma_L} = \pi|_{\Gamma_L} : \Gamma_L \to L$ the
projections.  Choose a connection $\nabla$ as
in~\S\ref{Sb:std-disk-bundle} and denote by $H_P^{\nabla} \subset
T(P)$ the horizontal distribution corresponding to it in $P$.

Let $f:L \to \mathbb{R}$ be a Morse function and $(\cdot, \cdot)$ a
Riemannian metric on $L$. Put $X = -\textnormal{grad} f$. Let
$X^{\textnormal{hor}}$ be the horizontal lift of $X$ to $\Gamma_L$
using $H_P^{\nabla}$. We will now modify $X^{\textnormal{hor}}$ into a
``negative almost gradient'' vector field on $\Gamma_L$ with respect
to some Morse function.

Denote by $x_1, \ldots, x_m$ the critical points of $f$. Choose a
small chart $\mathcal{U}_i$ around each $x_i$ and a trivialization
$\tau_i : \mathcal{U}_i \times S^1 \to \Gamma_L|_{\mathcal{U}_i}$.
Next, choose for every $i$ a Morse function $h_i: S^1 \to \mathbb{R}$
with exactly two critical points $p'_i$ and $p''_i$ of indices $0$ and
$1$ respectively. Let $Y_i = -\textnormal{grad}\, h_i$ with respect to
the standard metric on $S^1$.  Extend $Y_i$ to a vector field on
$\mathcal{U}_i \times S^1$ in a vertical way, i.e. by setting its
component in the $\mathcal{U}_i$ direction to be $0$. The resulting
field will be still denoted by $Y_i$.

Finally, for every $i$ choose a smooth cutoff function $\alpha_i: L
\to [0,1]$ with the following properties: there exist two
neighborhoods $\mathcal{V}_i \subset \mathcal{W}_i \subset
\mathcal{U}_i$ of $x_i$ with $\overline{\mathcal{V}}_i \subset
\mathcal{W}_i$, and $\overline{\mathcal{W}}_i \subset \mathcal{U}_i$
such that $\alpha_i \equiv 1$ on $\mathcal{V}_i$ and $\alpha_i \equiv
0$ outside $\mathcal{W}_i$.  Fix $\varepsilon > 0$. We define a vector
field $X_{\varepsilon}$ on $\Gamma_L$ by:
\begin{equation} \label{Eq:X-eps} X_{\varepsilon} =
   X^{\textnormal{hor}} + \varepsilon \sum_{i=1}^m (\alpha_i \circ
   \pi_{\Gamma_L}) d\tau_i(Y_i).
\end{equation}
It is easy to see that for $\varepsilon >0$ small enough this vector
field is ``negative almost gradient'' for the following Lyapunov
function on $\Gamma_L$:
$$f_{\varepsilon} = f \circ \pi_{\Gamma_L} + \varepsilon \sum_{i=1}^m
(\alpha_i \circ \pi_{\Gamma_L}) h_i \circ \tau_i^{-1}.$$

Note that outside of the neighborhoods $\mathcal{U}_i$ we have
$f_{\varepsilon} = \pi_{\Gamma_L}^* f$ and therefore all critical
points of $f_{\varepsilon}$ are contained in $\bigcup \mathcal{U}_i$.
Using the trivializations $\tau_i$ one can see that all of them lie in
fibers of critical points of $f$. Moreover, to any $x_i \in \Crit f$
there are exactly two critical points $x_i'$, $x_i''$ with
$\tau_i^{-1}(x'_i) = (x_i, p'_i)$ and $\tau_i^{-1}(x''_i) = (x_i,
p''_i)$. The indices of these critical points are given by $|x'_i| =
|x_i|$ and $|x''_i| = |x_i|+1$.

%In the definition of pearly complexes we restrict our attention to 
%regular almost complex structures $J_\Sigma$ on $\Sigma$ and \emph{admissible}
%almost complex structures $J$ on $M$ which are compatible with $J_\Sigma$. While 
%the precise definition of such $J$ will be given in section ~\ref{S:stretching}, 
%roughly speaking, this means the following. Let us identify the complement of the 
%skeleton $\Delta$ with $\mathcal{N}$ via $F: \mathcal{N} \to M$. We require from 
%$J$ that the projection $\pi: (\mathcal{N}, \, J) \to (\Sigma, J_\Sigma)$
%is J-holomorphic outside of a small neighborhood $U$ of $\Delta$. In addition, 
%$J$ should be stretched, that is, as manifodls with with an almost complex structure,
%the fibers of $\mathcal{N} \setminus U \to \Sigma$ should be isomorphic to 
%disks in $\mathbb{C}$ with a sufficently large area.

We now turn to the almost complex structures that will be used in the
pearl complexes of $L$ and $\Gamma_L$. We first choose a generic tame
almost complex structure $J_\Sigma$ on $\Sigma$.  Then, once
$J_\Sigma$ is fixed, we restrict to a class of almost complex
structures $J$ on $M$ which we call {\em admissible}. The precise
definition is given in~\S\ref{S:stretching}.  Here is a rough
description: identify the complement of the skeleton $\Delta$ with
$\mathcal{N}$ via proposition ~\ref{P:embedding}. We require that the
projection $\pi : \mathcal{N} \to \Sigma$ is $(J,
J_\Sigma)$-holomorphic outside a small neighborhood $U$ of $\Delta$.
In addition, $(\mathcal{N}, \omega, J)$ is assumed to have a long
enough ``neck'' in the sense of ``stretching of the neck'' procedure.
The precise definitions are given in~\S\ref{S:stretching}.

Put $\data = (f, \rho, J_\Sigma)$ and $\widetilde{\data}_\varepsilon =
(f_\varepsilon, X_\varepsilon, J)$. We now define maps $$i :
\cplx^*(L; \data) \to \cplx^*(\Gamma_L;
\widetilde{\data}_\varepsilon), \quad p : \cplx^*(\Gamma_L;
\widetilde{\data}_\varepsilon) \to \cplx^{*-1}(L; \data)$$ as follows.
Let $0 \leq k \leq n$, and denote by $\textnormal{Crit}_k(f)$ the set
of critical points of $f$ of index $k$. Define $i$ by:
$$i(x) = x' \;\; \forall x \in \textnormal{Crit}_k(f).$$
To define $p$ note that $\textnormal{Crit}_k(f_{\varepsilon}) =
(\textnormal{Crit}_{k}(f))' \cup (\textnormal{Crit}_{k-1}(f))''$. Define:
$$p(x') = 0 \;\; \forall x \in \textnormal{Crit}_k(f), 
\quad \textnormal{and} \;\;\; p(y'') = y, \;\; \forall y \in
\textnormal{Crit}_{k-1}(f).$$ We extend $i$ and $p$ linearly over
$\Lambda$ to the whole of $\cplx^*(L; \data)$ and $\cplx^*(\Gamma_L;
\widetilde{\data}_\varepsilon)$.

The main statement of Theorem~\ref{T:exact-seq-1} can be reformulated
as follows: let $M, \, \Sigma, \, L$ be as described above.
\begin{thm} \label{T:chain-map-1} Assume that either
   $\dim_{\mathbb{R}} \Sigma \geq 4$ or $W$ is subcritical. For a
   generic choice of auxiliary data $\data$ described above and for an
   admissible $J$ the pearl complexes $\cplx^*(L; \data)$ and
   $\cplx^*(\Gamma_L; \widetilde{\data}_\varepsilon)$ are well defined
   and their cohomologies compute the quantum cohomologies $QH(L)$ and
   $QH(\Gamma_L)$ respectively.  The maps $i$ and $p$ are cochain maps
   and they form a short exact sequence:
  \begin{equation*}
      \begin{CD}
         0 @>>> \cplx^*(L; \data) @>{i}>> \cplx^*(\Gamma_L;
         \widetilde{\data}_\varepsilon) @>{p}>> \cplx^{*-1}(L; \data)
         @>>> 0
      \end{CD}
  \end{equation*}
  of cochain complexes. In particular, we have a long exact sequence
  in cohomology:
  \begin{equation*}
      \begin{CD}
         \cdots @>>> QH^{k}(L) @>{\delta}>> QH^{k+2}(L) @>{i}>>
         QH^{k+2}(\Gamma_L) @>{p}>> QH^{k+1}(L) @>{\delta}>> \cdots
      \end{CD}
  \end{equation*}
  The cohomological long exact sequence is canonical in the sense that
  it does not depend on the auxiliary data. The connecting
  homomorphism $\delta: QH^{*}(L) \to QH^{*+2}(L)$ is given by quantum
  multiplication with a class $e_F \in QH^2(L)$. Moreover, the maps
  induced by $i$ and $p$ in cohomology (which we continue to denote by
  $i$ and $p$) are compatible with the quantum products in the
  following sense:
   \begin{equation} \label{eq:multip-i-p-2} i(\alpha*\beta) =
      i(\alpha)*i(\beta), \quad p(\widetilde{\alpha}*i(\beta)) =
      p(\widetilde{\alpha})*\beta, \quad
      p(i(\alpha)*\widetilde{\beta}) = \alpha*p(\widetilde{\beta}),
   \end{equation}
   for every $\alpha, \beta \in QH^*(L)$ and $\widetilde{\alpha},
   \widetilde{\beta} \in QH^*(\Gamma_L)$.
\end{thm}

The exactness property of the short sequence above is obvious. The
nontrivial statements are:
\begin{itemize}
  \item $i$ and $p$ are chain maps. This property will follow from the
   results presented in~\S\ref{S:stretching} and~\S\ref{S:lifting}.
   The argument is concluded in~\S\ref{S:chain-maps}.
  \item the resulting sequence in homology is canonical.  The details
   are provided in~\S\ref{S:aux-data}.
  \item the connecting homomorphism is given by quantum multiplication
   by a class $e_F \in QH^2(L)$. This will be proved
   in~\S\ref{S:floer-euler}.
  \item the maps $i$ and $p$ satisfy the multiplicative
   identities~\eqref{eq:multip-i-p-2}. This will be proved
   in~\S\ref{S:products}.
\end{itemize}

\S\ref{S:stretching} will be devoted to precise definitions of the
class of almost complex structures used, and~\S\ref{S:transversality}
for establishing the transversality results.

\section{Stretching the neck and admissible almost complex structures}
\label{S:stretching}

In all constructions which follow in this paper we will restrict
ourselves to a specific class of almost complex structures which is
described as follows.  Fix a regular almost complex structure
$J_\Sigma$ on $\Sigma$ which is tamed by $\omega_\Sigma$. 
Given $r>0$, denote by
$$E_r = \{ u \in \mathcal{N} \mid |u|\leq r \}$$ the closed disk bundle 
of radius $r$ in $\mathcal{N}$ (we use here the Hermitian metric
$|\cdot |$ chosen in~\S\ref{Sb:std-disk-bundle}).

Fix $\varepsilon > 0$. Below we will use the embedding $\mathcal{N}
\stackrel{F}{\to} M$ from Proposition~\ref{P:embedding} in order to
identify $\mathcal{N}$ as well as $E_r \subset \mathcal{N}$ with their
images in $M$. The complement in $M$ of the $(r_0+\varepsilon)$-disk
bundle $E_{r_0+\varepsilon}$ gives us a neighbourhood of the skeleton
$\Delta$.  We denote this neighbourhood by $U$.

We choose a connection $\nabla$ as in section~\ref{Sb:std-disk-bundle}
and, using the corresponding horizontal distribution $H^\nabla$, we
define an almost complex structure $J_{\mathcal{N}}$ on
$\mathcal{N}$ as follows. For $v \in H^\nabla$ put
\begin{equation} \label{Eq:J-def-1} J_{\mathcal{N}}(v) = \left(D \pi
      \big|_{H^\nabla}\right)^{-1} J_\Sigma \circ D \pi(v) .
\end{equation}
We extend $J_{\mathcal{N}}$ in the vertical direction by
multiplication by $i$ in the fibers. We define an almost complex
structure $J_M$ on $M$ by setting it to be $F_*(J_{\mathcal{N}})$ on
$M \setminus U$ (i.e. the pushforward of $J_{\mathcal{N}}$ by the
embedding $F : \mathcal{N} \to M$). We then extend $J_M$ to the rest
of $M$ in a generic way.

Denote by $M^+, M^-$ the connected components of $M \setminus P$,
where $M^-$ is the component containing the skeleton $\Delta$. For any
$R \geq 0$ set $$M^R = M^- \cup ([-R, R] \times P) \cup M^+,$$ with
the obvious gluing along the boundaries, namely $\{-R\} \times P$ is
identified with $\partial M^-$ and $\{R\} \times P$ with $\partial
M^+$. See figure~\ref{f:split}. We define an almost complex structure
$J_R$ on $M^R$ by first setting it to be equal to $J_M$ on $M^+$,
$M^-$.  We then extend this almost complex structure to $[-R, R]
\times P$ in invariant way under translations along $[-R, R]$. The
resulting almost complex structure is only continuous near $\partial
M^{\pm}$ but can be deformed near the boundary $\partial ([-R, R]
\times P)$ to a smooth almost complex structure on $M^R$ which we
denote by $J_R$.  (For this smoothing we choose a uniform deformation
which depends only on the $(t, \theta)$ coordinates on $[-R, R] \times
P$ and is independent of the projection to $\Sigma$).

Having defined $J_R$ on $M^R$ we will push it back to $M$ in the
following way. Let $\phi_R : [-R, R + \varepsilon] \to [r_0, r_0 +
\varepsilon]$ be a diffeomorphism such that
$\frac{\text{d}}{\text{d}t}\phi_R = -1$ near the boundary of $[-R, R +
\varepsilon]$. Then $\phi_R$ induces a diffeomorphism 
\begin{equation} \label{eq:lambda_R}
\lambda_R : M^R \to M,
\end{equation}
defined by identity on $U$ and $M^+$.  Note also that $\lambda_R$
preserves both the projection to $\Sigma$ and the angular coordinate
in a neighbourhood of $[-R, R] \times P$, and deforms the first
coordinate on $[-R, R] \times P$ (as well as the radial coordinate in
a neighbourhood of $[-R, R] \times P$) according to $\phi_R$. The
pushforward of $J_R$ by $\lambda_R$ defines an almost complex
structure on $M$ which we will denote by the same $J_R$ by abuse of
notation. A simple computation, based on the
description~\eqref{Eq:omcan} of $\omega$, shows that $J_R$ on $M$ 
tames $\omega$. Moreover $J_R$ has the following property: the
projection from the $(r_0+\varepsilon)$-disk bundle of $\mathcal{N}$
to $\Sigma$
\[ \pi : (E_{r_0+\varepsilon}, J_R) \to (\Sigma, J_\Sigma) \]
is holomorphic.
   
\begin{figure}[htbp]
   \begin{center}
      \epsfig{file=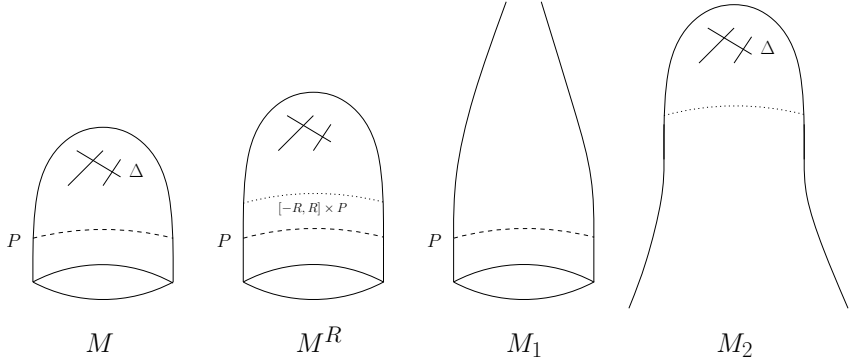, width=0.70\linewidth}
   \end{center}
   \caption{Splitting $M$ along $P$.}
   \label{f:split}
\end{figure}

For the rest of this section we will restrict our attention only to $W
= M \setminus \Sigma$. We denote by $J_W$ the restriction of the
almost complex structure $J_M$ to $W$. Put $$W^- = M^{-}, \quad W^+ =
M^{+} \setminus \Sigma, \quad W^R = M^R \setminus \Sigma.$$ We endow
these manifolds with the restrictions of the almost complex structures
we have just defined on $M^-$, $M^{+}$, $M^R$, i.e. $J_W$ and $J_R$.
The reason for defining all these structures beforehand on $M$ is that
later on in~\S\ref{S:seq-in-closed} we will use these structures to
obtain an analogous Floer-Gysin sequence for $\Gamma_L$ viewed as a
Lagrangian submanifold of $M$.

The construction above implies that replacing the almost complex
structure $J$ on $M$ (resp. $W$) by $J_R$ (with a large $R$) is
holomorphically equivalent to stretching the manifold $M$ (resp. $W$)
along $P$ in the sense of SFT~\cite{BEHWZ:compactness, EGH:SFT}. We
denote by
\[
\mathcal{J} = \mathcal{J} (J_\Sigma, U, R_0) = \left\{J_R \, \left|
      \, R > R_0 \right.\right\}
\]
the space of the stretched complex structures. 

For $J \in \mathcal{J}$ denote by
\begin{equation*}
   \pearlspace_0 (J) = \bigcup_{\delta(x,y,\mathbf{A}) = 0} 
   \pearlspace (x, y, \mathbf{A}; \widetilde{\data}_{\varepsilon})
\end{equation*}
the union of moduli spaces of pearl trajectories with zero virtual
dimension (for any critical points $x, y$) for $\Gamma_L \subset W$.

Given $r>0$, denote by $$E_r^* = E_r \setminus \Sigma \subset
\mathcal{N}$$ the punctured disk bundle or radius $r$ over $\Sigma$.
For $0<r_1 < r_2$ denote by $$E_{r_1, r_2} = E_{r_2} \setminus
\textnormal{Int\,} E_{r_1}$$ the (closed) annulus bundle over $\Sigma$
of inner radius $r_1$ and outer radius $r_2$. We call it the $(r_1,
r_2)$-annulus bundle of $\mathcal{N}$ over $\Sigma$.

The purpose of working with almost complex structures in $\mathcal{J}$
is the following:
\begin{prop}\label{p:stretching-neck}
   There exists $R_0 > 0$ such that for every $J_R$ as described above
   with $R > R_0$ the following holds: every pearly trajectory $\gamma
   \in \pearlspace_0 (J_R)$ is contained in the image $F(E_{r_0,
     r_0+\varepsilon})$ of the $(r_0, r_0+\varepsilon)$-annulus bundle
   of $\mathcal{N}$ under $F$.
\end{prop}

Before proving this proposition we derive an important corollary. From
now on we will fix the constant $R_0$ which is large enough (so that
the conclusions of Proposition~\ref{p:stretching-neck} hold) and
will work with $J_R$ where $R > R_0$. We call $\mathcal{J} =
\mathcal{J} (J_\Sigma, U, R_0)$ the space of {\em admissible} almost
complex structures. The following corollary is an immediate
consequence of Proposition~\ref{p:stretching-neck}.
\begin{cor} \label{c:proj-pearly}
   Let $\data = (f, (\cdot, \cdot), J_\Sigma)$ be auxiliary data with
   generic $J_{\Sigma}$, and $\widetilde{\data}_\varepsilon =
   (f_\varepsilon, X_\varepsilon, J)$ as in~\S\ref{S:short-exact}
   where the almost complex structure $J$ is admissible. Then any
   $\gamma \in \pearlspace_0 (J)$ projects via $\pi$ to a genuine
   pearly trajectory on $\Sigma$.
\end{cor}
Note that the index of the projection $\pi(\gamma)$ might sometimes be
$1$ rather than $0$.

\begin{rem} \label{r:stretch} As we will see in the proof of
   Proposition~\ref{p:stretching-neck} below, the conclusions of
   Proposition~\ref{p:stretching-neck} and
   Corollary~\ref{c:proj-pearly} continue to hold also for pearly
   trajectories $\gamma \in \pearlspace (x, y, \mathbf{A};
   \widetilde{\data}_{\varepsilon})$ with $\delta(x,y,\mathbf{A}) = 1$
   provided that the minimal Chern number $C_{\Sigma}$ of $\Sigma$ is
   at least $2$. Here by the minimal Chern number of $\Sigma$ we mean
   the following number: $C_{\Sigma} = \min \{ c_1^{\Sigma}(S) \mid S
   \in \pi_2(\Sigma), \, c_1^{\Sigma}(S) > 0\}$.
\end{rem}

\begin{proof}[Proof of Proposition~\ref{p:stretching-neck}]
   First of all note that by the maximum principle every non-constant
   $J$-holomorphic disk (for $J \in \mathcal{J}$) $u:(D, \partial D)
   \to (W, \Gamma_L)$ must satisfy $u(\textnormal{Int\,} D) \subset W
   \setminus E_{r_0}$. The main part of the proof is to show that for
   $R_0 \gg 0$ the following holds: for every $R \geq R_0$ all
   $J_R$-holomorphic disks $u$ that participate in index $0$ pearly
   trajectories (for $(W, \Gamma_L)$) have their images lying inside
   $E_{r_0 + \varepsilon}$.

   Below we will refer to the results of~\cite{BEHWZ:compactness}. We
   remark that the statements of that paper hold also for holomorphic
   curves with boundary on Lagrangian submanifolds.

   Put $W^+_{\infty} = (-\infty, 0] \times P \cup_\partial W^+$ and
   $W^-_{\infty} = W^- \cup_\partial [0, \infty) \times P$ each glued
   along the boundary. The almost complex structure $J_W$ on $W^+$ and
   $W^-$ is extended to the cylindrical ends by invariance under
   transtaltion in $t$ coordinate.  (One smoothens the resulting
   almost complex structures near the boundary in the fiber direction
   in a standard way).  Set $(W^{\infty}, J^{\infty})$ to be equal to
   the disjoint union $W^+_{\infty} \cup W^-_{\infty}$, each endowed
   with the preceding almost complex structures.  This way, the split
   manifold $(W^{\infty}, J^{\infty})$ can be considered as a limit of
   $(W^R, J_R)$ when $R \to \infty$, in the sense
   of~\cite{BEHWZ:compactness, EGH:SFT}. See figure~\ref{f:split}.

   Assume by contradiction, that for a generic almost complex
   structure $J_\Sigma$ on $\Sigma$ the statement of the proposition
   is not true, that is, for any $R > 0$ there exists a pearly
   trajectory $\gamma \in \mathcal{P}_0 (J_R)$ which leaves the image
   of the$(r_0+\varepsilon)$-disk bundle.  Let $R_n$ be a sequence of
   stretching parameters with $R_n \longrightarrow \infty$ and let
   $\gamma_n \in \mathcal{P}_0 (J_{R_n})$ be a sequence of pearly
   trajectories with zero index which leave the
   $(r_0+\varepsilon)$-disk bundle.  Under the holomorphic
   identification between $(W, J_R)$ and $(W^R, J_R)$, we have a
   sequence of manifolds $W^{R_n}$ together with a sequence of pearly
   trajectories in $W^{R_n}$. We will use the same notation $\gamma_n$
   for these trajectories.

   For simplicity of notation, we assume that each $\gamma_n$ contains
   a single holomorphic disk $u_n: (D, \partial D) \to (W^{R_n},
   \Gamma_L)$. (The general case is similar.) Restricting ourselves to
   a subsequence if needed, we may assume that all $u_n$ have the same
   Maslov index. We denote by $u'_n: (D, \partial D) \to (W,
   \Gamma_L)$ the disks corresponding to $u_n$ via $\lambda_R$, i.e.
   $u'_n = \lambda_r \circ u_n$.

   Using the notation of~\cite{BEHWZ:compactness}, the $\omega$-energy
   of a $J$-holomorphic curve $u$ in $W^R$ translates in our notation
   to the following:
   \[
   E_\omega (u) = \int_{u^{-1}(W^+ \cup W^-)} u^* \omega +
   \int_{u^{-1} ([-R, R] \times P)} u^*\pi_\Sigma^* \omega_\Sigma .
   \] 
   In view of monotonicity of $\Gamma_L$, the area of the disks
   ${u'}_n : D \to W$ satisfies $\int_D {u'}_n^* \omega = C$, where
   the constant $C$ is independent of $n$. A simple computation (based
   on~\eqref{Eq:omcan}) shows that:
   \[
   \int_{u^{-1} ([-R, R] \times P)} {u}^*\pi_\Sigma^* \omega_{\Sigma}
   \leq \int_{u^{-1} ([-R, R] \times P)} {u}^* \lambda_R^*\omega.
   \]
   It follows that $E_\omega(u_n) \leq C$ for every $n$. Lemma~9.2
   of~\cite{BEHWZ:compactness} implies then a uniform bound on the
   full energy $E(u_n)$ (see ~\cite{BEHWZ:compactness} for the
   definition of this energy). 

   Theorem~10.3 of~\cite{BEHWZ:compactness} describes the
   compactification of the space of $J$-holomorphic curves $\{u:D \to
   (W^R, J_R) \, | \, E(u) \leq C\}$. According to this result, there
   is a subsequence $u_{n_k}$ of $u_n$ which converges to a so-called
   holomorphic building $\overline{u}$ in $W^\infty$. This
   $\overline{u}$ is a disconnected $J_{\infty}$-holomorphic curve
   which consists of the following connected components:
   \begin{itemize}
     \item a $J$-holomorphic map $u_1: (S_1, \partial S_1) \to
      (W^+_{\infty}, \Gamma_L)$, where $S_1$ is a disk with one or
      more punctures. Near these punctures $u_1$ is asymptotically
      cylindrical and converges to a periodic orbit of the Reeb vector
      field of $(P, \alpha)$. (Here $\alpha$ is the connection
      $1$-form as chosen in~\S\ref{Sb:std-disk-bundle}.) Note that due
      to our choice of $\alpha$ the periodic orbits of the Reeb vector
      field are precisely the fibres of the circle bundle $P \to
      \Sigma$.
     \item a number of $J$-holomorphic maps, each of them looks like
      $u_2: S_2 \to W^-_{\infty}$ where
      $S_2$ is a sphere with one or more punctures. $u_2$ is
      asymptotically cylindrical near each puncture in a similar way
      to $u_1$. For simplicity we will assume that there exists one 
      such map. In the case there are many, the argument is the same.
     \item in addition, $\overline{u}$ may contain a number of
      $J$-holomorphic maps $u_i : S_i \to \R \times P$ where each
      $S_i$ is a sphere with one or more punctures each. $u_i$ are
      asymptotically cylindrical near each puncture as well.
   \end{itemize}
   Moreover, the components of $\overline{u}$ fit over the punctures,
   so they admit gluing to a topological disk.

   Coming back to our situation, there is a subsequence of
   $\{\gamma_n\}$ that converges to a pearly-like trajectory
   $\overline{\gamma}$ which has instead of a usual holomorphic disk a
   $J_{\infty}$-holomorphic building $\overline{u}$ attached. We claim
   that this implies that the virtual dimension of the corresponding
   moduli space of trajectories is positive.  This will give a
   contradiction to our initial assumption that $\gamma_n \in
   \mathcal{P}_0 (J_{R_n})$. Note that apriori, in addition to the
   above limit, one may have all possible limits of pearly
   trajectories as described in~\cite{Bi-Co:qrel-long, Bi-Co:rigidity},
   e.g. breaking of gradient trajectories, bubbling of disks or
   spheres etc.  For simplicity of notation, we assume that the
   holomorphic building $\overline{u}$ consists only of two
   components: a punctured disk $u_1 : (S_1, \partial S_1) \to
   W^+_{\infty}$ and a punctured sphere (i.e. a finite energy plane)
   $u_2: S_2 \to W^-_{\infty}$, where each component has a single
   puncture. The general case can be treated in a similar way to what
   is done below.

   By the definition of $J_{\infty}$ on $W^+_{\infty}$, the projection
   $\pi_1:W^+_{\infty} \to \Sigma$ is $(J_{\infty},
   J_{\Sigma})$-holomorphic, hence $\pi_1$ sends $u_1$ to a punctured
   disk $\pi_1 \circ u_1:(S_1, \partial S_1) \to (\Sigma, L)$. The
   periodic orbits at infinity project via $\pi_1$ to single points in
   $\Sigma$ since they are exacly the fibres of the circle bundle $P
   \to \Sigma$. Due to the asymptotic behavior of $u_1$ near the
   puncture $z$ we obtain that $\pi_1 \circ u_1$ extends continuously
   at the puncture.  Therefore $z$ is a removable singularity and
   $\pi_1 \circ u_1$ becomes a genuine $J_{\Sigma}$-holomorphic disk.

   We would like now to project $u_2:S_2 \to W^-_{\infty}$ to
   $\Sigma$. However, this cannot be done directly. Recall that on
   $W^-_{\infty}$ we have a projection defined only away from the
   skeleton, $\pi_2:W^-_{\infty} \setminus \Delta \to \Sigma$, and
   moreover this projection is not holomorphic on $U \setminus
   \Delta$. We deal with this difficulty as follows. As $\text{codim }
   \Delta > 2$, we can always perturb $u_2$ near $\Delta$ (in a
   non-holomorphic way) and obtain a new surface
   ${\widetilde{u}_2}:S_2 \to W^-_{\infty}$ with ${\widetilde{u}_2}
   (S_2) \cap \Delta = \emptyset$. Then $\pi_2 \circ
   {\widetilde{u}_2}$ gives a (not necessarily holomorphic) sphere $v
   : S^2 \to \Sigma$. (Again, the puncture goes to a point at which we
   have a removable singularity.)  We claim that $v$ has a positive
   Chern number. To see this recall that $\Sigma$ is monotone, hence
   $c_1^\Sigma = \lambda [\omega_\Sigma]$ on $\pi_2 (\Sigma)$ for some
   $\lambda > 0$.  Therefore we have:
   \begin{equation} \label{eq:c1-v} c_1^\Sigma ([v]) = \lambda
      \omega_{\Sigma} ([v]) = \lambda \int_{S_2} {\widetilde{u}_2}^*
      \pi_2^* \omega_\Sigma = \lambda
      \int_{{\widetilde{u}_2}^{-1}(W^-)} {\widetilde{u}_2}^* \pi_2^*
      \omega_{\Sigma} + \lambda
      \int_{{\widetilde{u}_2}^{-1}(W^-_{\infty} \setminus W^-)}
      {\widetilde{u}_2}^* \pi_2^* \omega_{\Sigma}.
   \end{equation}
   The 2'nd term is non-negative since $\pi_2$ is holomorphic on
   $W^-_{\infty} \setminus W^-$. As for the first term we have:
   $$\int_{{\widetilde{u}_2}^{-1}(W^-)} 
   {\widetilde{u}_2}^* \pi_2^*\omega_{\Sigma} =
   e^{(r_0+\varepsilon)^2} \int_{{\widetilde{u}_2}^{-1}(W^-)}
   {\widetilde{u}_2}^* \omega = e^{(r_0+\varepsilon)^2}
   \int_{u_2^{-1}(W^-)} u_2^* \omega > 0,$$ where the equalities 
   follow from Stokes theorem (recall that the perturbation
   ${\widetilde{u}_2}$ took place away from the boundary $u_2(\partial
   S_2)$). The last inequality holds because $u_2$ is holomorphic.
   This proves that $c_1^{\Sigma}([v])>0$.
   
   Next, replace in the ``pearly'' trajectories $\overline{\gamma}$
   the holomorphic curve $u_2$ by its perturbation $\widetilde{u}_2$.
   We continue to denote this trajectory by $\overline{\gamma}$.
   Consider now its projection $\pi \circ \overline{\gamma}$ to
   $\Sigma$. The projected trajectory is a pearly trajectory on
   $\Sigma$ whose disk $\pi_1 \circ u_1$ has a non-holomorphic sphere
   $v$ attached, and moreover $c_1^{\Sigma}([v])>0$. ($v$ cannot be
   constant because in this case $u_2$ would have zero
   $\omega$-energy.) Denote by $\gamma_{\Sigma}$ the trajectory
   obtained from $\pi \circ \overline{\gamma}$ after removing the
   sphere $v$. Note that $\gamma_{\Sigma}$ is a genuine pearly
   trajectory.

   Denote by $A \in H_2(W \setminus \Delta, \Gamma_L)$ the total
   homology class in $\overline{\gamma}$ and by $B \in H_2(\Sigma,L)$
   the total homology class in $\gamma_{\Sigma}$ after the sphere $[v]$ is removed, 
   i.e. $B = \pi_*(A) - [v]$. Let $\widetilde{x},
   \widetilde{y}$ be the starting and the ending critical points for
   $\overline{\gamma}$.  Thus $\gamma_{\Sigma}$ connects $x_\Sigma =
   \pi(\widetilde{x})$ with $y_\Sigma = \pi(\widetilde{y})$. As
   $\gamma_{\Sigma}$ is a genuine pearly trajectory and $J_{\Sigma}$
   is regular, the virtual dimension of the corresponding moduli space
   $\mathcal{P}(x_\Sigma, y_\Sigma, B; J_\Sigma)$ is non-negative:
   \begin{equation*}
      |x_\Sigma| - |y_\Sigma| + \mu_L (B) - 1 \geq 0
   \end{equation*}
   Note that $|\widetilde{y}| \geq |y_\Sigma|$ and $|\widetilde{x}|
   \leq |x_\Sigma| + 1$.  Therefore
   \begin{equation*}
      |y_\Sigma| - |x_\Sigma| \leq |\widetilde{y}| - |\widetilde{x}| 
      + 1.
   \end{equation*}
   We also have:
   \begin{equation} \label{eq:proj-u-bar}
      \mu_{\Gamma_L} ([\overline{u}]) = 
      \mu_{\Gamma_L}(A) = \mu_L(\pi_* A) = 
      \mu_L(B) + 2 c_1^{\Sigma}([v]) \geq \mu_L(B) + 2.
   \end{equation}
   All together this gives us
   \begin{equation*}
      |\widetilde{y}| - |\widetilde{x}| + 
      \mu_{\Gamma_L}([\overline{u}]) - 1 \geq |y_\Sigma| - 
      |x_\Sigma| - 1 + \mu_L(B) + 2 - 1 = 
      \left(|y_\Sigma| - |x_\Sigma| + \mu_L(B) - 1 \right) + 1 > 0,
   \end{equation*}
   which contradicts the assumption that we are in a moduli space of
   index $0$.

   The other configurations that might appear in the limit of
   $\gamma_n$ can be dealt with by a combination of the argument above
   and the compactification of spaces of pearly trajectories as
   described in~\cite{Bi-Co:qrel-long, Bi-Co:rigidity}.
\end{proof}

\begin{rem} \label{r:transversality-stretch} Note that in the proof of
   Proposition~\ref{p:stretching-neck} we have used {\em only}
   transversality for spaces of pearly trajectories on $(\Sigma,L)$,
   not for $(W,\Gamma_L)$.
\end{rem}

\section{Transversality} \label{S:transversality} The aim of this
section is to establish the needed transversality results for the
spaces of pearly trajectories involved in the quantum cohomologies of
$L$ and $\Gamma_L$ that appear in our long exact sequence. While the
general theory of pearl homology~\cite{Bi-Co:qrel-long,
  Bi-Co:rigidity} assures this transversality for generic choice of
auxiliary data, this is apriori not the case in our setting. For
example, the almost complex structures $J$ that we use on $W$ are not
arbitrary as they depend strongly on $J_{\Sigma}$, in particular they
cannot be assumed to be generic in the strict sense of the word. Still
we will see below that transversality can be still achieved by taking
$J_{\Sigma}$ to be generic.

\subsection{Regularity of $J_R$} \label{sb:reg-J_R} Holomorphic disks
in $u:(D, \partial D) \to (W, \Gamma_L)$ fall into two types. Those who
go out from $E_{r_0+\varepsilon}$ and those who remain entirely inside
$E_{r_0+\varepsilon}$. Transversality for the first type is easy to
achieve: recall that in our set of admissible $J$'s there was no
restriction on $J$ outside of $E_{r_0+\varepsilon}$. Thus we can take
$J$ to be generic on $M \setminus E_{r_0+\varepsilon}$, and the
general theory~\cite{McD-Sa:Jhol-2} assures that such $J$'s will be
regular for this type of disks.

We now turn to those disks that are entirely contained in
$E_{r_0+\varepsilon}$. In fact, as we saw in~\S\ref{S:stretching},
these are the most relevant disks, as all pearly trajectories of index
$0$ involve only disks inside $E_{r_0+\varepsilon}$.

We want to show that for a choice of a regular $J_\Sigma$ on $\Sigma$
any \emph{admissible} $J_R$ (as it is constructed
in~\S\ref{S:stretching}) satisfies regularity conditions on the disk
bundle $E_{r_0+\varepsilon}$. This would imply that the moduli space
$\mathcal{M}^*(A;J_R)$ of simple $J_R$-holomorphic disks $u:(D,
\partial D) \to (E_{r_0+\varepsilon}, \Gamma_L)$ with $u_*([D])=A$ is
a smooth finite dimensional manifold.

To prove the statement, we replace $(E_{r_0+\varepsilon}, J_R)$ by a
disk bundle $E_{A(R)} \subset (\mathcal{N}, J_\mathcal{N})$ using the identifications
defined in~\S\ref{S:stretching}, where $A(R)$ depends on $R$.
Below we will use the same notation $\Gamma_L$ for the image of
$\Gamma_L$ in $E_R$.  Recall from~\cite{McD-Sa:Jhol-2} (see~\S 3.1
there) that regularity of an almost complex structure means the
surjectivity of the linearization of the $\bar{\partial}$--operator
$D_u$ at each $J$-holomorphic disk $u : (\disk, \partial \disk) \to
(E_R, \Gamma_L)$.

Let $u: (\disk, \partial \disk) \to (E_R, \Gamma_L)$ be a holomorphic
disk.  Note, that the projection $\pi \circ u : (\disk, \partial
\disk) \to (\Sigma, L)$ is $J_\Sigma$-holomorphic.  Pick a holomorphic
trivialization $g : (\pi \circ u)^* \mathcal{N} \to \disk \times \C$.
Using this trivialization, we associate to $u$ a pair of holomorphic
maps $(u_\Sigma, u_\mathcal{N})$ where $u_\Sigma = \pi \circ u$ and
$u_\mathcal{N} : \disk \to \C$ is the projection of $g \circ u$ to the
second component. Accordingly, we have an associated pair of
linearizations of the $\bar{\partial}$--operator $(D_{u_\Sigma},
D_{u_\mathcal{N}})$.  For the surjectivity of $D_u$ it is sufficient
to show that both $D_{u_\Sigma}$ and $D_{u_\mathcal{N}}$ are
surjective.  This property holds for $D_{u_\Sigma}$ from the
regularity of $J_\Sigma$. The same is true for $D_{u_\mathcal{N}}$
since the almost complex structure in the fiber $\C$ (multiplication
by $i$) is regular.

\subsection{Transversality for pearly trajectories of index $0$}
\label{sb:trans-ind-0} Let $\data = (f, (\cdot, \cdot), J_\Sigma)$ be
a choice of Morse function, metric on $L$ and almost complex structure on
$\Sigma$. Recall from~\S\ref{Sb:setting} that in order to construct
$\widetilde{\data}_{\varepsilon}$ we need the following additional
auxiliary objects: $(\nabla, \alpha, h, J_R)$, where $\nabla$ is a
connection as chosen in~\S\ref{Sb:std-disk-bundle}, $\alpha$
represents a choice of cutoff functions near $\textnormal{Crit}(f)$
and $h$ stands for a collection of Morse functions $S^1 \to
\mathbb{R}$, as was described in~\S\ref{Sb:setting}. Here $J_R$ is an
admissible almost complex structure on $M$ which is induced from
$J_\Sigma$ and satisfies Proposition~\ref{p:stretching-neck}. We will
use the same notation $J_R$ for the induced almost complex structure
on $E_{r_0+\varepsilon}$.

Denote by $Q_{f}$ the image of the embedding:
$$
(L \setminus \textnormal{Crit}(f)) \times \mathbb{R}_{>0}
\hooklongrightarrow L \times L, \quad (x, t) \longmapsto (x,
\Phi^{f}_t(x)),
$$
where $\Phi^{f}_t$ is the negative gradient flow of $f$. Similarly,
define $Q_{\widetilde{X}_\varepsilon} \subset \Gamma_L \times
\Gamma_L$ to be the image of the embedding:
\[
(\Gamma_L \setminus \textnormal{Crit}(f_{\epsilon})) \times
\mathbb{R}_{>0} \hooklongrightarrow \Gamma_L \times \Gamma_L, \quad
(x, t) \longmapsto (x, \Phi^X_t(x))
\] 
where $\Phi^X_t$ is the flow of $\widetilde{X}_\varepsilon$.  Let
$\mathcal{M}(A, J)$ be the moduli space of holomorphic disks in the
homology class $A \in H_2(W,\Gamma_L)$. For a sequence $\bold{A} =
(A_1, \ldots, A_l)$ of non-zero classes $A_i \in H_2 (W, \Gamma_L)$
put $$\mathcal{M}(\bold{A}, J) = \mathcal{M}(A_1, J) \times \ldots
\times \mathcal{M}(A_l, J).$$ The space $\mathcal{M}(\mathbf{A}, J)$
comes with an evaluation map:
\[
ev_\bold{A} : \mathcal{M}(\bold{A}, J) \longrightarrow
\Gamma_L^{\times 2l}, \quad ev_\bold{A} (u_1, \ldots, u_l) = \left(
   u_1(-1), u_1(1), \ldots, u_l(-1), u_l(1) \right).
\]
Similarly we have the spaces $\mathcal{M}^*(A_i,J) \subset
\mathcal{M}(A_i,J)$ of simple disks and $\mathcal{M}^*(\bold{A}, J) =
\mathcal{M}^*(A_1, J) \times \ldots \times \mathcal{M}^*(A_l, J)
\subset \mathcal{M}(\mathbf{A},J)$. Note that in general
$\mathcal{M}(\mathbf{A},J)$ might not be a smooth manifold (even for
generic $J$'s). On the other hand, by what we have just seen
in~\S\ref{sb:reg-J_R} for generic admissible $J$ the spaces
$\mathcal{M}^*(\mathbf{A},J)$ are smooth manifolds.
(See~\cite{Bi-Co:qrel-long} for more details on this issue.)  Denote
by $H \subset \textnormal{Aut}(D) \cong PSL(2,\mathbb{R})$ the
subgroup of all biholomophisms $\sigma: D \to D$ which fix the two
points $-1,1 \in D$, $\sigma(\pm 1) = \pm 1$. The group $H$ acts on
$\mathcal{M}^*(A_i,J)$ by parametrization, i.e. $\sigma \cdot u = u
\circ \sigma^{-1}$. Applying this to each factor of
$\mathcal{M}^*(A_i,J)$ we obtain an action of $H^{\times l}$ on
$\mathcal{M}^*(\mathbf{A},J)$.

Let $\tilde{x}, \, \tilde{y} \in \Crit \;
(\widetilde{X}_\varepsilon)$.  Put $$\widetilde{R} = W^s_{\tilde{x}}
\times (Q_{\widetilde{X}_\varepsilon})^{\times (l-1)} \times
W^u_{\tilde{y}}.$$ With this notation we have:
$$\mathcal{P}(\tilde{x},\tilde{y}, \mathbf{A};
\widetilde{\data}_{\varepsilon}) =
ev_{\mathbf{A}}^{-1}(\widetilde{R}) / H^{\times l}.$$

\begin{prop} \label{p:transversality-pearl} Let $\data = (f, (\cdot,
   \cdot), J_{\Sigma})$ be generic data on $(\Sigma, L)$. Let $J_R$ be
   an admissible almost complex structure as in
   Proposition~\ref{p:stretching-neck}, and let $h$ be a generic
   collection of functions. Let $\tilde{x}, \, \tilde{y} \in \Crit \;
   (\widetilde{X}_\varepsilon)$, $\mathbf{A}$ with $\delta =
   \delta(\tilde{x}, \tilde{y}, \mathbf{A})\leq 0$. Then:
   \begin{enumerate}
     \item Every tuple of holomorphic disks $\mathbf{u} \in
      \mathcal{M}(\bold{A}, J_R)$ that participates in
      $\mathcal{P}(\tilde{x},\tilde{y}, \mathbf{A};
      \widetilde{\data}_{\varepsilon})$ consists of simple and
      absolutely distinct disks (see Definition~3.1.1
      in~\cite{Bi-Co:qrel-long} for the definition).
     \item The restriction of $ev_{\mathbf{A}}$ to
      $\mathcal{M}^*(\mathbf{A},J_R)$ is transverse to $\widetilde{R}$.
   \end{enumerate}
   In particular the spaces of pearly trajectories
   $\mathcal{P}(\tilde{x},\tilde{y}, \mathbf{A};
   \widetilde{\data}_{\varepsilon})$ are smooth manifolds of dimension
   $\delta$. (In particular when $\delta<0$ they are void.) Moreover,
   when $\delta=0$ these manifolds are compact, hence consist of a
   finite number of elements.
\end{prop}
Recall that by the results of~\cite{Bi-Co:qrel-long}, for a generic
choice of data $\data$, the same result as in
Proposition~\ref{p:transversality-pearl} holds for $(\Sigma,L)$
whenever the virtual dimension $\delta(x,y, \pi_*(A))$ is $\leq 1$.
The main point in Proposition~\ref{p:transversality-pearl} is that
this continues to hold for also for $(W,\Gamma_L)$ even if one uses
the (apriori non-generic) data $\widetilde{\data}_{\varepsilon}$ which
depends on $\data$. We remark however that in contrast to
$(\Sigma,L)$, for $(W,\Gamma_L)$ we have to restrict only to pearly
trajectories of index $0$. The reason is that the proof goes by
comparing the transversality of $ev_{\mathbf{A}}$ (for $(W, \Gamma_L)$
with that of $ev_{\pi_*(\mathbf{A})}$ (for $(\Sigma, L)$). If $\gamma$
is a pearly trajectory on $(W, \Gamma_L)$ of index $\delta(\gamma)$
then the index $\delta(\pi(\gamma))$ of its projection satisfies:
$\delta(\pi(\gamma)) \leq \delta(\gamma) + 1$, where equality might
occur. Thus if $\delta(\gamma)=1$ we might have
$\delta(\pi(\gamma))=2$ and transversality for index $2$ trajectory is
not known. Therefore, we restrict on $(W,\Gamma_L)$ to spaces of
virtual dimension $0$ only. However, as we will see
in~\S\ref{sb:well-def} this is enough for our purposes.

\begin{proof}[Proof of Proposition~\ref{p:transversality-pearl}]
   In view of Proposition~\ref{p:stretching-neck} we may assume that
   all disks involved in pearly trajectories corresponding to
   $\textnormal{Image\,}(ev_{\mathbf{A}}) \cap \widetilde{R}$ lie
   inside $E_{r_0+\varepsilon}$. Therefore we can project all pearly
   trajectories from $\mathcal{P}(\tilde{x},\tilde{y}, \mathbf{A};
   \widetilde{\data}_{\varepsilon})$ and obtain pearly trajectories on
   $(\Sigma, L)$. We will also view each of the classes $A_i$ in
   $\mathbf{A}$ as elements of $H_2(E_{r_0+\varepsilon}, \Gamma_L)$.
   An important point that will be used a few times in the proof below
   is that if $\gamma \in \mathcal{P}(\Gamma_L; \tilde{x}, \tilde{y},
   A; \widetilde{D}_{\varepsilon})$ has index $0$ then its projection
   $\pi(\gamma)$ to $\Sigma$ has index $\leq 1$. Therefore, if $\data$
   is generic then $\pi(\gamma)$ consists only of simple and
   absolutely distinct disks and moreover we have transversality for
   $ev_{\pi_*(\mathbf{A})}$.

   Denote by $ev^*_{\mathbf{A}}$ the restriction of $ev_{\mathbf{A}}$
   to $\mathcal{M}^*(\mathbf{A},J_R)$. Write
   $$\mathcal{M}^{*,d}(\mathbf{A},J_R) 
   \subset \mathcal{M}^*(\mathbf{A},J_R)$$ for the open subset of
   those tuples $\mathbf{u} = (u_1, \ldots, u_l)$ which consist of
   {\em absolutely distinct disks} in the sense
   of~\cite{Bi-Co:qrel-long} (see Definition 3.1.1 there).
   (Absolutely distinct means roughly speaking that no disk $u_i$ has
   its image entirely covered by the union of the rest of the disks,
   i.e.  that $u_i(D) \not \subset \cup_{j \neq i} u_j(D)$ for every
   $i$.)  Denote by $ev^{*,d}_{\mathbf{A}}$ the restriction of
   $ev_{\mathbf{A}}$ to the latter subspace. Note that by the
   discussion in~\S\ref{sb:reg-J_R} both
   $\mathcal{M}^*(\mathbf{A},J_R)$ and
   $\mathcal{M}^{*,d}(\mathbf{A},J_R)$ are smooth manifolds for a
   generic admissible $J_R$.

   The first step of the proof is to show that $ev^{*,d}_{\mathbf{A}}$
   is transverse to $\widetilde{R}$.

   Let $\bold{q} = (q_1, \ldots, q_{2l}) \in \Gamma_L^{\times 2l}$
   belong to the intersection of $\text{Image}(ev^{*,d}_\bold{A})$ and
   $\widetilde{R}$.  Pick a sequence of disks $\bold{\tilde{u}} =
   (\tilde{u}_1, \ldots, \tilde{u}_l) \in \mathcal{M}^{*,d}(\bold{A},
   J_R)$ such that $ev^{*,d}_\bold{A}(\bold{\tilde{u}}) = \mathbf{q}$.
   We denote the projections $\pi(\bold{\tilde{u}}), \pi(\tilde{x}),
   \pi(\tilde{y})$ to $\Sigma$ by $\bold{u} = (u_1, \ldots, u_l), x$
   and $y$, respectively.

   The proof goes by comparison of $ev^{*,d}_{\mathbf{A}}$ and
   $\widetilde{R}$ with their counterparts in $(\Sigma, L)$ namely
   $ev^{*,d}_{\pi_*A}$ and $R = W^s_{x} \times (Q_{-\nabla f})^{\times
     (l-1)} \times W^u_{y}$, which are assumed to be transverse (due
   to a generic choice of $\mathscr{D}$).  Note that our choice of
   auxiliary data implies that $R = \pi(\widetilde{R})$.  Similarly,
   the lifting Lemma~\ref{l:lift-a-disk} (see~\S\ref{S:lifting} below)
   together with the projection property of $J_R$ ensure that
   $ev^{*,d}_{\pi_*A} = \pi(ev^{*,d}_A)$.

   At each $q_i$ we choose a splitting $T_{q_i}\Gamma_L \simeq
   H^\nabla_{q_i} \oplus \mathbb{R}$ where $H^\nabla$ denotes the
   horizontal distribution of the connection $\nabla$ and $\mathbb{R}$
   is the tangent space of the fiber. Then
   $T_\mathbf{q}\Gamma_L^{\times 2l} \simeq \bigoplus H^\nabla_{q_i}
   \oplus \mathbb{R}^{\times 2l}$ and the restriction $D\pi:\bigoplus
   H^\nabla_{q_i} \times \{0\} \to T_{\pi (\mathbf{q})}L^{\times 2l}$
   is an isomorphism. Using the splitting
   $T_\mathbf{q}\Gamma_L^{\times 2l} \simeq \bigoplus H^\nabla_{q_i}
   \oplus \mathbb{R}^{\times 2l}$ we introduce coordinates $(v, r_1,
   \ldots, r_{2l})$ on $T_\mathbf{q}\Gamma_L^{\times 2l}$ where $v \in
   \bigoplus H^\nabla_{q_i}$ and $r_k \in \R$.

   By Lemma~\ref{l:lift-a-disk} $J_\Sigma$-holomorphic disks $u:(D^2,
   \partial D^2) \to (\Sigma, L)$ correspond to one-parametric
   families of disks $\tilde{u}:(D^2,\partial D^2) \to
   (E_{r_0+\varepsilon}, \Gamma_L)$ which are parametrized by $S^1$.
   More exactly, if $\tilde{u}$ is one such lift, then the others are
   given by rotations $\{e^{i \theta} \cdot \tilde{u}\}$ in the fibers
   of $E_{r_0+\varepsilon}$.  Therefore, $\mathcal{M}^*(\bold{A},
   J_R)$ admits an $(S^1)^{\times l}$ action $G$ which corresponds to
   independent rotation of the lifts of each disk $u_k$. This implies
   that $ev^{*,d}_\bold{A} (G \, \bold{\tilde{u}}) \subseteq
   \text{Image}(ev^{*,d}_\bold{A})$. Consequently, $V_1 = T_\mathbf{q}
   ev^{*,d}_\bold{A} (G \, \bold{\tilde{u}}) \subseteq T_\mathbf{q}
   \text{Image}(ev^{*,d}_\bold{A})$. Note, that $V_1 = \{0\} \times
   \{(r_1, r_1, r_2, r_2, \ldots, r_l, r_l)\}_{r_i \in \mathbb{R}}$ in
   the coordinates described above.  On the other hand , each
   $Q_{\widetilde{X}_\varepsilon}$ also admits a similar $S^1$-action.
   This gives rise to an $(S^1)^{\times (l-1)}$-action on
   $\widetilde{R}$ which implies that $V_2 = \{0\} \times \{(0, r_1,
   k_1 r_1, \ldots, r_{l-1}, k_{l-1} r_{l-1}, 0)\}_{r_i \in
     \mathbb{R}} \subseteq T_\mathbf{q}\widetilde{R}$ (The constants
   $k_i$ are equal to $1$ in the case when the corresponding gradient
   trajectory segment does not pass through any neighbourhood
   $\mathcal{U}$ of a critical point. In the case when it does, we
   still have $k_i \neq 0$.)

   Now we analyze the possible configurations of the critical points
   $\tilde{x}, \, \tilde{y}$. Below we will use the following
   observation: let $\pi : U_1 \to U_2$ be a surjective linear map.
   Let $V$ be a linear subspace of $U_1$.  Assume that $\ker (\pi)
   \subseteq V$. Then $V = U_1$ if and only if $\pi(V) = U_2$.
   \begin{itemize}
     \item $\tilde{x} = x'$. In this case $T_{q_1} W^s_{\tilde{x}}$
      contains the subspace $\{0\} \times \mathbb{R}$, therefore $V_3
      = \{0\} \times \{(r, 0, \ldots, 0)\}_{r \in \mathbb{R}}
      \subseteq T_\mathbf{q}\widetilde{R}$. We now have: $V_1 + V_2 +
      V_3 = \{0\} \times \mathbb{R}^{\times 2l} \subseteq
      T_\mathbf{q}\text{Image}(ev^{*,d}_\bold{A}) +
      T_\mathbf{q}\widetilde{R}$.  That is, the right-hand sum
      contains the complementary subspace $\{0\} \times
      \mathbb{R}^{\times 2l}$ which is the kernel of the projection
      $\pi : T_\mathbf{q} \Gamma_L^{2l} \to
      T_{\pi(\mathbf{q})}L^{2l}$.  The observation above implies that
      in this case the intersection is transverse if and only if the
      same is true for the respective projections
      $\text{Image}(ev^{*,d}_{\pi_{*}\bold{A}})$ and $R$. The latter
      are assumed to be transverse by a generic choice of the data
      $\data$ on $(\Sigma, L)$.
     \item $\tilde{y} = y''$. In this case $T_{q_1} W^u_{\tilde{y}}$
      contains the subspace $\{0\} \times \mathbb{R}$, therefore $V_3
      = \{0\} \times \{(0, \ldots, 0, r)\}_{r \in \mathbb{R}}
      \subseteq T_\mathbf{q}\widetilde{R}$. Once again, $V_1 + V_2 +
      V_3 = \{0\} \times \mathbb{R}^{\times 2l}$.  Using the same
      argument as before, we conclude that $ev^{*,d}_{\bold{A}}$ and
      $\widetilde{R}$ are transverse whenever their projections on
      $\Sigma$ are.
     \item The only case left to consider is $\tilde{x} = x''$ and
      $\tilde{y} = y'$.
   
      We denote $\widetilde{R}^\circ = \{e^{i \theta} \cdot
      W^s_{\tilde{x}}\}_{\theta \in [0, 2\pi]} \times
      (Q_{\widetilde{X}_\varepsilon})^{\times (l-1)} \times
      W^u_{\tilde{y}}$. Using argument similar to that in the previous
      cases, one can show that $\widetilde{R}^\circ$ intersects
      $\text{Image}(ev^{*,d}_\bold{A})$ in transverse way. Therefore
      $K = \widetilde{R}^\circ \cap \text{Image}(ev^{*,d}_\bold{A})$
      is a finite-dimensional manifold. It follows from a version of
      Sard's theorem, that for almost all values of $\theta_0$,
      $\widetilde{R}^{\theta_0} = \{e^{i \theta_0} \cdot
      W^s_{\tilde{x}}\} \times (Q_{\widetilde{X}_\varepsilon})^{\times
        (l-1)} \times W^u_{\tilde{y}}$ has transverse intersection
      with $\text{Image}(ev^{*,d}_\bold{A})$. Thus we can avoid
      non-transversality by a small perturbation of $x''$ in its fiber.
      Such perturbation corresponds to a perturbation of the
      appropriate Morse function $h_i$ as defined
      in~\S\ref{S:short-exact}.  For generic choice of functions
      $\{h_i\}$ this non-transversality phenomenon will not occur.
   \end{itemize}
   This concludes the proof that $ev^{*,d}_{\mathbf{A}}$ is transverse
   to $\widetilde{R}$.

   Next we claim that ${ev_{\mathbf{A}}}^{-1}(\widetilde{R}) =
   {ev^{*,d}_{\mathbf{A}}}^{-1}(\widetilde{R})$, namely that all
   tuples $\mathbf{u} = (u_1, \ldots, u_l)$ that participate in
   $\mathcal{P}(\tilde{x},\tilde{y}, \mathbf{A};
   \widetilde{\data}_{\varepsilon})$ consist of simple absolutely
   distinct disks. This can be done either by repeating the arguments
   from Section~3 of~\cite{Bi-Co:qrel-long} or alternatively by
   looking at the projection $\pi(\gamma)$ of $\gamma$ to $(\Sigma,
   L)$. Indeed, if the disks in $\gamma$ are either non-simple or not
   absolutely distinct then the same would hold for the disks in
   $\pi(\gamma)$ too. However, this is not the case for $\pi(\gamma)$
   since for a generic $\data$ all disks in pearly trajectories of
   index $\leq 1$ on $(\Sigma, L)$ must be simple and absolutely
   distinct (see Proposition~3.1.3 in~\cite{Bi-Co:qrel-long}).

   Finally, the fact that $\mathcal{P}(\tilde{x},\tilde{y},
   \mathbf{A}; \widetilde{\data}_{\varepsilon})$ is compact when
   $\delta(\tilde{x}, \tilde{y}, \mathbf{A})=0$ can be proved in a
   similar way as in Section~3 of~\cite{Bi-Co:qrel-long}. One analyzes
   all possible apriori limits of sequence of pearly trajectories from
   $\mathcal{P}(\tilde{x},\tilde{y}, \mathbf{A};
   \widetilde{\data}_{\varepsilon})$ and deduces that those
   configurations that do not appear in
   $\mathcal{P}(\tilde{x},\tilde{y}, \mathbf{A};
   \widetilde{\data}_{\varepsilon})$ belong to moduli spaces of
   virtual dimension $<0$. But such spaces must be void due to the
   transversality result we have just proved.
\end{proof}

\subsection{Well-definedness of the pearl complex
  $\mathcal{C}(\Gamma_L; \widetilde{\data}_{\varepsilon})$}
\label{sb:well-def}

Having established transversality for the moduli spaces
$\mathcal{P}(\Gamma_L; \tilde{x}, \tilde{y}, \mathbf{A};
\widetilde{\data}_{\varepsilon})$ whenever $\delta(\tilde{x},
\tilde{y}, \mathbf{A})=0$ we are ready to prove that
$\mathcal{C}(\Gamma_L; \widetilde{\data}_{\varepsilon})$ is well
defined and its cohomology is isomorphic to $QH(\Gamma_L)$. This is
done as follows.

First note that due to Proposition~\ref{p:transversality-pearl} the
pearly differential $\widetilde{d}$ on $\mathcal{C}(\Gamma_L;
\widetilde{\data}_{\varepsilon})$ is well defined as an operator.
(Note however that as we have not established transversality for
$1$-dimensional moduli spaces we apriori do not yet that
$\widetilde{d} \circ \widetilde{d}=0$.)

Let $\widetilde{\data}'_{\varepsilon}=(f'_{\varepsilon},
X'_{\varepsilon}, J')$ be a small and generic perturbation of the data
$\widetilde{\data}_{\varepsilon}$ where $f'_{\varepsilon} =
f_{\varepsilon}$, $(f_{\varepsilon}, X'_{\varepsilon})$ is negative almost gradient 
and $J'$ is not necessarily admissible (hence can be
taken to be really generic). Denote by $\widetilde{d}'$ the pearly
differential of $\mathcal{C}(\Gamma_L;
\widetilde{\data}'_{\varepsilon})$. By the general
theory~\cite{Bi-Co:qrel-long, Bi-Co:rigidity}, $\widetilde{d}'$ is
indeed a differential and $$H^*(\mathcal{C}(\Gamma_L;
\widetilde{\data}'_{\varepsilon}), \widetilde{d}') \cong
QH^*(\Gamma_L).$$ Clearly $\mathcal{C}^*(\Gamma_L;
\widetilde{\data}'_{\varepsilon}) = \mathcal{C}^*(\Gamma_L;
\widetilde{\data}_{\varepsilon})$ as graded vector spaces.  Finally,
the transversality result of Proposition~\ref{p:transversality-pearl}
together with standard arguments imply that $\widetilde{d} =
\widetilde{d}'$ which proves our claim.

\section{Lifting pearly trajectories} \label{S:lifting} Denote by
$\mathcal{N} \to \Sigma$ the normal bundle of $\Sigma$ in $M$, viewed
as a complex line bundle as in~\S\ref{Sb:std-disk-bundle}. We identify
$\Sigma$ with the zero section of $\mathcal{N}$. We use the connection
$\nabla$ on $\mathcal{N}$ to define an almost complex structure
$J_{\mathcal{N}}$ on the total space of $\mathcal{N}$, as was done at
the beginning of~\S\ref{S:stretching} (see~\eqref{Eq:J-def-1} there,
where the almost complex structure was denoted by $J_M$).

In this section we show that any pearly trajectory on $(\Sigma, L)$
(with respect to $\mathscr{D} = (f, (\cdot, \cdot), J_{\Sigma})$
admits a lift to $(\mathcal{N} \setminus \Sigma, \Gamma_L)$ with
respect to the corresponding data $\widetilde{\data}_\varepsilon =
(f_\varepsilon, X_\varepsilon, J_{\mathcal{N}})$. Due to compactness
properties such lifts are contained in a certain disk bundle of
$\mathcal{N}$, hence using the identification $(E_{A(R)}, J_\mathcal{N}) \to (E_{r_0 +
  \varepsilon}, J_R)$ induced by $\lambda_R$, one obtains the same result for $(W, \Gamma_L; J_R)$
(under assumption that the stretching parameter $R$ is large enough).

Moreover, the set of lifts of any non-constant trajectory is
parametrized by $S^1$. Having specified appropriate boundary
conditions, one obtains a unique lift, hence in the view of the
projection property established in Corollary~\ref{c:proj-pearly}
in~\S\ref{S:stretching} there is one-to-one correspondence between
index $0$ pearly trajectories in $(\Sigma, L)$ and those on $(W,
\Gamma_L)$. More precisely, we will see that for any $x, y \in
\text{Crit} (f)$ and $A \in \pi_2 (\Sigma, L)$ such that $|y| - |x| +
\mu(A) - 1 = 0$, we have:
\[
\# \mathcal{P} (x, y, A; \mathscr{D}) = \# \mathcal{P} (x', y',
\pi_*^{-1}(A); \widetilde{\data}_{\varepsilon}) = \# \mathcal{P} (x'',
y'', \pi_*^{-1}(A);\widetilde{\data}_{\varepsilon}).
\]
Here by $\pi_*$ we mean the homomorphism $\pi_* : \pi_2(W \setminus
\Delta,\Gamma_L) \to \pi_2(\Sigma,L)$ which is an isomorphism
(see~\eqref{eq:pi_*} before Proposition~\ref{P:monotone-gl}), hence it
makes sense to write $\pi_*^{-1}$.

\subsection{Lifting of disks} \label{S:lift-disks}
\begin{lem} \label{l:lift-a-disk} Let $u: (D^2, \partial D^2) \to
   (\Sigma, L)$ be a $J_\Sigma$-holomorphic disk. Given $\xi \in
   \partial D^2$ and $\tilde{p} \in \Gamma_L \cap \pi^{-1}(u (\xi))$
   there is a unique lift $\tilde{u}: (D^2, \partial D^2; i) \to
   (\mathcal{N} \setminus \Sigma, \Gamma_L; J_\mathcal{N})$ of $u$
   such that $\tilde{u} (\xi) = \tilde{p}$.
\end{lem}

\begin{proof}
   The pull back bundle $(u^* \mathcal{N}, u^*J_\mathcal{N}) \to D$
   admits a holomorphic trivialization as $(D^2\times \mathbb{C},
   J_0)$ where $J_0$ acts by multiplication by $i$ in both
   coordinates.  Under this trivialization $u^*
   \Gamma_L\big|_{u(\partial D^2)}$ corresponds to a circle bundle
   $\Gamma_u = \left\{ (\zeta, q) \in \partial D^2 \times \mathbb{C}
      \; \big| \; |q| = h (\zeta)\right\}$ where $h:\partial D^2 \to
   \mathbb{R}_{>0}$ is a smooth function measuring the radius of the
   unit circle in $\mathcal{N}$ (in the original hermitian metric on
   $\mathcal{N}$) with respect to the trivialization.
	
   In the trivialization above any lift of $u$ is given by $\tilde{u}
   = (u, \Psi)$ with holomorphic $\Psi : D^2 \to \mathbb{C}$ which
   satisfies the following conditions:
   \begin{itemize}
     \item $\Psi (z) \neq 0$ for all $z \in D^2$
     \item $|\Psi (\zeta)| = h (\zeta)$ for any $\zeta \in \partial
      D^2$
     \item $\Psi (\xi) = p$ where $(\xi, p)$ is the image of
      $\tilde{p}$ in our trivialization
   \end{itemize}
   
   In order to show existence of $\tilde{u}$, we take $g: D^2 \to
   \mathbb{R}$ to be the harmonic function which solves Dirichlet
   problem with boundary conditions $g (\zeta) = \log(h (\zeta))$.
   Denote by $f$ its harmonic conjugate. Then $\Psi_0 = e^{g + if}$ is
   a holomorphic function which satisfies the first two conditions.
   Its rotation $\Psi = \frac{p}{\Psi_0(\xi)} \Psi_0$ is a function
   which fulfills all the three conditions.
	
   For uniqueness we argue that if $\tilde{u}_1 = (u, \Psi_1)$,
   $\tilde{u}_2 = (u, \Psi_2)$ are two lifts, then $\varphi =
   \frac{\Psi_1}{\Psi_2}$ is a holomorphic function $D^2 \to
   \mathbb{C}$ {\em without zeros} which satisfies $|\varphi(\zeta)| =
   1$ for all $\zeta \in
   \partial D^2$. A simple application of the maximum principle shows
   that it must be constant. We note that $\varphi(\xi) = 1$,
   therefore $\varphi \equiv 1$.
\end{proof}

\subsection{Lifting of pearly trajectories} \label{S:lift-traj} Let
$\gamma \in \mathcal{P} (x, y, A; \data)$ be a pearly trajectory.
Again, to simplify the notation we assume without loss of generality
that $\gamma$ consists of a single disk $u$ and two gradient
trajectories $(\gamma_0, \gamma_1)$. Pick an arbitrary point $p \in
\textnormal{Image\,} \gamma_0$. We claim that {\sl for any $\tilde{p}
  \in \pi^{-1}(p) \cap \Gamma_L$ there is a unique lift
  $\tilde{\gamma} \in \mathcal{P} (\tilde{x}, \tilde{y},
  \pi_*^{-1}(A); \widetilde{\data}_{\varepsilon})$ of $\gamma$, which
  consists of a disk $\tilde{u}$ and $(\tilde{\gamma}_0,
  \tilde{\gamma}_1)$ such that $\tilde{p} \in \textnormal{Image\,}
  \tilde{\gamma}_0$.} (Here $\tilde{x}, \tilde{y}$ are critical points
lying in the fibers of $x, y$, and we cannot control in advance if
they will be of type $(\cdot)'$ or $(\cdot)''$.)

To prove this statement we note that there exists a unique lift
$\tilde{\gamma}_0$ of $\gamma_0$ to a trajectory along the flow of
$\widetilde{X}_\varepsilon$ which satisfies $\tilde{p} \in
\textnormal{Image\,} \tilde{\gamma}_0$.  Denote by $\tilde{\xi}$ the
endpoint of $\tilde{\gamma}_0$. Using Lemma~\ref{l:lift-a-disk} we
obtain a unique lift $\tilde{u}$ of $u$ with $\tilde{u}(-1) =
\tilde{\xi}$. Finally, there is a unique lift of $\gamma_1$ to a
gradient trajectory $\tilde{\gamma}_1$ which starts from
$\tilde{u}(1)$.

Thus all lifts of $\gamma$ are parametrized by the circle $\pi^{-1}(p)
\cap \Gamma_L$.  It is easy to see that exactly one such lift
$\gamma''$ starts from $x''$ and exactly one (we denote it by
$\gamma'$) ends at $y'$.  Assume that $\gamma$ has index $0$. Then by
dimension argument $\gamma''$ must end at $y''$. A similar argument
shows that $\gamma'$ must connect $x'$ to $y'$.

Other configurations of pearly trajectories are dealt in a similar
way: we pick a point $p$ on one of the gradient trajectory segments.
Then all lifts $\tilde{\gamma}$ of $\gamma$ are parametrized by the
lift $\tilde{p}$ of $p$.  In the case when $\gamma$ consists of a
single disk $u$ passing through critical points $x, y$, the lifts
$\tilde{\gamma}$ consist of the lift $\tilde{u}$ of $u$ together with
two gradient trajectories lying in the fibers above $x, y$. It is easy
to see that in this case too there is unique lift which connects $x'$
to $y'$, and one which connects $x''$ to $y''$.

Putting this together with Corollary~\ref{c:proj-pearly} we obtain:
\begin{equation}\label{Eq:lift-1}
   \# \mathcal{P} (x, y, A; \data) = 
   \# \mathcal{P} (x', y',
   \pi_*^{-1}(A); \widetilde{\data}_{\varepsilon}) = 
   \# \mathcal{P} (x'', y'', \pi_*^{-1}(A); \widetilde{\data}_{\varepsilon}).
\end{equation}
From dimension argument we get:
\begin{equation}\label{Eq:lift-2}
   \mathcal{P} (x'', y', \pi_*^{-1}(A); \widetilde{\data}_{\varepsilon}) 
   = \emptyset.
\end{equation}

\subsection{Chain property for $i$ and $p$} \label{S:chain-maps} We
are now finally ready to show that the maps $i$ and $p$ are chain
maps. We will denote by $d$ the differential on the pearl complex 
$\mathcal{C}(\data)$ for $(\Sigma,L)$ and by $\widetilde{d}$ the
differential of the pearl complex
$\mathcal{C}(\widetilde{\data}_{\varepsilon})$ for $(W, \Gamma_L)$
with the data $\widetilde{\data}_{\varepsilon}$ as constructed in the
previous sections.

Recall that $d : \mathcal{C}^*(\data) \to \mathcal{C}^{*+1}(\data)$ is
defined by:
\begin{equation*}
   dy = \sum_{x, A} \# 
   \pearlspace (x, y, A; \data) x 
   \, t^{\overline{\mu}(A)}.
\end{equation*}
Accordingly, for $\widetilde{d}:
\cplx^*(\widetilde{\data}_{\varepsilon}) \to
\cplx^{*+1}(\widetilde{\data}_{\varepsilon})$:
\begin{equation*}
   \widetilde{d} \, \widetilde{y}  = \sum_{\widetilde{x}, B} \# 
   \pearlspace (\widetilde{x}, \widetilde{y}, B; \widetilde{\data}_{\varepsilon})
   \widetilde{x} 
   \, t^{\overline{\mu}(B)}.
\end{equation*}
Recall also that we have an isomorphism $\pi_2(W \setminus
\Delta,\Gamma_L) \to \pi_2(\Sigma,L)$ induced by the projection $\pi:
W \setminus \Delta \to \Sigma$. Recall also that $\mu_{\Gamma_L}(B) =
\mu_L(\pi_*(B))$ for every $B \in \pi_2(W \setminus \Delta,\Gamma_L)$
(see Proposition~\ref{P:monotone-gl}). To simplify the notation, we
will write below $\mu$ for both $\mu_{\Gamma_L}$ and $\mu_L$.

From~\eqref{Eq:lift-1} and~\eqref{Eq:lift-2} we get:
\begin{equation}\label{Eq:lift-diff}
	\begin{array}{rcl}
           \widetilde{d} y' &=& \sum_{x', A} 
           \# \pearlspace (x', y', \pi^* A; \widetilde{\data}_{\varepsilon}) 
           x' \, t^{\overline{\mu}(A)} +  
           \sum_{x'', A} 
           \# \pearlspace (x'', y', \pi_*^{-1}(A); \widetilde{\data}_{\varepsilon}) 
           x'' \, t^{\overline{\mu}(A)} = \\
           &=& \sum_{x, A} \# \pearlspace (x, y, A; \data) 
           x' \, t^{\overline{\mu}(A)} +  
           \sum_{x, A} 0 \cdot x'' \, t^{\overline{\mu}(A)} ,\\
           \widetilde{d} y'' &=& \sum_{x', A} 
           \# \pearlspace (x', y'', \pi_*^{-1}(A); \widetilde{\data}_{\varepsilon}) 
           x' \, t^{\overline{\mu}(A)} +  
           \sum_{x'', A} 
           \# \pearlspace (x'', y'', \pi_*^{-1}(A); \widetilde{\data}_{\varepsilon}) 
           x'' \, t^{\overline{\mu}(A)} = \\
           &=& \sum_{x, A} 
           \# \pearlspace (x', y'', \pi_*^{-1}(A); \widetilde{\data}_{\varepsilon}) 
           x' \, t^{\overline{\mu}(A)} +  
           \sum_{x, A} \# \pearlspace (x, y, \pi_*^{-1}(A); \data) 
           x'' \, t^{\overline{\mu}(A)} .
   \end{array}
\end{equation}
These identities immediately imply that $i$ and $p$ are chain maps.
Indeed:
\begin{equation}\label{Eq:chain-maps}
   \begin{array}{rcl}
      \widetilde{d} \, i(y) &=& \widetilde{d} y' = \sum_{x, A} 
      \# \pearlspace (x, y, A; \data) 
      x' \, t^{\overline{\mu}(A)} = \\
      &=& i \left(\sum_{x, A} \# \pearlspace (x, y, A; \data) 
         x \, t^{\overline{\mu}(A)} \right) = i (d y) ,\\
      p(\widetilde{d} y') &=& p \left(\sum_{x, A} 
         \# \pearlspace (x, y, A; \data) 
         x' \, t^{\overline{\mu}(A)}\right) = 0 = 
      d(0) = d\, p (y') ,\\
      p(\widetilde{d} y'') &=& p \left(\sum_{x, A} 
         \# \pearlspace (x', y'', \pi_*^{-1}(A); \widetilde{\data}_{\varepsilon}) 
         x' \, t^{\overline{\mu}(A)}\right) + \\ 
      && + p \left(\sum_{x, A} \# \pearlspace (x, y, \pi_*^{-1}(A); \data) 
         x'' \, t^{\overline{\mu}(A)}\right) = \\
      &=& 0 + \sum_{x, A} \# \pearlspace (x, y, \pi_*^{-1}(A); \data) 
      x \, t^{\overline{\mu}(A)} = \diff y = \diff \, p (y'') .
   \end{array}
\end{equation}
\Qed

\section{Independence of auxiliary data} \label{S:aux-data} 

Let $\data^0 = (f_0, (\cdot, \cdot)_0, J_\Sigma^0)$ and $\data^1 =
(f_1, (\cdot, \cdot)_1, J_\Sigma^1)$ be two choices of auxiliary data
for the pearl complex of $L \subset \Sigma$. Denote by
$\widetilde{\data}^0_{\varepsilon}$ and
$\widetilde{\data}^1_{\varepsilon}$ corresponding choices of data for
$(W, \Gamma_L)$ as constructed in~\S\ref{S:short-exact}. Recall
from~\cite{Bi-Co:qrel-long, Bi-Co:qrel-long} that there exists a
comparison map
\[
\Phi^c_{\data^0, \data^1} : \mathcal{C}^* (\data^1) \longrightarrow
\mathcal{C}^* (\data^0)
\]
which is a chain map with respect to pearly differentials and induces
an isomorphism in cohomology $\Phi^h_{\data^0, \data^1} :
H^*(\mathcal{C}(\data^1)) \to H^*(\mathcal{C}(\data^0))$. We use here
the following convention. Maps with superscript $^c$ (e.g. $\Phi^c$)
denote chain maps, while superscript $^h$ indicates the induced map in
cohomology (e.g. $\Phi^h$ is the induced map in cohomology for
$\Phi^c$).

Note that while the maps $\Phi^c_{\data^0, \data^1}$ are not unique
they are uniquely defined up to cochain homotopy, hence the maps
$\Phi^h_{\data^0, \data^1}$ are canonical. An analogous comparison map
$\Phi^c_{\widetilde{\data}^0_{\varepsilon},
  \widetilde{\data}^1_{\varepsilon}}$ exists for the corresponding
pearl complexes of $\Gamma_L$.

The comparison maps are natural in the following sense: for any three
choices of data $\data^0, \data^1, \data^2$ we have
in cohomology:
$$\Phi^h_{\data^0, \data^1} \circ
\Phi^h_{\data^1, \data^2} = \Phi^h_{\data^0,
  \data^2}, \quad \Phi^h_{\data^0, \data^0} =
\text{Id}.$$

In this section we show that the chain maps $i$ and $p$ are compatible
with these comparison maps, hence after passage to cohomology they can
be viewed as canonical maps between the corresponding Lagrangian
quantum cohomologies.

\subsection{Construction of $\Phi^c_{\data^0, \data^1}$}
\label{sb:Phi-data}
We recall here the construction from~\cite{Bi-Co:qrel-long}.

Adding a positive constant to $f_1$, if necessary, we may assume that
$f_1(x) > f_0(x)$ for any $x \in L$.
Following~\cite{Co-Ra:Morse-Novikov}, Lemma 1.17, we pick a $C^\infty$
function $v : [0,1] \to [0, 1]$ which satisfies
\[
v(0) = 1; \; v(1) = 0; \; v'(0) = v'(1) = 0;
\]
\[
v'(t) < 0 \, (0 < t < 1); \; v''(0) < 0 < v''(1).
\]
and define $F : L \times [0, 1] \to \mathbb{R}$ by $F(x, t) =
v(t)f_0(x) + (1-v(t))f_1(x)$. We allow a small perturbation of $F$
away from the boundary of $L \times [0,1]$ in order to make the
construction generic. The function $F$ extends $f_i$ (viewed as
functions on the boundary components $L_i = L \times \{i\}$ of $L
\times [0,1]$) and has all critical points on the boundary. In fact, 
$$\textnormal{Crit}(F) = \textnormal{Crit}(f_0) \times \{0\} \cup
\textnormal{Crit}(f_1) \times \{ 1\}.$$ The indices of these critical
points satisfy:
$$|(x,0)| = |x|+1, \quad |(y,1)| = |y|.$$
Pick a Riemannian metric $(\cdot, \cdot)$ on $L \times [0, 1]$ which
restricts to $(\cdot, \cdot)_i$ on each $L_i$. As the space of almost
complex structures on $\Sigma$ is connected, we can pick a generic
path $J_\Sigma^t$, $0 \leq t \leq 1$ which connects $J_\Sigma^0$ to
$J_\Sigma^1$.

The chain map $\Phi^c_{\data^0, \data^1}: \mathcal{C}^*(\data^1)
\longrightarrow \mathcal{C}^*(\data^0)$ is defined as follows.  Let $x
\in \textnormal{Crit}(f_0)$ and $y \in \textnormal{Crit}(f_1)$ and $A
\in H_2 (\Sigma, L)$. Consider the critical points $(x,0), (y,1) \in
\textnormal{Crit}(F)$. Denote by $\widehat{\mathscr{P}}(x, y, A)$ the
moduli space of pearly-like trajectories which consist of the
following objects: an increasing sequence $0\leq t_1 < \cdots < t_l
\leq 1$, a collection of disks $u_i : (D^2,
\partial D^2) \to (\Sigma \times \{t_i\}, L \times \{t_i\})$, $i=1,
\ldots, l$, which are $J_\Sigma^{t_i}$ holomorphic ($t_i$ is fixed for
each $u_i$) and a sequence of negative gradient trajectories $\gamma_i
\subset L \times [0,1]$ of $F$ connecting consecutive disks in a
similar way we had for usual pearly trajectories.The first trajectory
should start at $(x, 0)$ and the last ends at $(y, 1)$.  Moreover $\sum [u_i] = A$.
(As was the
case with usual pearly trajectories we allow $A=0$, in which case we
do not have disks at all (i.e.  $l=0$) and the whole pearly
trajectory consists of a negative gradient trajectory of $F$.)
We refer the reader
to~\cite{Bi-Co:qrel-long} for the precise details of this
construction.

For a generic choice of the data involved, each
$\widehat{\mathscr{P}}(x,y,A)$ is a smooth manifold of dimension
$\widehat{\delta}(x,y,A) = |(x, 0)|-|(y, 1)| - 1 - \mu(A) = |y| - |x|
+ \mu(A)$.  Moreover, when $\widehat{\delta} = 0$ the space
$\widehat{\mathscr{P}}(x,y,A)$ is compact, hence consists of a finte
number of trajectories. Define
\[
\Phi^c_{\data^0, \data^1} (y) = \sum_{\widehat{\delta}(x,y,A)=0} \#
\widehat{\mathscr{P}}(x, y, A) x t^{\overline{\mu} (A)},
\]
where the sum is taken over all $x \in \textnormal{Crit}(f_0)$ and $A$
with $\widehat{\delta}(x,y,A) = 0$.  The same construction works well
if one replaces Morse functions and their negative gradient flow by a
negative almost gradient vector field as in~\S\ref{sb:neg-almost-grad}.

We will now exhibit $\Phi^c_{\widetilde{\data}^0_{\varepsilon},
  \widetilde{\data}^1_{\varepsilon}}$ as a ``lift'' of
$\Phi^c_{\data^0, \data^1}$.  We extend
$\widetilde{X}_{\varepsilon_i}^i$ to $\Gamma_L \times [0, 1]$ in the
following way: pick a connection on $\mathcal{N} \times [0, 1]$ as
in~\S\ref{Sb:std-disk-bundle} which extends $\nabla^i$ on boundaries
$\mathcal{N} \times \{ i \}$.  The cutoff functions $\alpha_k$ are
extended into a tubular neighborhood of a boundary in $L \times [0,
1]$ by $\alpha_k(x) \beta_{i_k} (t)$ where $\beta_{i_k} : [0, 1] \to
[0, 1]$ ($i_k \in \{0, 1\}$) is a smooth cutoff function which is
equal to 1 near $i_k$ and vanishes outside a $1/3$ neighborhood of
$i_k$.  Now we use the same construction as in~\S\ref{S:short-exact}:
lift the negative gradient flow of $F$ to $X^{hor}$ using the
horizontal distribution of $\nabla$ and put
\[
\widetilde{X} = X^{hor} + \varepsilon_i \sum (\alpha_k \beta_{i_k}
\circ \pi_{\Gamma_L}) \circ D\tau_k (Y_k)
\]
where $k$ indexes all critical points of $F$, $\tau_k$ are local
trivializations of $\Gamma_L \times [0,1]$ and $Y_k$ are vertical
vector fields near the critical point of $f_i$ as
in~\S\ref{S:short-exact}. We obtain a negative almost gradient vector
field which restricts to $\widetilde{X}_{\varepsilon}^i$ on the
boundary and whose projection coincides with the negative gradient
field of $F$ on $L \times [0, 1]$. The lift of $J_\Sigma^t$ is
constructed in the similar manner as in~\S\ref{S:stretching}. As the
pearl complex $\mathcal{C}^*(\Gamma_L;
\widetilde{\data}^i_{\varepsilon})$ does not change as one increases the
stretching parameter $R$ for the almost complex structure $J_R$, we
may assume that this $R$ is the same as the one used for the pearl
complexes $\widetilde{\data}^i_{\varepsilon}$. Moreover, we require
that the parameter $R$ is large enough so that all the disks which
participate in $0$-index trajectories in $\widehat{\mathscr{P}}(x, y,
A)$ (there is a finite number of them) are located in the appropriate
disk bundle which corresponds to the stretching of a lift of
$J_\Sigma^{t_i}$. Transverality is obtained in analogous way as
in~\S\ref{S:transversality}. By the results of~\cite{Bi-Co:qrel-long,
  Bi-Co:rigidity} $\Phi^c_{\widetilde{\data}^1_{\varepsilon},
  \widetilde{\data}^0_{\varepsilon}}$ are chain homotopic to the
comparison maps $\mathcal{C}^*(\Gamma_L;
\widetilde{\data}^1_{\varepsilon}) \to \mathcal{C}^*(\Gamma_L;
\widetilde{\data}^0_{\varepsilon})$ constructed by the general theory.

We now exploit the special relation between the pearly trajectories on
$\Gamma_L \times [0,1]$ and those on $L \times [0,1]$. We have a lift
$J_t$ of $J_\Sigma^t$ for which all the relevant pearly moduli spaces
$\widehat{\mathscr{P}} (\tilde{x}, \tilde{y}, A)$ project to pearly
moduli spaces on $\Sigma$. A lifting procedure, completely analogous
to the one in~\S\ref{S:lifting}, shows that when
$\widehat{\delta}(x,y,A)=0$ we have:
\[
\# \widehat{\mathscr{P}}(\Gamma_L; x', y', \pi_*^{-1}(A)) = \#
\widehat{\mathscr{P}}(\Gamma_L; x'', y'', \pi_*^{-1}(A)) = \#
\widehat{\mathscr{P}}(L; x, y, A),
\]
while $\widehat{\mathscr{P}}(\Gamma_L; x'', y', \pi_*^{-1}(A)) =
\emptyset$.  These identities show that the following diagram is
commutative on the chain level:
\[
\begin{array}{cccccccccc}
   0 & \to & \mathcal{C}^*(L; \data^1) & \stackrel{i_1}{\to} 
   & \mathcal{C}^*(\Gamma_L; \widetilde{\data}^1_{\varepsilon}) & 
   \stackrel{p_1}{\to} & & \mathcal{C}^{*-1}(L; \data^1) & 
   \to & 0 \\
   & & \Phi^c_{\data^0, \data^1} \downarrow & & 
   \Phi^c_{\widetilde{\data}^0_{\varepsilon},
  \widetilde{\data}^1_{\varepsilon}}\downarrow & & & 
   \Phi^c_{\data^0, \data^1} \downarrow & \\
   0 & \to & \mathcal{C}^*(L; \data^0) & \stackrel{i_0}{\to} & 
   \mathcal{C}^*(\Gamma_L; \widetilde{\data}^0_{\varepsilon}) & 
   \stackrel{p_0}{\to} & & \mathcal{C}^{*-1}(L; \data^0) & \to & 0
\end{array}
\]
It follows that the maps induced in cohomology by $i$ and $p$ do not
depend on the choices of the auxiliary data in the sense that they are
compatible with the comparison maps. In other words, these maps are
canonical.

\section{Product structure} \label{S:products}

\subsection{Multiplicative structure.}
Recall from~\cite{Bi-Co:qrel-long, Bi-Co:rigidity, Bi-Co:Yasha-fest}
that $QH(L)$ has a quantum product $*$ which turns it into an
associative (but not necessarily commutative) unital ring.  

The quantum product is defined in following way. Pick a Riemannian
metric $(\cdot, \cdot)$ on $L$, an almost complex structure
$J_{\Sigma}$ on $\Sigma$ and three Morse functions $f_1, f_2, f_3$ on
$L$. Put $\data_i = (f_i, (\cdot, \cdot), J_{\Sigma})$, $i=1,2,3$.
Let $x \in \textnormal{Crit}(f_1)$ and $q \in L$ a point (which is not
necessarily a critical point of $f_1$). Fix also $A_1 \in
H_2(\Sigma,L)$. Denote by $\mathscr{P}(q, x, A_1; \data_1)$ the space
of pearly trajectories going from $q$, converging to $x$ and with
total homology class $A_1$. We have similar spaces for $\data_2$ and
$\data_3$. Now let $x \in \textnormal{Crit}(f_1)$, $y \in
\textnormal{Crit}(f_2)$, $g \in \textnormal{Crit}(f_3)$, and $A \in
H_2(\Sigma,L)$. Consider the space of tuples $(\gamma_1, \gamma_2,
\gamma_3, u)$ which consist of a $J$-holomorphic disk $u:(D,
\partial D) \to (\Sigma, L)$ (which is allowed to be constant) and a
triple of pearly trajectories
$$(\gamma_1, \gamma_2, \gamma_3) \in 
\mathscr{P} \bigl(u(e^{2 \pi i/3}),x, A_1; \data_1\bigr) \times
\mathscr{P}\bigl(u(e^{-2 \pi i/3}),y, A_2; \data_2\bigr) \times
\mathscr{P}\bigl(z, u (1), A_3; \data_3\bigr),$$ where $A = [u] + A_1
+ A_2 + A_3 \in H_2 (\Sigma, L)$. We denote the space of such tuples
$(\gamma_1, \gamma_2, \gamma_3, u)$ by
$\mathscr{P}_{\textnormal{prod}}(z, x, y, A)$.

The virtual dimension of $\mathscr{P}_{\textnormal{prod}}(z, x, y, A)$
is given by $\delta = |z| - |x| - |y| + \mu(A)$. If $\delta \leq 1$
then for a generic choice of data $(f_1, f_2, f_3, (\cdot, \cdot),
J_{\Sigma})$, the space $\mathscr{P} (z, x, y, A)$ is a smooth
manifold of dimension $\delta$. Moreover, when $\delta = 0$ the moduli
space consists of a finite number of elements
(see~\cite{Bi-Co:qrel-long, Bi-Co:rigidity}). Define now a chain level
operation
$$\mathcal{C}(\data_1) \otimes
\mathcal{C}(\data_2) \longrightarrow \mathcal{C}(\data_3), \quad x
\otimes y \longmapsto x*y,$$ by:
\[
x * y = \sum \# \mathscr{P}_{\textnormal{prod}}(z, x, y, A) z
t^{\overline{\mu}(A)},
\]
where the summation goes over $z,A$ with $\delta(z,x,y,A) = 0$. This
operation descends to an associative unital product on $QH^*(L)$.

The same construction works of course for $\Gamma_L \subset W$ too.
We will now implement it on $\Gamma_L$ using auxiliary data induced
from that of $L$ so that it is adapted to our situation. We would like
to lift the pearly configurations from
$\mathscr{P}_{\textnormal{prod}}(z, x, y, A)$ to $(W, \Gamma_L)$ in a
similar way to what we have done for the `usual' pearly trajectories.

Consider three lifts of $-\textnormal{grad} f_i$, $i=1,2,3$, to
negative almost gradient vector fields $\widetilde{X}^\varepsilon_1,
\widetilde{X}^\varepsilon_2, \widetilde{X}^\varepsilon_3$ on
$\Gamma_L$ as described at the end of~\S\ref{S:hf}. Consider also an
admissible almost complex structure $J_R$ on $M$ induced by
$J_{\Sigma}$ as in~\S\ref{S:stretching} with stretching parameter $R$
large enough.

For a generic choice of parameters the spaces
$\mathscr{P}_{\textnormal{prod}}(\tilde{z}, \tilde{x}, \tilde{y}, A;
\Gamma_L)$ enjoy similar transversality properties as
in~\S\ref{S:transversality} and one may use them to define a chain
level product which descends to the quantum product on
$QH^*(\Gamma_L)$.

The projection property for
$\mathscr{P}_{\textnormal{prod}}(\tilde{z}, \tilde{x}, \tilde{y},
A;\Gamma_L)$ follows from the construction by similar arguments as
in~\S\ref{S:stretching}. Moreover, arguing in a similar manner as
in~\S\ref{S:lifting} we establish the following identities.  For every
critical points $x,y,z$ of $f_1, f_2, f_3$ respectively and $A \in
\pi_2(\Sigma,L)$ with $|z|-|x|-|y|+\mu(A) = 0$:
\[
\begin{array}{rcl}
   \# \mathscr{P}_{\textnormal{prod}}(z, x, y, A; L) &=& 
   \# \mathscr{P}_{\textnormal{prod}}(z', x', y', \pi_*^{-1}(A); 
   \Gamma_L) \\
   \# \mathscr{P}_{\textnormal{prod}}(z, x, y, A; L) &=& 
   \# \mathscr{P}_{\textnormal{prod}}(z'', x'', y', \pi_*^{-1}(A);\Gamma_L) 
   = 
   \# \mathscr{P}_{\textnormal{prod}}(z'', x', y'', \pi_*^{-1}(A); \Gamma_L).
\end{array}
\]
Moreover, $\mathscr{P}_{\textnormal{prod}}(z'', x', y', 
\pi_*^{-1}(A);\Gamma_L)$ does not have any zero-dimensional
components. All together this implies that for every $x \in
\textnormal{Crit}(f_1)$, $y \in \textnormal{Crit}(f_2)$ and $\tilde{x}
\in \textnormal{Crit}(\widetilde{X}^\varepsilon_1)$, $\tilde{y} \in
\textnormal{Crit}(\widetilde{X}^\varepsilon_2)$ we have:
\begin{equation} \label{eq:i-p-prod}
   i(x * y) = i(x) * i(y), \quad p(\tilde{x} * i(y)) = p(\tilde{x})
   * y, \quad p(i(x) * \tilde{y}) = x * p(\tilde{y}).
\end{equation}
Note that these identities hold on the chain level.

\section{The Floer-Euler class} \label{S:floer-euler} Denote by
$\delta: QH^k(L) \longrightarrow QH^{k+2}(L)$ the connecting
homomorphism in the long exact sequence of
Theorem~\ref{T:chain-map-1}. Denote by $1 \in QH^0(L)$ the unity. Define:
\begin{equation} \label{eq:FE-class}
   e_F = \delta(1) \in QH^2(L).
\end{equation}
We call this class the Floer-Euler class of $\Gamma_L \to L$.

\begin{prop} \label{p:FE-mult} For every $\alpha \in QH^*(L)$ we have:
   $$\delta(\alpha) = \alpha*e_F = e_F*\alpha.$$
\end{prop}
\begin{proof}
   The proof follows easily by noting the multiplicative properties of
   the morphisms $i$ and $p$ (see~\eqref{eq:multip-i-p-2} in
   Theorem~\ref{T:chain-map-1}) together with the fact that the pearly
   differentials on $\mathcal{C}(L)$ and on $\mathcal{C}(\Gamma_L)$
   satisfy the Leibniz rule with respect to the quantum chain level
   operation.
\end{proof}

\section{The positive pearl complex} \label{S:positive} Recall
from~\cite{Bi-Co:Yasha-fest, Bi-Co:rigidity} that the quantum
cohomology of a {\em monotone} Lagrangian $K$ admits also a positive
version, $Q^+H^*(K)$. The construction goes as follows. Let $\Lambda^+
= \mathbb{Z}_2[t]$ be the ring of polynomials in $t$, graded so that
$|t| = N_K$. Let $\data = (f, \rho, J)$ be a pearly data and put
$\mathcal{C}_+ (K; \data) = \mathbb{Z}_2 \langle \textnormal{Crit}(f)
\rangle \otimes \Lambda^+$. We grade $\mathcal{C}_+$ in the same way
as $\mathcal{C}$, i.e. by Morse indices on the left factor and using
the grading of $\Lambda_+$ on the right factor. We endow
$\mathcal{C}_+(K; \data)$ with the same differential $d$ which was
defined for $\mathcal{C}(K; \data)$ in~\S\ref{S:hf}. The fact that
this $d$ maps $\mathcal{C}_+$ into $\mathcal{C}_+$ follows from the
the monotonicity of $K$ since the Maslov index of non-constant
holomorphic disks is always strictly positive.

The cohomology of $(\mathcal{C}_+(K; \data),d)$ is denoted by
$Q^+H^*(K)$ and is called the positive quantum cohomology of $K$. By
the results of~\cite{Bi-Co:rigidity} it does not depend on $\data$.

Note that in contrast to $QH^*(K)$, $Q^+H^*(K)$ is quite different
from $HF^*(K)$ and there is no isomorphism between the two. Note also
that $Q^+H(K)$ can never vanish (unlike $HF(K)$).
See~\cite{Bi-Co:rigidity, Bi-Co:Yasha-fest} for more on that.

Note also that there is an obvious inclusion of cochain complexes
$\mathcal{C}_+(K;\data) \longrightarrow \mathcal{C}(K;\data)$. The
resulting morphism in cohomology $\theta_K : Q^+H(K) \longrightarrow
QH(K)$ is canonical. However, in general it is not injective.

Going back to our Lagrangians $L$ and $\Gamma_L$ we have:
\begin{thm} \label{T:exact-seq-positive} Theorem~\ref{T:chain-map-1}
   continues to hold if one replaces everywhere $\mathcal{C}^*$ by
   $\mathcal{C}_+^*$ and $QH^*$ by $Q^+H^*$. The corresponding class
   $e^+_F$ belongs to $Q^+H^2(L)$. Moreover the morphisms
   $\theta_L:Q^+H(L) \longrightarrow QH(L)$ and
   $\theta_{\Gamma_L}:Q^+H(\Gamma_L) \longrightarrow QH(\Gamma_L)$
   give rise to a long commutative diagram that maps the long exact
   sequence for $QH$ to the corresponding long exact sequence for
   $Q^+H$. Moreover we have $\theta_L(e^+_F) = e_F$. (Therefore, from
   now on we will denote both classes by $e_F$).
\end{thm}
\begin{proof}
   The proof is done precisely the same as for
   Theorem~\ref{T:chain-map-1} by noting that, due to monotonicity,
   all differentials, cochain maps and connecting homomorphisms in the
   proof of Theorem~\ref{T:chain-map-1} always involve only
   non-negative powers of $t$.
\end{proof}

\subsection{Comparison with the sequence in singular homology}
\label{sb:qh-vs-sing}
Let $\data = (f, \rho, J)$ be auxiliary pearl datum for the Lagrangian
$K$. Denote by $\data' = (f, \rho)$ the corresponding Morse datum, and
by $CM(K;\data')$ the corresponding Morse complex.

Denote by 
\begin{equation} \label{eq:spec}
   \widetilde{\sigma}: \mathcal{C}^*_+(K;\data)
   \longrightarrow CM^*(K;\data')
\end{equation}
the morphism induced by sending $t \in \Lambda^+$ to $0$, i.e.
$\widetilde{\sigma}(x) = x$ for every $x \in \textnormal{Crit}(f)$ and
$\widetilde{\sigma}(x t^i) = 0$ for every $i>0$. It is easy to see
that $\widetilde{\sigma}$ is a cochain map
(see~\cite{Bi-Co:Yasha-fest} Section 4.3). We denote the resulting map
in cohomology
\begin{equation} \label{eq:sigma} \sigma: Q^+H^*(K) \longrightarrow
   H^*(K;\mathbb{Z}_2).
\end{equation}
This map is canonical in the sense that it does not depend on $\data$.

Going back to the Floer-Gysin sequence we obtain the following
commutative diagram:

\[
\begin{array}{cccccccccc}
   0 & \to & \mathcal{C}_+^*(L; \data) & \stackrel{i}{\to} 
   & \mathcal{C}_+^*(\Gamma_L; \data_{\epsilon})  & 
   \stackrel{p}{\to} & & \mathcal{C}_+^{*-1}(L; \data) & 
   \to & 0 \\
   & & \widetilde{\sigma} \downarrow & & 
   \widetilde{\sigma} \downarrow & & & 
   \widetilde{\sigma}  \downarrow & \\
   0 & \to & CM^*(L; {\data}') & \stackrel{i'}{\to} & 
   CM^*(\Gamma_L; \data'_{\epsilon}) & 
   \stackrel{p'}{\to} & & CM^{*-1}(L; {\data}') & \to & 0
\end{array}
\]
where the maps $i'$ and $p'$ are defined exactly in the same way as
$i$ and $p$. Note that the long exact sequence in cohomology induced
by the bottom short sequence is precisely the Gysin sequence of the
circle bundle $\Gamma_L \to L$ for singular (or Morse) cohomology. We
now obtain a map between the two long exact sequences (induced by the
$\widetilde{\sigma}$'s):
\[
\begin{array}{cccccccccccc}
   \cdots & \to & Q^+H^k(L) & \stackrel{* e_F}{\to} 
   & Q^+H^{k+2}(L)  & 
   \stackrel{i}{\to} & & Q^+H^{k+2}(\Gamma_L) & 
   \stackrel{p}{\to} &  & Q^+H^{k+1}(L) & \to \cdots \\
   & & \sigma \downarrow & & 
   \sigma \downarrow & & & 
   \sigma  \downarrow & & & \sigma  \downarrow \\
   \cdots  & \to & H^k(L;\mathbb{Z}_2) & \stackrel{\cup e}{\to} & 
   H^{k+2}(L;\mathbb{Z}_2) & 
   \stackrel{i'}{\to} & & H^{k+2}(\Gamma_L;\mathbb{Z}_2) & 
   \stackrel{p'}{\to} &  & H^{k+1}(L;\mathbb{Z}_2) & \to \cdots 
\end{array}
\]
From this it is easy to see that $\sigma(e_F) = e$. In this sense, the
Floer-Euler class can be viewed as a deformation of the classical
Euler class.

\begin{rem} \label{r:spec-N_L>3}
   The chain map in~\eqref{eq:spec} fits into the following exact
   sequence of cochain complexes:
   \begin{equation} \label{eq:exact-seq-spec}
      \begin{CD}
         0 @>>> t\, \mathcal{C}^{*-N_L}_+(L;\data) @>>>
         \mathcal{C}^*_+(L;\data) @>\widetilde{\sigma}>> 
         CM^*(L;\data') @>>> 0
      \end{CD}
   \end{equation}
   where the first map is the inclusion.  Since
   $\mathcal{C}^{k-N_L}_+(L;\data)=0$ for every $0 \leq k < N_L$ it
   follows, after passing to the long exact sequence in cohomology,
   that $\sigma : Q^+H^k(L) \to H^k(L;\mathbb{Z}_2)$ is injective for
   every $0 \leq k <N_L$. In particular if $N_L \geq 3$ and if $e_F
   \neq 0$ then $e \neq 0 \in H^2(L;\mathbb{Z}_2)$.
\end{rem}

\section{More on the Floer-Euler class} \label{S:floer-euler-more}
Recall from~\cite{Bi-Co:rigidity} that a Lagrangian $L$ is called wide
if there exists an isomorphism of $\Lambda$--modules:
\begin{equation} \label{eq:wide} QH^*(L) \cong (H(L;\mathbb{Z}_2)
   \otimes \Lambda)^*.
\end{equation}
Note that in this case we also have:
\begin{equation} \label{eq:wide+} Q^+H^*(L) \cong (H(L;\mathbb{Z}_2)
   \otimes \Lambda^+)^*.
\end{equation}
It is important to note however, that for a wide Lagrangian $L$ there
is in general no canonical isomorphism in~\eqref{eq:wide}
or~\eqref{eq:wide+} (at least not for all degrees $*$). Therefore it
is in general {\em impossible} to make a canonical identification
$$Q^+H^2(L) = (H(L;\mathbb{Z}_2) \otimes \Lambda^+)^2.$$ Nevertheless, if $L$ is
wide and $N_L \geq 3$ the identification is possible for degree $*=2$
and we have a canonical identification:
\begin{equation} \label{eq:wide-NL=3}
   Q^+H^2(L) = H^2(L;\mathbb{Z}_2).
\end{equation}
See~\cite{Bi-Co:rigidity}, Section~4.5 for more on that. When $N_L =
2$ we still have short exact sequence: 
\begin{equation} \label{eq:wide+-N_L=2} 0 \longrightarrow
   H^0(L;\mathbb{Z}_2)t \stackrel{j}{\longrightarrow} Q^+H^2(L)
   \stackrel{\sigma}{\longrightarrow} H^2(L;\mathbb{Z}_2)
   \longrightarrow 0,
\end{equation}
where the morphism $\sigma$ is the one defined in~\eqref{eq:sigma}.

\begin{prop} \label{p:wide-EF} Let $L \subset \Sigma$ be a monotone
   wide Lagrangian. Let $e \in H^2(L;\mathbb{Z}_2)$ be the Euler class of the
   circle bundle $\Gamma_L \to L$ and $e_F \in Q^+H^2(L)$ the
   Floer-Euler class. Then:
   \begin{enumerate}
     \item If $N_L \geq 3$ then via the
      identification~\eqref{eq:wide-NL=3} we have $e_F = e$.
     \item If $N_L = 2$ then $\sigma(e_F) = e$.
   \end{enumerate}
\end{prop}

We now examine closer the case $N_L = 2$. Denote by $c_1^{\mathcal{N}}
\in H^2(\Sigma;\mathbb{Z})$ the first Chern class of the normal bundle
of $\Sigma$ in $M$ (so that if $PD [\Sigma] = k a \in
H^2(M;\mathbb{Z})$, with $a$ being an integral lift of $[\omega] \in
H^2(M;\mathbb{R})$, then $c_1^{\mathcal{N}} = k a|_{\Sigma} \in
H^2(\Sigma;\mathbb{Z})$). Note that in our notation the Euler class $e
\in H^2(L;\mathbb{Z}_2)$ is the restriction to $L$ of the modulo-$2$
reduction of $c_1^{\mathcal{N}}$.

Assume that $L$ is wide and that $e = 0 \in H^2(L;\mathbb{Z}_2)$.
From Proposition~\ref{p:wide-EF} and~\eqref{eq:wide+-N_L=2} it follows
that $e_F = r\,t$ for some $r \in \mathbb{Z}_2$. We would now like to
identify this coefficient $r$.

Let $A \in H_2(\Sigma, L)$ with $\mu(A)=2$. Denote by
$\mathcal{M}(A,J_{\Sigma})$ the space of $J_{\Sigma}$-holomorphic
disks $u:(D, \partial D) \to (\Sigma,L)$ with $u_*[D] = A$. Denote by
$G = \textnormal{Aut}(D) \cong PSL(2,\mathbb{R})$ the group of
biholomorphisms of $D$. The group $G$ acts on
$(\mathcal{M}(A,J_{\Sigma}) \times \partial D)$ by $\sigma \cdot (u,z)
= (u\circ \sigma^{-1}, \sigma(z))$. We now have an evaluation map
$$ev: (\mathcal{M}(A,J_{\Sigma}) \times \partial D)/G \longrightarrow L, 
\quad ev(u,z) = u(z).$$ As $ev$ is a smooth map between two closed
manifolds of the same dimension it has a well defined degree modulo
$2$ which we denote $\nu(A) \in \mathbb{Z}_2$. Define now a class $D_L
\in H_2(\Sigma,L)$ by $$D_L = \sum_{A \in H_2(\Sigma,L), \, \mu(A)=2}
\nu(A) A.$$ By Proposition~4.2.1 of~\cite{Bi-Co:rigidity} wideness of
$L$ implies that $\partial(D_L) = 0$, where $\partial: H_2(\Sigma,L;
\mathbb{Z}_2) \to H_1(L;\mathbb{Z}_2)$ is the connecting homomorphism.
Denoting by $i: H_2(L;\mathbb{Z}_2) \to H_2(\Sigma;\mathbb{Z}_2)$ and
by $j:H_2(\Sigma;\mathbb{Z}_2) \to H_2(\Sigma,L;\mathbb{Z}_2)$ the
homomorphisms induced by inclusion, it follows that there exists an
element $S_L \in H_2(\Sigma;\mathbb{Z}_2)$ so that $j(S_L) = D_L$, and
moreover that $S_L$ is unique upto a summand coming from
$i(H_2(L;\mathbb{Z}_2))$. Denote by $c \in H^2(\Sigma;\mathbb{Z}_2)$ the
modulo-$2$ reduction of $c_1^{\mathcal{N}} \in
H^2(\Sigma;\mathbb{Z})$. As $c|_L = e = 0 \in H^2(L;\mathbb{Z}_2)$,
the value of $\langle c, S_L \rangle \in \mathbb{Z}_2$ depends only on
$D_L$.

\begin{prop} \label{p:quantum-part-of-e_F} Let $L \subset \Sigma$ be a
   wide Lagrangian with $N_L = 2$ and with $e=0$. Then $$e_F = \langle
   c, S_L \rangle t.$$
\end{prop}
The proof is rather straightforward and follows from the definition of
the Floer-Euler class $e_F$ as the image of $1 \in QH^0(L)$ under the
connecting homomorphism: $e_F = \delta(1)$. We therefore omit the
details.

Next we would like to establish a relation between the first Chern
class of the normal bundle $\mathcal{N} \to \Sigma$ of $\Sigma$ in $M$
and the Floer-Euler class $e_F \in QH^2(L)$. Recall
from~\cite{Bi-Co:qrel-long, Bi-Co:rigidity, Bi-Co:Yasha-fest} that
$QH(L)$ has a structure of a module over the quantum cohomology
$QH(\Sigma;\Lambda) = H(\Sigma) \otimes \Lambda$ of the ambient
manifold $\Sigma$, where the latter is endowed with the quantum
product ring structure. For reasons of compatibility with $QH(L)$ we
use here $\Lambda$ as the coefficients for $QH(\Sigma;\Lambda)$, which
is an obvious extension of the usual ring of coefficients commonly
used for $QH(\Sigma)$. (See Section~2.5 of~\cite{Bi-Co:Yasha-fest} or
Section~2.1.2 of~\cite{Bi-Co:rigidity} for more details on this.) This
module structure is given by a degree preserving morphism:
$$QH(\Sigma;\Lambda) \otimes_{\Lambda} QH(L) \longrightarrow QH(L), \quad
a\otimes \alpha \longmapsto a*\alpha.$$ (Since this module structure
is compatible with the quantum multiplications of both $QH(\Sigma)$
and $QH(L)$ we have denoted it by abuse of notation by $*$ too.) A
similar construction works with $\Lambda$ replaced by $\Lambda^+$
everywhere.

Consider now the map $$r_{\scriptscriptstyle L}: QH^*(\Sigma;\Lambda)
\longrightarrow QH^*(L), \quad a \longmapsto a*1.$$ We view this map
as a quantum analogue of the classical restriction map $H^*(\Sigma)
\to H^*(L)$, \, $a \mapsto a|_L$.  Note that the image of
$c_1^{\mathcal{N}}$ under the classical restriction is the classical
Euler class $e \in H^2(L)$.  The following proposition is a quantum
version of this:
\begin{prop} \label{p:quant-rest-c_1}
   Let $L \subset \Sigma$ be a monotone Lagrangian. Denote by $c \in
   H^2(\Sigma;\mathbb{Z}_2)$ the modulo-$2$ reduction of
   $c_1^{\mathcal{N}} \in H^2(\Sigma;\mathbb{Z})$. Then
   $$e_F = r_{\scriptscriptstyle L}(c).$$
\end{prop}
The proof is again straightforward and is based on a Morse theoretic
interpretation of the class $c_1^{\mathcal{N}} \in H^2(\Sigma)$ using
the classical Gysin sequence for the circle bundle $\mathcal{P} \to
\Sigma$. We omit the details.

\section{An analogous exact sequence in $(M, \omega)$}
\label{S:seq-in-closed}

In this section we discuss the analogous sequence which arises when
one replaces the ambient manifold $W = M \setminus \Sigma$ with $M$.
Recall from~\cite{Bi-Co:rigidity} (Section~6.4) that for a suitable
choice of the parameter $r_0$ in the construction of $\Gamma_L$
in~\S\ref{Sb:lag-s1} $\Gamma_L$ becomes monotone also when viewed as a
Lagrangian submanifold of $M$. (In contrast with $\Gamma_L \subset W$,
here there is a unique $r_0$ which makes $\Gamma_L$ monotone in $M$).

Assume that the minimal Maslov number $N_L$ of $L$ is even and $\geq
2$. As in Proposition~\ref{P:monotone-gl} the homomorphism $\pi_2(M
\setminus \Delta, \Gamma_L) \to \pi_2(M, \Gamma_L)$ induced by the
inclusion is surjective. We also have:
\begin{equation} \label{eq:p2_2-M-Delta}
   \pi_2(M \setminus \Delta,
   \Gamma_L) \cong \pi_2(\mathcal{N}, \Gamma_L) = \mathbb{Z} F \oplus
   \pi_2(\mathcal{N} \setminus \Sigma, \Gamma_L),
\end{equation}
where $F$ is the class represented by the vertical disks in the fibres
of the disk bundle $E_{r_0} \to \Sigma$, i.e. by $\{ v \in
\mathcal{N}_p \mid |v| \leq r_0\}$. (Here $p$ is a point in $L$ and
$\mathcal{N}_p$ is the fibre over $p$.) Moreover the Maslov class of
$\Gamma_L$ in $M$ behaves as follows (compare to~\ref{P:monotone-gl}):
$$\mu_{\Gamma_L}(F)=2, \quad \mu_{\Gamma_L}(A) = \mu_L(\pi_*(A)), \;\;
\forall A \in \pi_2(\mathcal{N} \setminus \Sigma, \Gamma_L).$$
It follows that $N_{\Gamma_L} = 2$.

Since the minimal Maslov numbers of $\Gamma_L$ and $L$ are now
different we will use the following extension of the coefficient ring
for the pearl complex of $L$. Put $\mathcal{A} = \mathbb{Z}_2[q^{-1},
q]$, with $|q|=2$ and let $\Lambda = \mathbb{Z}_2[t^{-1}, t]$ with
$|t|=N_L$ as before. We define on $\mathcal{A}$ a structure of an
$\Lambda$-algebra via the ring homomorphism $\Lambda \ni t \longmapsto
q^{\frac{N_L}{2}} \in \mathcal{A}$. Given auxiliary data $\data$, we define the
pearl complex on $L$ using coefficients in $\mathcal{A}$:
$$\mathcal{C}(L; \data; \mathcal{A}) = \mathcal{C}(L;\data) 
\otimes_{\Lambda} \mathcal{A},$$ with the obvious extension of the
pearly differential by linearity over $\mathcal{A}$. We denote the
correpsonding cohomology by $QH(L; \mathcal{A})$. As for $\Gamma_L$ we
define the data $\widetilde{\data}_{\varepsilon}=(f_\varepsilon,
X_\varepsilon, J)$ as in~\S\ref{S:short-exact}. Here $J$ is an
admissible almost complex structure induced from $J_{\Sigma}$ as
described in~\S\ref{S:stretching} but now $J$ is defined on the whole
of $M$.  Note that by the construction in~\S\ref{S:stretching} such
$J$'s coincide with $J_{\Sigma}$ on $\Sigma$, hence $\Sigma$ is a
$J$-holomorphic submanifold. The pearl complex of $\Gamma_L \subset M$
is defined as usual, but we denote the coefficients by $\mathcal{A}$
(rather than $\Lambda$ which is already used for $L$). In order to
distinguish the pearl complex of $\Gamma_L \subset M$ from that of
$\Gamma_L \subset W$ we denote the former by
$(\mathcal{C}_{M}(\Gamma_L; \widetilde{\data}_{\varepsilon}),
\widetilde{d}_M)$ and the latter by $(\mathcal{C}_{W}(\Gamma_L;
\widetilde{\data}_{\varepsilon}), \widetilde{d}_W)$. We denote their
cohomologies by $QH_M(\Gamma_L)$ and $QH_W(\Gamma_L)$ respectively.

In this new setup a similar version of
Proposition~\ref{p:stretching-neck} holds, namely:
\begin{prop}\label{p:stretching-neck-M}
   For generic $\data$, there exists $R_0 > 0$ such that for every
   $J_R$ as described above with $R > R_0$ the following holds: every
   pearly trajectory $\gamma \in \pearlspace_0 (J_R)$ is contained in
   the image $F(E_{r_0+\varepsilon})$ of the $(r_0+\varepsilon)$-disk
   bundle of $\mathcal{N}$ under $F$.
\end{prop}
\begin{proof}
   The proof is almost identical to that of
   Proposition~\ref{p:stretching-neck} except of the following points.
   First, by the maximum principle, if $u:(D, \partial D) \to (M,
   \Gamma_L)$ is a $J_R$-holomorphic disk then either $u(D)$ is
   contained in $M \setminus \textnormal{Int\,} E_{r_0}$
   or $u(D)$ intersects $\Sigma$. For those disks that lie entirely in
   $M \setminus \textnormal{Int\,} E_{r_0}$ the proof of
   Proposition~\ref{p:stretching-neck} holds without any change. 

   Now suppose that we have a sequence of pearly trajectories
   $\gamma_{n}$ which contain $J_{R_n}$-holomorphic disks
   $u_{R_n}$ such that $u_{R_n}(D)$ intersect $\Sigma$ (as well as the
   complement of $E_{r_0+\varepsilon}$, as was assumed in the proof of
   Proposition~\ref{p:stretching-neck}).  Arguing exactly as in the
   proof of Proposition~\ref{p:stretching-neck} we obtain a
   holomorphic building in $M^{\infty}$ part of which, say $\bar{u}'$
   is in $M^-_{\infty}$ and another part $\bar{u}''$ in
   $M^{+}_{\infty}$ (which also intersects $\Sigma$). There may appear
   additional part $\bar{u}'''$ whose components lie in the cylinder 
   $\R \times P$. The first part,
   $\bar{u}'$, can be analyzed and dealt with as in the proof of
   Proposition~\ref{p:stretching-neck}. In particular we assume that
   one of its components $u_2$, after being perturbed to lie away from
   $\Delta$, projects into a sphere $v$ in $\Sigma$ with positive
   Chern number. The second part, $\bar{u}''$ might contain components
   of the following kinds:
   \begin{enumerate}
     \item holomorphic spheres $u_s''$ (appearing as bubbles) in
      $M^+_{\infty}$.
     \item disks $u_d''$ in the class $F$ (or its multiples).
     \item holomorphic curves $u_{-}''$ similar to $u_1$ from the
      proof of Proposition~\ref{p:stretching-neck} defined on a
      punctured disk or sphere and at the punctured asymptotically go
      to a periodic orbit at $-\infty$ in $M^+_{\infty}$.
     \item some other genuine holomorphic disks $u''_o$ in
      $(M^{\infty}, \Gamma_L)$ (lying in a compact part of
      $M^{+}_{\infty}$).
   \end{enumerate}
   Note that the projection of the disks $u_d''$ via $\pi$ must be
   constant (since $\pi_*(F)=0$), hence these disks are vertical.
   Next, the projection of the curves of the type $u_{-}''$ gives us
   in $\Sigma$ a holomorphic curve with a removable singularities at
   the punctures, precisely as was done with $\pi_1 \circ u_1$ in the
   proof of Proposition~\ref{p:stretching-neck}. The disks of the type
   $u''_o$ project to genuine holomorphic disks in $(\Sigma,L)$.
   Components of $\bar{u}'''$ (if any) project to holomorphic spheres.
   Consider now the pearly trajectory $\overline{\gamma}$ obtained
   from the limit of the $\gamma_n$. We remove from
   $\overline{\gamma}$ the component $u_2$, $\bar{u}'''$, the vertical disks
   $u''_d$ (if there are any) and the holomorphic spheres $u''_s$ (if
   there are any), and then project the rest to $\Sigma$ via $\pi$.
   We thus obtain a genuine pearly trajectory $\gamma_{\Sigma}$ for
   $(\Sigma, L)$. Denote by $A \in H_2(M, \Gamma_L)$ the total
   homology class of the holomorphic curves (including $u_2$) in
   $\overline{\gamma}$. Since the vertical disks (if there are any)
   have constant projection the total homology class $B \in
   H_2(\Sigma,L)$ of the holomorphic curves involved in
   $\gamma_{\Sigma}$ is:
   $$B = \pi_*(A) - [v] - \pi_*[u_s''] - \pi_*[\bar{u}'''].$$ (In case there are no spheres 
   $u_s''$, we have $[u_s''] = 0$).

   Now for every $C \in H_2(\mathcal{N}, \Gamma_L)$ we have
   (see~\eqref{eq:p2_2-M-Delta}):
   $$\mu_{\Gamma_L}(C) = \mu_{L}(\pi_*(C)) + 2C \cdot [\Sigma],$$ 
   where $C \cdot \Sigma$ stands for the intersection number between
   $C$ and $\Sigma$. We thus obtain:
   \begin{equation} \label{eq:mu-gl-M}
      \mu_{\Gamma_L}([\bar{u}]) = \mu_{\Gamma_L}(A) = \mu_L(B) + 
      2c_1^{\Sigma}([v]) + 2c_1^{\Sigma}(\pi_*[u_s'']) + 2c_1^{\Sigma}(\pi_*[\bar{u}''']) + 2 A \cdot [\Sigma].
   \end{equation}
   We now claim that $A \cdot [\Sigma] \geq 0$. Indeed the class $A$
   is represented by $J_{R_n}$-holomorphic disks (those that appear in
   each of the $\gamma_n$'s) and $\Sigma$ is $J_{R_n}$-holomorphic.
   The claim follows from positivity of intersections.
   
   Next, by monotonicity we have $c_1^{\Sigma}([u_s'']) \geq 0$. By
   the same argument as in the proof of
   Proposition~\ref{p:stretching-neck} we also have $c_1^{\Sigma}[v]
   \geq 1$ (in contrast to $u''_s$, we explicitely assumed that $v$
   does occur). Going back to~\eqref{eq:mu-gl-M} we obtain the
   inequality $$\mu_{\Gamma_L}([\bar{u}]) \geq \mu_L(B) + 2,$$ which
   is the same as~\eqref{eq:proj-u-bar} in the proof of
   Proposition~\ref{p:stretching-neck}.

   The rest of the proof continues exactly as for
   Proposition~\ref{p:stretching-neck}.
\end{proof}

Having established Proposition~\ref{p:stretching-neck-M} we can prove
transversality for moduli spaces involved in $\mathcal{C}_M(\Gamma_L;
\widetilde{\data}_{\varepsilon})$ in the same way done
in~\S\ref{S:transversality}.

We now define the maps $i : \mathcal{C}^*(L; \data; \mathcal{A}) \to
\mathcal{C}_M^*(\Gamma_L; \widetilde{\data}_{\varepsilon})$ and $p :
\mathcal{C}_M^*(\Gamma_L; \widetilde{\data}_{\varepsilon}) \to
\mathcal{C}^{*-1}(L; \data; \mathcal{A})$ exactly as
in~\S\ref{S:short-exact}.

It remains to show that these remain chain maps also with respect to
the pearly differential $\widetilde{d}_M$ of $\Gamma_L $ in $M$.  The
proof of this goes along the same lines as that for $W$: we compare
pearly trajectories in $(M, \Gamma_L)$ with those on $(\Sigma, L)$.
In particular we project pearly trajectories from $(M, \Gamma_L)$ to
$(\Sigma, L)$ as we did for $(W, \Gamma_L)$. The discussion for the
lifting property which is presented in~\S\ref{S:lifting} applies here
with the following modifications. Lemma~\ref{l:lift-a-disk} shows that
any holomorphic disk $u : (D^2, \partial D^2) \to (\Sigma, L)$ admits
a unique holomorphic lift to a disk in $(M \setminus \Sigma,
\Gamma_L)$ having specified appropriate boundary conditions. In
addition, there is a family of lifts of $u$ to holomorphic disks which
intersect $\Sigma$.  For any such lift $\tilde{u}$ of $u$, we have
$\mu_{\Gamma_L} (\tilde{u}) = \mu_L (u) + 2 [u] \cdot [\Sigma]$ (here
$[u] \cdot [\Sigma]$ stands for the intersection product in homology).
A simple index computation shows that the virtual dimension of any
{\em lifted} trajectory which contains disks intersecting $\Sigma$ is
greater than zero, so these do not contribute to the differential. In
addition there may appear trajectories in $(M, \Gamma_L)$ whose
projections to $(\Sigma, L)$ are degenerate.  Analyzing possible
configurations of degenerated trajectories, one shows that the only
index $0$ trajectories which appear this way are trajectories
consisting of a single vertical disk in the fiber of $\pi : E_{r_0}
\to \Sigma$.  For every critical point $x$ on $L$ there is unique such
trajectory which connects $x''$ to $x'$ (its projection to $\Sigma$
consists of a single point $x$).

The above discussion yields the following identities for every $x \in
\textnormal{Crit}(f)$:
\begin{equation} \label{eq:chain-map-M}
   \begin{array}{rcl} \widetilde{d}_M (x')  &=& \widetilde{d}_W (x'), \\
      \widetilde{d}_M (x'') &=& \widetilde{d}_W (x'') - x' q \end{array}
\end{equation}
We have indentified here $\cplx_M^*(\Gamma_L;
\widetilde{\data}_{\varepsilon})$ with $\cplx_W^*(\Gamma_L;
\widetilde{\data}_{\varepsilon}) \otimes_{\Lambda} \mathcal{A}$ as
graded vector spaces. The additional summand $x' \otimes q$ in
$\widetilde{d}_M (x'')$ comes from the vertical disks described above.

A straighforward computation now shows that $i$ and $p$ are chain
maps.  We now have the following version of
Theorem~\ref{T:chain-map-1}:
\begin{thm} \label{T:chain-map-M}
	The maps $i$ and $p$ form a short exact sequence 
  \begin{equation*}
      \begin{CD}
         0 @>>> \cplx^*(L; \data;\mathcal{A}) @>{i}>>
         \cplx_M^*(\Gamma_L; \widetilde{\data}_{\varepsilon}) @>{p}>>
         \cplx^{*-1}(L; \data;\mathcal{A}) @>>> 0
      \end{CD}
  \end{equation*}
  of cochain complexes. For a generic choice of data $\data$ and an
  admissible corresponding data $\widetilde{\data}_{\varepsilon}$ the
  maps $i$ and $p$ are chain maps.  In particular, we have a long
  exact sequence
  \begin{equation*}
      \begin{CD}
         \cdots @>>> QH^{k}(L;\mathcal{A}) @>{\delta} >>
         QH^{k+2}(L;\mathcal{A}) @>{i}>> QH_M^{k+2}(\Gamma_L) @>{p}>>
         QH^{k+1}(L;\mathcal{A}) @>{\delta}>> \cdots
      \end{CD}
  \end{equation*}
  Moreover, this exact sequence in homology is canonical in the sense
  that it does not depend on the auxiliary data. The connecting
  homomorphism $\delta$ is given by quantum multiplication by a class
  $e'_F \in QH^2(L;\mathcal{A})$ (which does not depend on the
  auxiliary data), i.e.  $\delta(\alpha) = \alpha * e'_F$ for every
  $\alpha \in QH^*(L;\mathcal{A})$. The relation between $e'_F$ and
  the Euler-Floer class from Theorem~\ref{T:chain-map-1} is given by
  $e'_F = e_F - q$, where we view here $e_F$ as a class in
  $QH^2(L;\mathcal{A})$.
\end{thm}
The independence of the choice of auxiliary data issues are treated in
a similar way to those in for $W$. Finally,~\eqref{eq:chain-map-M}
implies that the connecting homomorphisms $\delta
: QH^{k}(L) \to QH^{k+2}(L)$ in the sequences for $(M, \Gamma_L)$ and
that for $(W, \Gamma_L)$ are related as follows:
\begin{equation} \label{eq:boundary-M} 
   \delta_M = \delta_W - q.
\end{equation}
(Here $q$ stands for multiplication by $q$.) The fact that $e'_F =
e_F - q$ follows now from similar arguments as
in~\S\ref{S:floer-euler}.

\section{Further results and generalizations} \label{S:further} Here
we present a generalization of Theorem~\ref{T:exact-seq-1} that allows
to replace $\Sigma$ by a product $\Sigma \times Q$ with a symplectic
manifold $Q$. Here is the precise setting.

Let $(Q, \omega_Q)$ be a closed symplectic manifold. Let $L \subset
(\Sigma \times Q, \omega_{\Sigma} \oplus \omega_Q)$ be a Lagrangian
submanifold. Define the circle bundles $P_r \to \Sigma$ as
in~\S\ref{Sb:std-disk-bundle}. Denote by $\pi'_{r_0} : P_{r_0} \times
Q \to \Sigma \times Q$ the projection and define $$\Gamma_L = (F
\times \Id)({\pi'_{r_0}}^{-1}(L)) \subset W \times Q,$$ where $F$ is
the embedding from Proposition~\ref{P:embedding}. A simple computation
shows that $\Gamma_L$ is a Lagrangian submanifold of $W \times Q$ if we
endow this manifold with the symplectic structure
$$\omega_{r_0} = \omega \oplus e^{-r_0^2} \omega_Q.$$

We now fix $r_0$ once and for all and consider $\Gamma_L$ as
Lagrangian submanifold of $(W \times Q, \omega_{r_0})$. We have the
following version of Proposition~\ref{P:monotone-gl} which is proved
in~\cite{Bi:Nonintersections}:
\begin{prop} \label{P:monotone-gl-WxQ} Assume that either
   $\dim_{\mathbb{R}} \Sigma \geq 4$, or that $\dim_{\mathbb{R}}
   \Sigma = 2$ and $W = M \setminus \Sigma$ is subcritical. Let $(Q,
   \omega_{Q})$ be as above and $L \subset \Sigma \times Q$ be a
   Lagrangian submanifold. Let $\Gamma_L \subset W \times Q$ be the
   Lagrangian circle bundle over $L$ as constructed above.  Then:
   \begin{enumerate}
     \item The homomorphism $\iota_*: \pi_2(W \times Q \setminus
      \Delta \times Q, \Gamma_L) \to \pi_2(W\times Q, \Gamma_L)$,
      induced by the inclusion, is surjective. When $\dim_{\mathbb{R}}
      \Sigma \geq 6$, $\iota_*$ is an isomorphism. The same statement
      holds also for homology, i.e.  if one replaces $\pi_2$ by
      $H_2$.
     \item For every $B \in \pi_2 (W \times Q \setminus \Delta \times
      Q, \Gamma_L)$ we have:
      $$\mu_{\Gamma_L}(B) = \mu_L(\pi'_*(B)),$$
      where $\pi': (W \times \Delta) \times Q \to \Sigma \times Q$ is
      the projection induced by $W \setminus \Delta \to \Sigma$.
   \end{enumerate}
   In particular, if $L \subset \Sigma \times Q$ is monotone then
   $\Gamma_L \subset W \times Q$ is monotone too, and $N_{\Gamma_L} =
   N_L$.
\end{prop}
Note that if $L \subset \Sigma \times Q$ is monotone then in
particular $(Q, \omega_{Q})$ is a spherically monotone manifold, i.e.
there exists $\lambda > 0$ so that $\omega_Q(S) = \lambda c^{Q}_1(S)$
for every $S \in \pi_2(\Sigma)$.

We now have the following generalization:
\begin{thm} \label{t:chain-map-WxQ}
   Theorems~\ref{T:chain-map-1},~\ref{T:exact-seq-positive}, the
   discussion in~\S\ref{sb:qh-vs-sing} as well as
   Propositions~\ref{p:quantum-part-of-e_F},~\ref{p:quant-rest-c_1}
   continue to hold for monotone $L \subset \Sigma \times Q$ and
   $\Gamma_L \subset W \times Q$.
\end{thm}
The proof is very similar to the proofs of the analogous statements
for the case $Q = \textnormal{pt}$, i.e. $L \subset \Sigma$ and
$\Gamma_L \subset W$. Still, there are a few points where some
adjustments are needed. We indicate them below.

First of all, the construction of the chain maps $i$ and $p$ is the
same as before. As for the almost complex structures, we use the
following adjustments. Fix a generic $\omega_{\Sigma}$-tamed almost
complex structure $J_{\Sigma}$ on $\Sigma$ and an $\omega_Q$-tamed
almost complex structure $J_Q$ on $Q$. Let $J^0_{\Sigma \times Q} =
J_{\Sigma} \oplus J_Q$ be the split almost complex structure on
$\Sigma \times Q$. We will work with almost complex structures 
$J_{\Sigma \times Q}$ on $\Sigma \times Q$ that are generic small
perturbations of $J^0_{\Sigma \times Q}$.  This class is obviously
enough in order to obtain transversality for the pearl complex of $L
\subset \Sigma \times Q$. Given such a generic $J_{\Sigma \times Q}$
we construct, as in~\S\ref{S:stretching}, the almost complex
structures on $\mathcal{N} \times Q$, $M \times Q$, $W \times Q$ etc.
as well as their stretched versions on $W^R \times Q$ etc. We denote
the resulting almost complex structure by $\widetilde{J}_{\Sigma
  \times Q}$ (we omit here the parameter $R$ to simplify the
notation). Note that $W \times Q$ is not symplectically convex at
infinity anymore, and the maximum principle does not apply due to the
$Q$ factor. To go about this difficulty we fix $0<r_1< r_0$ and adjust
$\widetilde{J}_{\Sigma \times Q}$ on $(\textnormal{Int\,} E_{r_1})
\times Q$ so that it coincides with $J_W \oplus J_Q$ (i.e.  the lift
of $J^0_{\Sigma \times Q}$) on $(\textnormal{Int\,} E_{r_1/2}) \times
Q$.  We denote the resulting almost complex structure by
$\widetilde{J}'_{\Sigma \times Q}$ and call them admissible.  Such
almost complex structures are enough in order to ensure compactness
for holomorphic disks in $W \times Q$ with boundary on $\Gamma_L$. The
reason is that the projection to $W$ is holomorphic on
$(\textnormal{Int\,} E_{r_1/2}) \times Q$ and the maximum principle
applies to these projections. Thus holomorphic disks with boundary on
$\Gamma_L$ cannot escape to infinity.

The preceding construction of admissible almost complex structures
creates however a new problem. The problem is that due to the
perturbation in $E_{r_1} \times Q$ these almost complex structures are
not compatible with the projection $(W \setminus U)\times Q \to
\Sigma \times Q$ in the domain $(E_{r_1} \times Q$ (in the sense that
the projection is not holomorphic anymore). This compatibility was
crucial in the proof of Proposition~\ref{p:stretching-neck}.  To solve
this problem, fix $r_0'$ such that $0< r_1 < r_0' < r_0$.  We claim
that for $J_{\Sigma \times Q}$ close enough to $J^0_{\Sigma \times Q}$
and admissible $\widetilde{J}'_{\Sigma \times Q}$'s induced by such
$J_{\Sigma \times Q}$'s the following holds: all $\widetilde{J}$
holomorphic disks $u:(D,
\partial D) \to (W \times Q, \Gamma_L)$ lie in the domain $(M
\setminus E_{r_0'}) \times Q$. Indeed if the contrary would happen
then there exists a sequence $J_n \to J^0_{\Sigma \times Q}$ on
$\Sigma \times Q$ and a sequence of corresponding admissible almost
complex structures $\widetilde{J}'_n$ on $W \times Q$ together with
$\widetilde{J}'_n$-holomorphic disks $u_n$ whose image intersects
$E_{r_0'}$ for every $n$. In the limit, when $n \to \infty$,
$\widetilde{J}'_n$ converges to a split almost complex structure
$\widetilde{J}_0 = J_{W} \oplus J_{Q}$ and (after passing to a
subsequence) the disks $u_n$ converge to a
$\widetilde{J}_0$-holomorphic curve $u_{\infty}$ (with some bubble
components) with boundary on $\Gamma_L$. As $\widetilde{J}_0$ is a
split almost complex structure the projection of $u_{\infty}$ to $W$
is $J_W$ holomorphic.  The projection of its boundary lies in
$P_{r_0}$ and there is an interior point lying in $E_{r_0'}$. This
contradicts the maximum principle. It now follows that all pearly
trajectories lie above the hypersurface $P_{r_0'} \times Q$, where the
projection to $\Sigma \times Q$ is indeed holomorphic.

There is yet another point in the proof of
Proposition~\ref{p:stretching-neck} where an additional argument is
needed. One has to take care of another possible configuration of
holomorphic curves appearing in the limit while stretching the neck.
Namely, holomorphic spheres that might appear in the holomorphic
building $\bar{u}$ as bubbles from the limit of the sequence
$u_{n_k}$. These spheres might appear now since $W \times Q$ is not
exact symplectic manifold anymore, due to the $Q$ factor.  However,
due to the monotonicity of $Q$ these spheres have positive Chern
numbers hence the total Maslov index of $\bar{u}$ still drops after
removing them, and a similar argument to the proof of
Proposition~\ref{p:stretching-neck} goes through.

The other components in the proof of
Proposition~\ref{t:chain-map-WxQ}, such as the lifting and the
transversality are carried out in a similar way to the case $Q =
\textnormal{pt}$ with almost no significant adjustments.
\Qed

\section{Applications and examples} \label{S:applications-exp}

Recall that a symplectic manifold $(\Sigma, \omega_{\Sigma})$ is
called {\em spherically monotone} if there exists $\lambda > 0$ such
that $\omega_{\Sigma}(S) = \lambda c^{\Sigma}_1(S)$ for every $S \in
\pi_2(S)$. We define the minimal Chern number of $\Sigma$ to be:
$$C_{\Sigma} = \min \{ c^{\Sigma}_1(S) \mid S \in \pi_2(S), \; 
c^{\Sigma}_1(S) > 0\}.$$ We use the convention that $\min \emptyset =
\infty$ (e.g. in case $\pi_2(\Sigma)=0$).

\begin{thm} \label{t:QH-Sigma-2-priodic} Let $(\Sigma,
   \omega_{\Sigma})$ be a spherically monotone symplectic manifold
   with minimal Chern number $C_{\Sigma}$. Suppose that $(\Sigma,
   \omega_{\Sigma})$ can be embedded as a symplectic hyperplane
   section in a symplectic manifold $M$ so that $M \setminus \Sigma$
   is subcritical. Then $C_{\Sigma} < \infty$ and
   $H^{*(\textnormal{mod} 2C_{\Sigma})}(\Sigma;\mathbb{Z}_2)$ is
   $2$--periodic, i.e. for every $k \in \mathbb{Z}$ we have:
   $$\bigoplus_{i \in \mathbb{Z}} 
   H^{k + 2iC_{\Sigma}}(\Sigma;\mathbb{Z}_2) \cong \bigoplus_{i \in
     \mathbb{Z}} H^{k +2 + 2iC_{\Sigma}}(\Sigma;\mathbb{Z}_2).$$
\end{thm}
The simplest example when this happens is $\Sigma = {\mathbb{C}}P^n$.
Then we have $C_{\Sigma} = n+1$ and the $2$-periodicity is easy to
verify. More examples of $\Sigma \subset M$ with subcritical
complement can be found in~\cite{Bi-Je:small-dual}.

A theorem similar to~\ref{t:QH-Sigma-2-priodic}, with coefficients in
$\mathbb{Z}$, has been recently obtained in~\cite{Bi-Je:small-dual},
without any appeal to Lagrangian submanifolds. The theorem
in~\cite{Bi-Je:small-dual} deals with projectively embedded algebraic
manifolds which have a so called {\em small dual}. This class of
manifolds is closely related to the subcriticality of $M \setminus
\Sigma$ (see~\cite{Bi-Je:small-dual} for more details).

\begin{proof}[Proof of Theorem~\ref{t:QH-Sigma-2-priodic}]
   We will derive our theorem from Theorem~\ref{t:chain-map-WxQ}.

   Put $(Q, \omega_Q) = (\Sigma, -\omega_{\Sigma})$ so that 
   $$(\Sigma \times Q, \omega_{\Sigma} \oplus \omega_Q) = 
   (\Sigma \times \Sigma, \omega_{\Sigma} \oplus -\omega_{\Sigma}).$$
   Let $L = \{(x,x) \mid x \in \Sigma \} \subset \Sigma \times Q$ be
   the diagonal embedding of $\Sigma$. Then $L$ is Lagrangian and it
   is easy to see that it is monotone with minimal Maslov number
   $N_L = 2C_{\Sigma}$.

   Put $W = M \setminus \Sigma$ and let $\Gamma_L \subset W \times Q$
   be the Lagrangian circle bundle over $L$ as constructed
   in~\S\ref{S:further}. By Proposition~\ref{P:monotone-gl-WxQ} $\Gamma_L$ is
   monotone too and $N_{\Gamma_L} = N_L = 2C_{\Sigma}$.

   Since $W$ is subcritical we have $HF(\Gamma_L)=0$ hence
   $QH(\Gamma_L)=0$. By Theorem~\ref{t:chain-map-WxQ} the Floer-Gysin
   long exact sequence splits into many isomorphisms:
   \begin{equation} \label{eq:QH-k=QH-k+2}
      QH^k(L) \cong QH^{k+2}(L).
   \end{equation}

   Next recall that there is a graded isomorphism of $\Lambda$
   modules: $QH^*(L) \cong HF^*(L,L)$. It is well known that for $L =
   \textnormal{diagonal} \subset (\Sigma \times \Sigma,
   \omega_{\Sigma} \oplus - \omega_{\Sigma})$ the self Floer
   cohomology $HF^*(L,L)$ is isomorphic as a graded $\Lambda$-module
   to $(H(L;\mathbb{Z}_2) \otimes \Lambda)^*$ (see
   e.g.~\cite{FO3:Chap-8, FO3-book-1}). The latter is just
   $(H(\Sigma;\mathbb{Z}_2) \otimes \Lambda)^*$.  Finally note that
   for every $k \in \mathbb{Z}$ we have:
   $$(H(\Sigma;\mathbb{Z}_2) \otimes \Lambda)^k = \bigoplus_{i \in \mathbb{Z}}
   H^{k+2iC_{\Sigma}}(\Sigma;\mathbb{Z}_2)t^{-i}.$$
\end{proof}
\begin{rem} \label{r:QH-Sigma} The isomorphism~\eqref{eq:QH-k=QH-k+2}
   is given by quantum multiplication by the Floer-Euler class $e_F
   \in QH^2(L)$. It follows that $e_F$ is an invertible class with
   respect to the Lagrangian quantum product. By
   Proposition~\ref{p:quant-rest-c_1} (see also
   Theorem~\ref{t:chain-map-WxQ}) the class $e_F$ can be written as
   the quantum restriction $e_F = r_{\scriptscriptstyle
     L}(\textnormal{pr}^*c)$. Here $c \in H^2(\Sigma;\mathbb{Z}_2)$ is
   the modulo-$2$ reduction of the first Chern class
   $c^{\mathcal{N}}_1 \in H^2(\Sigma; \mathbb{Z})$ of
   the normal bundle of $\Sigma$ in $M$, $\mathcal{N} \to \Sigma$, and
   $pr : \Sigma \times \Sigma \to \Sigma$ is the projection on the
   second factor.

   Next, note that the isomorphism $QH^*(L) \cong
   (H(\Sigma;\mathbb{Z}_2) \otimes \Lambda)^*$ can be rewritten as an
   isomorphism between the Lagrangian quantum cohomology of $L$ and
   the symplectic quantum cohomology of $\Sigma$: $QH^*(L) \cong
   QH^*(\Sigma)$. The latter isomorphism is, at least by a folklore
   result, not only an isomorphism of $\Lambda$-modules but in fact an
   isomorphism of rings (where both rings are endowed with their
   respective quantum products). The proof of this fact is rather
   straightforward, modulo transversality issues that arise when
   working with almost complex structure $J$ on $\Sigma \oplus \Sigma$
   for which the involution $(x,y) \to (y,x)$ becomes
   anti-holomorphic.
   
   Assuming this, it follows that $c \in H^2(\Sigma;\mathbb{Z}_2)$ is
   an invertible element in $QH(\Sigma)$ with respect to the quantum
   product. (c.f~\cite{Bi-Je:small-dual} for related algebro-geometric
   results over $\mathbb{Z}$.)
\end{rem}

\subsection{Proof of the Corollaries from~\S\ref{Sb:applications}}
\label{sb:proofs-applic}

\begin{proof}[Proof of Corollary~\ref{C:subcrit-1}]
   Since $W$ is subcritical any compact subset in $W$ is Hamiltonianly
   displaceable in the Weinstein completion of $W$
   (see~\cite{Bi-Ci:Stein}). In particular, $HF(\Gamma_L)=0$.
   Substituting this into the Floer-Gysin long exact sequence of
   Theorem~\ref{T:exact-seq-1} we obtain that multiplication by the
   Floer-Euler class $*e_F: HF^i(L) \to HF^{i+2}(L)$ is an isomorphism
   for every $i \in \mathbb{Z}$. This shows that $HF^i(L) \cong
   HF^{i+2}(L)$.

   Assume now that $HF(L) \neq 0$. Denote by $1 \in HF^0(L)$ the
   unity. Then there exists $a \in HF^{-2}(L)$ such that $a*e_F = 1$,
   hence $e_F$ is invertible.
\end{proof}

\begin{proof}[Proof of Corollary~\ref{C:subcrit-2}]
   As in the proof of Corollary~\ref{C:subcrit-1} we obtain that $e_F$ is
   invertible and since $HF(L) \neq 0$ we have $e_F \neq 0$ (note
   that, at least formally a zero element in the zero ring is
   invertible, so must exclude this case). By the discussion
   in~\S\ref{sb:qh-vs-sing} and in particular
   Remark~\ref{r:spec-N_L>3} we deduce that the modulo-$2$ reduction
   $e \in H^2(L;\mathbb{Z}_2)$ of the classical Euler class of the
   bundle $\Gamma_L \to L$ is not zero. This immediately implies that
   $\Gamma_L \to L$ is not trivial.

   Denote now the $\mathbb{Z}$-Euler class of $\Gamma_L \to L$ by
   $e_{\mathbb{Z}} \in H^2(L;\mathbb{Z})$ and by $e_{\mathbb{R}} \in
   H^2(L;\mathbb{R})$ its projection into the real cohomology.
   Clearly $e_{\mathbb{R}}=0$ since $e_{\mathbb{R}}$ is proportional
   to $\omega_{\Sigma}$ and $L$ is Lagrangian with respect to
   $\omega_{\Sigma}$. It follows that $e_{\mathbb{Z}}$ is torsion.
\end{proof}

We now turn to the proof of Corollary~\ref{C:quadric-1}. We will
actually prove the following more general version:
\begin{cor} \label{C:quadric-2} Let $L \subset \Sigma$ be a Lagrangian
   submanifold. Assume that $n = \dim L \geq 2$ and that $L$ satisfies
   one of the following conditions:
   \begin{enumerate}
     \item $H_1(L;\mathbb{Z}) = 0$.
     \item $n \geq 3$, $H_1(L;\mathbb{Z}) = 0$ is $2$-torsion (i.e.
      for every $\alpha \in H_1(L;\mathbb{Z})$ we have $2 \alpha = 0$)
      and either $\dim_{\mathbb{Z}_2} H^1(L;\mathbb{Z}_2) > 1$ or
      there exists $1 < i < n-1$ such that $H^i(L;\mathbb{Z}_2) \neq
      0$.
     \item $L$ is monotone and $QH(L) \neq 0$.
   \end{enumerate}
   Then $L \cap L_0 \neq \emptyset$.
\end{cor}
This corollary, under assumptions~(1) or~(2), has been proved before
in~\cite{Bi:Nonintersections} by somewhat different methods (see
Theorem~G there).

\begin{proof}
   Assume that $L \cap L_0 = \emptyset$. We will show that none of the
   conditions (1)-(3) in the statement of the corollary can be
   satisfied.

   Put $W = {\mathbb{C}}P^{n+1} \setminus \Sigma$. By the results
   of~\cite{Bi:Nonintersections} if $L \cap L_0 = \emptyset$ then
   $\Gamma_L \subset W$ is displaceable in the Weinstein completion of
   $W$, hence $HF(\Gamma_L)=0$. It follows that $e_F \in QH^2(L)$ is
   invertible.

   Next note that since $\Sigma$ is a quadric in ${\mathbb{C}}P^{n+1}$
   it normal bundle $\mathcal{N}$ in ${\mathbb{C}}P^{n+1}$ is actually
   $\mathcal{N} = \mathcal{O}_{{\mathbb{C}}P^{n+1}}(2)|_{\Sigma}$. It
   follows that the modulo-$2$ reduction of $c_1^{\mathcal{N}}$ is $0
   \in H^2(\Sigma; \mathbb{Z}_2)$. By
   Proposition~\ref{p:quant-rest-c_1} we have $e_F = 0$. But we have
   just showed that $e_F$ is invertible, hence $QH(L)=0$. This already
   rules out condition~(3).

   Assume now that~(1) holds. We will show that this implies that~(3)
   holds. Indeed, it is easy to see that $L$ is monotone and that the
   minimal Maslov number of $L$ is $N_L = 2n$. As $n \geq 2$ we have
   $N_L = 2n > n+1$ and standard arguments in Floer theory
   (see~\cite{Bi:Nonintersections}) show that $QH(L) \neq 0$. So~(3)
   holds.

   Assume now that~(2) is satisfied. We may assume that
   $H_1(L;\mathbb{Z}) \neq 0$ (otherwise we are in case~(1)).  It
   follows that $L$ is monotone and its minimal Maslov number is a
   multiple of $n$, say $N_L = kn$. If $k \geq 2$ we arrive at
   contradiction in a similar way as we did for case~(1). So assume
   that $k=1$, i.e. $N_L = n$. As $QH(L)=0$, standard arguments
   from~\cite{Bi:Nonintersections} (e.g. applying the spectral
   sequence described in that paper) show that if $n \geq 3$ then
   $H^1(L;\mathbb{Z}_2) = \mathbb{Z}_2$ and $H^i(L;\mathbb{Z}_2) = 0$
   for every $1< i < n-1$, contrary to the assumptions in~(2).
\end{proof}

\subsection{Examples revisited} \label{sb:exps-again} We review here in
retrospect the examples from the introduction after having developed
the theory in the paper.

\subsubsection{Lagrangians in ${\mathbb{C}}P^n$ with $2$-torsion
  $H_1(L;\mathbb{Z})$} \label{sbsb:lags-cpn-2-torsion}

It remains to explain here the computation of the Floer-Euler class.
Recall that $N_L = n+1$ and that there is a canonical isomorphism
$HF^*(L) \cong QH^*(L) \cong (H(L;\mathbb{Z}_2) \otimes \Lambda)^*$.
In particular $QH^2(L) \cong H^2(L;\mathbb{Z}_2) \cong \mathbb{Z}_2$.
We claim that under these identifications the Floer-Euler class $e_F$
equals the classical Euler class of the bundle $\Gamma_L \to L$ and
moreover that this must be the generator of
$H^2(L;\mathbb{Z}_2)=\mathbb{Z}_2$. To see that denote by $c \in
H^2({\mathbb{C}}P^n;\mathbb{Z}_2)$ the generator. Clearly $c$ is the
modulo-$2$ reduction of the first Chern class $c^{\mathcal{N}}_1$ of
the normal bundle of $\Sigma = {\mathbb{C}}P^n$ in $M =
{\mathbb{C}}P^{n+1}$. Therefore, by Proposition~\ref{p:quant-rest-c_1}
we have $e_F = r_{\scriptscriptstyle L}(c)$, where
$r_{\scriptscriptstyle L}$ is the quantum restriction map
$QH^*({\mathbb{C}}P^n) \to QH^*(L)$. But it is well known that $c \in
QH^2({\mathbb{C}}P^n)$ is invertible, hence $e_F =
r_{\scriptscriptstyle L}(c) = c*1$ cannot be $0$. ($c*1$ stands for module operation 
where $1$ is the unity of $QH^*(L)$.) It follows that
$e_F$ is the generator of $H^2(L;\mathbb{Z}_2)$.

\subsubsection{The Clifford torus revisited}
\label{sbsb:clif-revisited} We first compute the Floer-Euler class.
Clearly the classical Euler class of $\Gamma_L \to L$ is trivial since
$H^2(L;\mathbb{Z})$ has no torsion. We now use the recipe and notation
from~\S\ref{S:floer-euler-more}. By Section 6.2
of~\cite{Bi-Co:rigidity} (see also~\cite{Cho:Clifford,
  Cho-Oh:Floer-toric}) we have $S_L = [\mathbb{C}P^1] \in
H_2({\mathbb{C}}P^n;\mathbb{Z}_2)$. As $c^{\mathcal{N}}_1 =
\textnormal{PD}[{\mathbb{C}}P^{n-1}]$ we have $\langle
c^{\mathcal{N}}_1, S_L \rangle = 1$, hence by
Proposition~\ref{p:quantum-part-of-e_F}, $e_F = t$. Alternatively, we
could use Proposition~\ref{p:quant-rest-c_1} and the computations
in~\cite{Bi-Co:rigidity, Bi-Co:qrel-long} to calculate $e_F$.

It is interesting to examine what happens to the torus $\Gamma_L$ in
$M = {\mathbb{C}}P^{n+1}$ (rather than in $W = {\mathbb{C}}P^{n+1}
\setminus {\mathbb{C}}P^n$). A simple computation shows that
$\Gamma_L$ becomes now the standard Clifford torus of
${\mathbb{C}}P^{n+1}$.  By Theorem~\ref{T:chain-map-M} the Floer-Euler
class $e'_F$ is now $e'_F = e_F - t = 0$. (We use here the variable
$t$ instead of $q$ since $N_L = 2$ anyway.) It follows from
Theorem~\ref{T:chain-map-M} that the long exact sequence of $\Gamma_L$
in $M = {\mathbb{C}}P^{n+1}$ splits as:
\begin{equation*}
   \begin{CD}
      0 @>>> QH^{k}(L) @>{i}>> QH_{M}^{k}(\Gamma_L) @>{p}>>
      QH^{k-1}(L) @>>> 0.
   \end{CD}
\end{equation*}
It easily follows now that $\Gamma_L \subset {\mathbb{C}}P^{n+1}$ is
wide, i.e. $QH_M^*(\Gamma_L) \cong (H(\Gamma_L;\mathbb{Z}_2) \otimes
\Lambda)^*$.

\subsection{Wide and narrow Lagrangians} \label{sb:wide-narrow} Recall
from~\cite{Bi-Co:rigidity, Bi-Co:lagtop} that a Lagrangian submanifold
$L \subset \Sigma$ is called wide if there exists an isomorphism of
$\Lambda$-modules $QH(L) \cong H(L;\mathbb{Z}_2) \otimes \Lambda$. At
the other extremity we have narrow Lagrangians, i.e. Lagrangians $L$
with $QH(L)=0$. Of course, this notion is very sensitive to the choice
of the ground coefficients ring (in this case $\mathbb{Z}_2$), and
given a ring $K$ one could talk about $K$-wide and $K$-narrow
Lagrangians whenever $QH(L)$ can be defined over the ground ring $K$
(see~\cite{Bi-Co:lagtop} for more on that). Interestingly, when $K$ is
a field all known examples of Lagrangians are either wide or narrow.
This ``wide-narrow'' dichotomy can actually be proved for some
topological classes of Lagrangians such as Lagrangian tori (see e.g.
Theorem~1.2.2 in~\cite{Bi-Co:rigidity}). Below we will examine these
notions in view of the Floer-Gysin long exact sequence.

For simplicity assume that $N_L = 2$. By Theorem~\ref{T:chain-map-1},
if $L$ is narrow then so is $\Gamma_L$.

Assume now that $L$ is wide and that the $\mathbb{Z}_2$-Euler class $e
\in H^2(L;\mathbb{Z}_2)$ of $\Gamma_L \to L$ vanishes. By
Proposition~\ref{p:quantum-part-of-e_F} we have $e = r t$ for some $r
\in \mathbb{Z}_2$. By Theorem~\ref{T:chain-map-1}, if $r=1$ then
$\Gamma_L$ is narrow. Similarly, if $r=0$, then $\Gamma_L$ is wide.
It is interesting to note that if one considers $\Gamma_L$ as a
Lagrangian submanifold of $M$ then things get reversed. Indeed by
Theorem~\ref{T:chain-map-M} if $r=0$ then $\Gamma_L$ is wide in $M$,
while if $r=1$ then $\Gamma_L$ is narrow in $M$.  Note that examples
with $r=0$ are easy to construct: just take $\Sigma \subset M$ with
$c^{\mathcal{N}}_1 \in H^2(\Sigma;\mathbb{Z})$ which is divisible by
$2$ (e.g. $\Sigma=$ quadric in $M={\mathbb{C}}P^{n+1}$).

It would be interesting to study the same issues when $K$ is a general
field (other than $\mathbb{Z}_2$) or even $K = \mathbb{Z}$, assuming
that the Floer-Gysin sequence continues to hold in these cases (of
course, one should add here the assumptions that $L$ is oriented and
endowed with a spin structure. See~\S\ref{s:discussion}). Assume as
before that $K$ is a field, $L$ is $K$-wide and the $K$-Euler class $e
\in H^2(L;K)$ is $0$. Assume further that the class $D_L$ defined
in~\S\ref{S:floer-euler-more} is not $0$ (in particular for generic
$J$ there are holomorphic disks of Maslov index $2$ through a generic
point in $L$). One would expect that if $r \neq 0 \in K$ then
$\Gamma_L$ is narrow and if $r=0$ then $\Gamma_L$ is wide. Note that
by Proposition~\ref{p:quantum-part-of-e_F} one expects that whenever
$K$ has characteristic $0$ we should have $r \neq 0$. In other words,
if $K$ is a field of characteristic $0$ then $\Gamma_L$ should always
be $K$-narrow.

The situation should become more interesting over $K=\mathbb{Z}$. For
example, assume that $L$ is wide with $N_L = 2$ and with $e=0$. In
this case if $r \geq 2$ one would expect $QH(\Gamma_L)$ to have
torsion in the sense that $QH(\Gamma_L)\neq 0$ but $r \cdot a = 0$
for every $a \in QH(\Gamma_L)$.

\section{Discussion anf further questions} \label{s:discussion} Here we
briefly discuss possible extensions of the theory developed in the
paper and pose some questions.

All Floer and quantum cohomologies in this paper were defined over the
ground field $\mathbb{Z}_2$. It is well known that both theories can
be extended to work over any ground ring (e.g. $\mathbb{Z}$) under the
following conditions: the Lagrangians must be oriented and one should
fix a spin structure on them. These choices allow to orient the moduli
spaces of holomorphic disks and pearly trajectories in a coherent way.
Consequently the pearly differential can be defined over $\mathbb{Z}$.
See~\cite{FO3-book-1, FO3-book-2} for orientations of holomorphic
disks and Floer trajectories and~\cite{Bi-Co:lagtop} for pearly
trajectories and the pearl complex.

Considering our situation, assume that $L \subset \Sigma$ is oriented
and endowed with a spin structure $\mathfrak{s}_L$. The orientation of
$L$ induces a natural orientation on $\Gamma_L$ (we orient the fibers
of $\Gamma_L \to L$ with the orientation coming from the fibers of the
complex line bundle $\mathcal{N} \to \Sigma$). Moreover, the spin
structure $\mathfrak{s}_L$ induces a corresponding spin structure
$\mathfrak{s}_{\scriptscriptstyle \Gamma_L}$ on $\Gamma_L$.  With
these structures at hand the pearl complexes of $L$ and $\Gamma_L$ can
be defined over $\mathbb{Z}$. It seems very plausible that most of the
theory (i.e. the Floer-Gysin long exact sequence as well as the
analysis of the Floer-Euler class) continues to hold in this setting
too. In particular the Floer-Euler class $e_F$ will now be related to
the $\mathbb{Z}$ classical Euler class $e \in H^2(L;\mathbb{Z})$ and
moreover, the quantum contribution to $e_F$ whenever it exists will be
in $\mathbb{Z} t$ and might lead to more interesting computations and
stronger consequences. For example, when $W$ is subcritical (or more
generally, when $QH(\Gamma_L)=0$) one would expect that $e_F$ is
invertible over $\mathbb{Z}$ which is a much stronger restriction than
over $\mathbb{Z}_2$ (or even over a field).

In the same context, it would be interesting to study the relations
between the wide varieties of $L$ and $\Gamma_L$ via the techniques of
the paper once they are extended over $\mathbb{Z}$.
(See~\cite{Bi-Co:lagtop} for the definitions of wide varieties.) It
would also be interesting to study the invariants
from~\cite{Bi-Co:lagtop} for $L$ and $\Gamma_L$, e.g. the quadratic
forms and their discriminants, by our techniques.

Another interesting direction is to study the behavior of the
Floer-Gysin sequence with respect to other quantum structures, such as
the quantum module structure and the quantum inclusion. For example,
the quantum cohomology of $L$ is endowed with a structure of a
$QH(\Sigma)$-module and it seems likely that one can lift it to a
natural $QH(\Sigma)$-module structure on $QH(\Gamma_L)$. One would
then expect that the Floer-Gysin becomes compatible with these
$QH(\Sigma)$-module structures in the sense that the maps $i$, $p$
and the connecting homomorphism all become linear over $QH(\Sigma)$.
Note that this is obviously the case for the classical Gysin sequence.

Finally, we expect that much of the theory developed in this paper can
be generalized to Floer homologies of pairs of Lagrangians.  More
precisely, let $L_1, L_2 \subset \Sigma$ be two Lagrangian
submanifolds and let $\Gamma_{L_1}, \Gamma_{L_2} \subset W$ be the
corresponding Lagrangian circle bundles over them. It seems plausible
that similarly to Theorem~\ref{T:exact-seq-1} there should be a long
exact sequence relating $HF(L_1, L_2)$ to $HF(\Gamma_{L_1},
\Gamma_{L_2})$. Of course, one could try to extend this to questions
relating the $A_{\infty}$-algebras (or Fukaya categories) of
Lagrangians in $\Sigma$ and the corresponding ones in $W$.

\bibliography{bibliography}
%\bibliography{/home/biran/latex/general/Bibliography.bib}
%\input{bibliography.tex}
%\input{bibliography.tex}

\end{document}